\documentclass{article}
\usepackage[utf8]{inputenc}
\usepackage{hyperref}
\usepackage{amsmath, amsthm, amssymb, latexsym,epsfig,amsthm,enumerate,multicol,wasysym}
\usepackage{xcolor, comment}
\usepackage{bbm}
\usepackage{mathtools}
\usepackage{enumitem}
\usepackage[pagewise]{lineno}
\usepackage[symbol]{footmisc}

\numberwithin{equation}{section}

\title{Controlling Unknown Linear Dynamics with Almost Optimal Regret}
\author{J. Carruth, M. F. Eggl, C. Fefferman, C. W. Rowley}
\date{\today}

\newcommand{\opt}{\text{opt}}

\newcommand{\E}{\text{E}}
\newcommand{\eE}{\emph{E}}
\newcommand{\R}{\mathbb{R}}
\newcommand{\bg}{\text{big}}

\newcommand{\prob}{\text{Prob}}
\newcommand{\eprob}{\emph{Prob}}

\newcommand{\cost}{\textsc{Cost}}

\newcommand{\mx}{\text{max}}
\newcommand{\emx}{\emph{max}}

\newcommand{\las}{\text{L}a\text{S}}
\newcommand{\lqs}{\text{L}q\text{S}}

\newcommand{\elas}{\emph{L}a\emph{S}}
\newcommand{\elqs}{\emph{L}q\emph{S}}

\newcommand{\cJ}{\mathcal{J}}
\newcommand{\cD}{\mathcal{D}}
\newcommand{\cF}{\mathcal{F}}
\newcommand{\var}{\text{Var}}
\newcommand{\pro}{\textsc{Pro}}
\newcommand{\sml}{\text{small}}
\newcommand{\esml}{\emph{small}}
\newcommand{\tny}{\text{tiny}}
\newcommand{\etny}{\emph{tiny}}
\newcommand{\mr}{\text{MR}}

\newcommand{\hr}{\text{HR}}
\newcommand{\ehr}{\emph{HR}}

\newcommand{\rare}{\text{rare}}
\newcommand{\br}{\text{BR}}

\newcommand{\cg}{\text{CG}}

\newcommand{\osc}{\text{osc}}
\newcommand{\regret}{\textsc{Regret}}
\newcommand{\bayes}{\text{Bayes}}
\newcommand{\ebayes}{\emph{Bayes}}
\newcommand{\Prologue}{\textsc{Prologue}}
\newcommand{\Main}{\textsc{Main Act}}

\newcommand\extrafootertext[1]{%
    \bgroup
    \renewcommand\thefootnote{\fnsymbol{footnote}}%
    \renewcommand\thempfootnote{\fnsymbol{mpfootnote}}%
    \footnotetext[0]{#1}%
    \egroup
}

\newtheorem{thm}{Theorem}[section]
\newtheorem{lem}[thm]{Lemma}
\newtheorem{rmk}{Remark}[section]
\newtheorem{claim}{Claim}
\newtheorem{cor}{Corollary}[section]

\begin{document}
\maketitle

\begin{abstract}
    Continuing from the companion paper \cite{boundeda}, we consider a simple control problem in which the underlying dynamics depend on a parameter $a$ that is unknown and must be learned. In this paper, we assume that $a$ can be any real number and we do not assume that we have a prior belief about $a$. We seek a control strategy that minimizes a quantity called the regret. Given any $\varepsilon>0$, we produce a strategy that minimizes the regret to within a multiplicative factor of $(1+\varepsilon)$.
\end{abstract}

\section{Introduction}\label{sec: intro}
Continuing from \cite{carruth2022controlling,fefferman2021optimal}, we explore a new flavor of adaptive control theory, which we call ``agnostic control.'' \extrafootertext{This work was supported by AFOSR grant FA9550-19-1-0005 and by the Joachim Herz Foundation.}

Many works in adaptive control theory attempt to control a system whose underlying dynamics are initially unknown and must be learned from observation. The goal is then to bound $\regret$, a quantity defined by comparing our expected cost with that incurred by an opponent who knows the underlying dynamics. Typically one tries to achieve a regret whose order of magnitude is as small as possible after a long time. Adaptive control theory has extensive practical applications; see, e.g., \cite{bertsekas2012dynamic, cesa2006prediction, hazancontrol, powell2007approximate}.

In some applications, we don't have the luxury of waiting for a long time. This is the case, e.g., for a pilot attempting to land an airplane following the sudden loss of a wing, as in \cite{Brazy:2009}. Our goal, here and in \cite{boundeda}, is to achieve the absolute minimum possible regret over a fixed, finite time horizon. This poses formidable mathematical challenges, even for simple model systems.

We will study a one-dimensional, linear model system whose dynamics depend on a single unknown parameter $a$. When $a$ is large positive, the system is highly unstable. (There is no ``stabilizing gain'' for all $a$.) We suppose that the unknown $a$ may be any real number and we don't assume that we are given a Bayesian prior probability distribution for it.

Modulo an arbitrarily small increase in regret, we reduce the problem to a Bayesian variant in which the unknown $a$ is confined to a finite set and governed by a prior probability distribution.

For the Bayesian problem, our task is to find a strategy that minimizes the expected cost. This leads naturally to a PDE (a Bellman equation).

In \cite{boundeda}, we prove that the optimal strategy for Bayesian control is indeed given in terms of the solution of the Bellman equation, and that any strategy significantly different from that optimum incurs a significantly higher cost. We proceed modulo assumptions about existence and regularity of the relevant PDE solutions, for which we lack rigorous proofs. (However, we have obtained numerical solutions, which seem to behave as expected.)

Let us now explain the above in more detail. 

\subsection{The Model System}

Our system consists of a particle moving in one dimension, influenced by our control and buffeted by noise. The position of our particle at time $t$ is denoted by $q(t) \in \R$. At each time $t$, we may specify a ``control'' $u(t) \in \R$, determined by history up to time $t$, i.e., by $(q(s))_{s \in [0,t]}$. A ``strategy'' (aka ``policy'') is a rule for specifying $u(t)$ in terms of $(q(s))_{s \in [0,t]}$ for each $t$. We write $\sigma, \sigma', \sigma^*, \text{etc.}$ to denote strategies. The noise is provided by a standard Brownian motion $(W(t))_{t\ge 0}$.

The particle moves according to the stochastic ODE
\begin{equation}\label{eq: intro 1}
dq(t) = \big(aq(t) + u(t)\big)dt + dW(t), \qquad q(0) = q_0,
\end{equation}
where $a$ and $q_0$ are real parameters. Due to the noise in \eqref{eq: intro 1}, $q(t)$ and $u(t)$ are random variables.

Over a time horizon $T>0$, we incur a $\cost$, given\footnote[2]{By rescaling, we can consider seemingly different cost functions of the form $\int_0^T(q^2+\lambda u^2)$ for $\lambda >0$.} by
\begin{equation}\label{eq: intro 2}
    \cost = \int_0^T \big\{ (q(t))^2 + (u(t))^2\big\} dt.
\end{equation}
This quantity is a random variable determined by $a, q_0,T$ and our strategy $\sigma$. Here, the starting position $q_0$ and time horizon $T$ are fixed and known. 

We would like to keep our cost as low as possible. We examine several variants of the above control problem, making successively weaker assumptions regarding our knowledge of the parameter $a$. Those variants are as follows.

\subsection*{Variant I: Classical Control}
We suppose first that the parameter $a$ is known. We write $\cJ(\sigma, a; T,q_0)$ to denote the expected $\cost$ incurred by executing a given strategy $\sigma$. Our task is to pick $\sigma$ to minimize $\cJ(\sigma, a; T,q_0)$. As shown in textbooks (e.g., \cite{astrom}), the optimal strategy $\sigma$ is given by
\[
u(t) =- \kappa(T-t,a)q(t)
\]
for a known elementary function $\kappa$; see also Section \ref{sec: setup} below. We denote this strategy by $\sigma_\opt(a)$. It will be important later to note that $\sigma_\opt(a)$ satisfies the inequality
\begin{equation}\label{eq: intro 3}
    |u(t)|\le C \max\{a,1\}\cdot |q(t)|\;\text{for an absolute constant } C.
\end{equation}

\subsection*{Variant II: Bayesian Control}
Next, suppose that the parameter $a$ is unknown, but is subject to a given prior probability distribution $d\prob(a)$ supported in an interval $[-A,A]$. Our goal is now to pick a strategy $\sigma$ to minimize our expected cost, given by
\begin{equation}\label{eq: intro 4}
    \int_{-A}^A \cJ(\sigma,a;T,q_0)\  d\prob(a).
\end{equation}

To solve this problem, we first note a major simplification: In principle, a strategy $\sigma$ is a one-parameter family of functions on an infinite-dimensional space, because it specifies $u(t)$ in terms of the path $(q(s))_{s \in [0,t]}$ for each $t$. However, one computes that the posterior probability distribution for the unknown~$a$, given past history $(q(s))_{s \in [0,t]}$, is determined by the prior $d\prob(a)$, together with the two observable quantities
\begin{equation}\label{eq: intro 5}
   \zeta_1(t) = \int_{0}^t q(s)\ dq(s) - \int_0^t q(s) u(s) \ ds,\qquad \zeta_2(t) = \int_0^t (q(s))^2 \ ds.
\end{equation}
Therefore, it is natural to suppose that our optimal strategy $\sigma_{\bayes}(d\prob)$ takes the form
\begin{equation}\label{eq: intro 6}
u(t) = \tilde{u}(q(t),\zeta_1(t),\zeta_2(t),t)
\end{equation}
for a function $\tilde{u}$ on $\R^4$.

So, instead of looking for a one-parameter family of functions on an infinite-dimensional space, we merely have to specify a function $\tilde{u}$ of four variables. It isn't hard to derive a PDE (the Bellman equation) for the function $\tilde{u}(q,\zeta_1,\zeta_2,t)$ that, according to heuristic reasoning, minimizes our expected cost. We have produced approximate solutions $\tilde{u}$ to the Bellman equation in numerical simulations, but we don't have rigorous proofs of existence or regularity. We proceed by imposing the

\underline{\textsc{PDE Assumption}}: The Bellman equation has a smooth solution, and the resulting control strategy satisfies the estimate
\begin{equation}\label{eq: intro 6.5}
  |u(t)| \le C_0 A^{m_0}[|q(t)|+1],  
\end{equation}
for constants $C_0, m_0$ independent of $A$.

Since our prior distribution $d\prob(a)$ is supported in $[-A,A]$, a glance at \eqref{eq: intro 3} suggests that the optimal Bayesian strategy should satisfy
\[
|u(t)| \le C \max\{A,1\}\cdot|q(t)|.
\]
Our numerical simulations appear to confirm this belief. Accordingly, \eqref{eq: intro 6.5} seems to be a very safe assumption.

Under the above PDE Assumption, we prove in \cite{boundeda} the natural result that the strategy $\sigma_{\bayes}(d\prob)$ arising from the Bellman equation indeed minimizes the expected cost given by \eqref{eq: intro 4}. Moreover, any strategy that differs significantly from $\sigma_{\bayes}(d\prob)$ incurs a significantly higher expected cost.

Our results on Bayesian control pave the way for our analysis of agnostic control. 

\subsection*{Variant III: Agnostic Control for Bounded $a$}
We suppose that our parameter $a$ is confined to a bounded interval $[-A,A]$ but is otherwise unknown. In particular, we don't assume that we are given a Bayesian prior probability distribution $d\prob(a)$. Consequently, we cannot define a notion of expected cost by formula \eqref{eq: intro 4}.

Instead, our goal will be to minimize \emph{worst-case regret}, defined by comparing the performance of our strategy with that of the optimal known-$a$ strategy $\sigma_\opt(a)$. We will introduce several variants of the notion of regret.

Let us fix a starting position $q_0$, a time horizon $T$, and an interval $[-A,A]$ guaranteed to contain the unknown $a$. To a given strategy $\sigma$, we associate the following functions on $[-A,A]$:
\begin{itemize}
    \item \emph{Additive Regret}, defined as
    \[
    \text{AR}(\sigma,a) = \cJ(\sigma,a;T,q_0) - \cJ(\sigma_\opt(a),a;T,q_0) \ge 0.
    \]
    \item \emph{Multiplicative Regret} (aka ``competitive ratio''), defined as
    \[
    \mr(\sigma,a) = \frac{\cJ(\sigma,a;T,q_0)}{\cJ(\sigma_\opt(a),a;T,q_0)} \ge 1.
    \]
    \item \emph{Hybrid Regret}, defined in terms of a parameter $\gamma>0$ by setting 
    \[
    \hr_\gamma(\sigma,a) = \frac{\cJ(\sigma,a;T,q_0)}{\cJ(\sigma_\opt(a),a;T,q_0)+\gamma}
    \]
\end{itemize}
Writing $\regret(\sigma,a)$ to denote any one of the above three functions on $[-A,A]$, we define the \emph{worst-case regret}
\begin{equation}\label{eq: intro wcr}
\textsc{Regret}^*(\sigma) = \sup\{\regret(\sigma,a): a \in [-A,A]\}.
\end{equation}
We seek a strategy $\sigma$ that minimizes worst-case regret.

The above notions are useful in different regimes. If we expect to pay a large cost, then we care more about multiplicative regret than about additive regret. (If we have to pay $10^9$ dollars, we are unimpressed by a saving of $10^5$ dollars.) Similarly, if our expected cost is small, then we care more about additive regret than about multiplicative regret. (If we pay only $10^{-5}$ dollars, we don't care that we might instead pay $10^{-9}$ dollars.) If we fix $\gamma$ to be a cost we are willing to neglect, then hybrid regret $\hr_\gamma(\sigma,a)$ provides meaningful information regardless of the order of magnitude of our expected cost.

So far, we have defined three flavors of worst-case regret, and posed the problem of minimizing that regret. The solution to our agnostic control problem is given by the following result, proved in \cite{boundeda}.

\begin{thm}\label{thm: intro 1}
    Fix $[-A,A]$, $q_0$, $T$ (and $\gamma$ if we use hybrid regret). Assume the \emph{PDE Assumption}. Then
    \begin{itemize}
        \item[\emph{(I)}] There exists a probability measure $d\eprob_*$ on $[-A,A]$ for which the optimal Bayesian strategy $\sigma_{\ebayes}(d\eprob_*)$ minimizes the worst-case regret among all possible strategies. \\Moreover,
        \item[\emph{(II)}] The measure $d\eprob_*$ is supported on a finite set $E \subset [-A,A]$, where
        \item[\emph{(III)}] $E$ is precisely the set of points $a \in [-A,A]$ at which the function $[-A,A] \ni a \mapsto \regret(\sigma_{\ebayes}(d\eprob_*),a)$ achieves its maximum.
    \end{itemize}
\end{thm}

So, for optimal agnostic control, we should pretend to believe that the unknown $a$ is confined to a finite set $E$ and governed by the probability distribution $d\prob_*$, even though in fact we know nothing about $a$ except that it lies in $[-A,A]$.

It is easy to see that a probability distribution $d\prob_*$ satisfying (II) and (III) gives rise to a Bayesian  strategy satisfying (I). The hard part of Theorem \ref{thm: intro 1} is the assertion that such a probability measure exists. 

Theorem \ref{thm: intro 1} lets us search for optimal agnostic strategies: We first guess a finite set $E$ and a probability measure $d\prob$ concentrated on $E$. By solving the Bellman equation, we produce the strategy $\sigma= \sigma_{\bayes}(d\prob)$, which allows us to compute the function $[-A,A] \ni a \mapsto \regret(\sigma,a)$. If the maxima of that function occur precisely at the points of $E$, then $\sigma$ is the desired optimal agnostic strategy. We have carried this out numerically for several $[-A,A], q_0, T$.

This concludes our discussion of agnostic control for bounded $a$. Finally, we pass to the most general case.

\subsection*{Variant IV: Fully Agnostic Control}

We make no assumption whatever regarding the unknown $a$; our $a$ may be any real number, and we are not given a Bayesian prior distribution for it.

Our goal is again to find a strategy that minimizes the \emph{worst-case regret}, defined as in \eqref{eq: intro wcr}, except that the $\sup$ is now taken over all $a \in \R$.

The main result of this paper is that, with negligible increase in regret, we can reduce matters to agnostic control for bounded $a$. More precisely, we will prove the following result.

\begin{thm}\label{thm: intro 2}
    Fix a time horizon $T$ and a nonzero starting position $q_0$, as well as constants $C_0, m_0$  (for which, see \eqref{eq: intro 6.5}). Then given $\varepsilon>0$ there exists $A>0$ for which the following holds.
    
    Let $\sigma$ be a strategy for the starting position $q_0$ and time horizon $T+\varepsilon$. Suppose $\sigma$ satisfies estimate \eqref{eq: intro 6.5} for the given $C_0, m_0,A$. 
    
    Then there exists a strategy $\sigma_*$ for the starting position $q_0$ and time horizon $T$, satisfying the following estimates.
    \begin{itemize}
        \item[\emph{(A)}] For $a \in [-A,A]$ we have
        \[
        \cJ(\sigma_*,a;T,q_0) \le \varepsilon + (1+\varepsilon)\sup\{\cJ(\sigma,a';T+\varepsilon,q_0) : |a'- a| \le \varepsilon |a|\}.
        \]
        \item[\emph{(B)}] For $a \notin [-A,A]$ we have
        \[
        \cJ(\sigma_*,a;T,q_0) \le \varepsilon + (1+\varepsilon) \cJ(\sigma_{\emph{opt}}(a),a;T,q_0).
        \]
    \end{itemize}
\end{thm}

So, if $a \in [-A,A]$, then $\sigma_*$ performs almost as well as $\sigma$; and if $a \notin [-A,A]$, then $\sigma_*$ performs almost as well as the optimal known-$a$ strategy $\sigma_\opt(a)$.

To apply Theorem \ref{thm: intro 2}, we take $\sigma$ to be an optimal agnostic strategy for the starting position $q_0$ and the time horizon $T+\varepsilon$, assuming $a$ to be confined to the interval $[-(1+\varepsilon)A, +(1+\varepsilon)A]$.

If our \emph{PDE Assumption} holds, then $\sigma$ satisfies the hypothesis of Theorem \ref{thm: intro 2}. It's easy to deduce from Theorem \ref{thm: intro 2} that the worst-case hybrid regret of the strategy $\sigma_*$ (for fully agnostic control) is at most $O(\varepsilon)$ percent worse than that of $\sigma$ (for agnostic control with $a$ confined to $[-(1+\varepsilon)A, (1+\varepsilon) A]$).

For any $\gamma>0$ and any strategy $\sigma$ for time horizon $T$ and starting position $q_0$ we let $\hr_\gamma^*(\sigma;T)$ denote the worst-case hybrid regret (over all $a \in \R$) of $\sigma$. The above discussion then implies the following.

\begin{cor}\label{cor: hr}
    Fix constants $T>0$, $\eta>0$, $q_0\ne 0$. Assume the \emph{PDE Assumption}. Then for any $\varepsilon>0$ we can construct a strategy $\sigma_{\emph{Ag}}$ for time horizon $T$ and starting position $q_0$ satisfying
    \[
    \ehr_\gamma^*(\sigma_{\emph{Ag}};T) \le (1+C\varepsilon)\cdot \ehr_\gamma^*(\sigma;T+\varepsilon)
    \]
    for any strategy $\sigma$ for time horizon $T+\varepsilon$ and starting position $q_0$.
\end{cor}

There are variants of Theorem \ref{thm: intro 2} and Corollary \ref{cor: hr} for the case of starting position $q_0=0$; see Corollary \ref{cor: q0}.

\subsection*{Recap}

Let us summarize what has been achieved.

Our goal is to minimize worst-case regret in the setting where $a$ may be any real number. Modulo an arbitrarily small increase in regret, we may reduce matters to the case in which $a$ is confined to a bounded interval $[-A,A]$. We then look for a probability measure $d\prob_*$ living on a finite subset $E \subset [-A,A]$ such that the regret of the optimal Bayesian strategy for $d\prob_*$ is maximized precisely on $E$. We can calculate the optimal Bayesian strategy for a given prior probability measure by solving a Bellman equation. However, our results are conditional; we have to make an assumption on the existence, size, and smoothness of solutions to the Bellman equation. In numerical simulations, we have produced evidence for our \emph{PDE Assumption}, and we have produced optimal agnostic strategies for cases in which the unknown $a$ is confined to a bounded interval.

\subsection{Ideas from the Proof of Theorem \ref{thm: intro 2}}
Recall that in Theorem \ref{thm: intro 2} we are given a strategy $\sigma$, a small positive $\varepsilon$, and a large positive $A$ depending on $\varepsilon$. The strategy $\sigma$ applies to a starting position $q_0 \ne 0$, and a time horizon $T + \varepsilon$. Our task is to find a strategy $\sigma_*$ satisfying conditions (A) and (B).

In this introduction, we allow ourselves some inaccuracy in the interest of simplicity. See Sections \ref{sec: almost optimal strategy}--\ref{sec: las} for correct details.

A recurring theme in the proof of Theorem \ref{thm: intro 2} is that because the unknown $a$ may be arbitrarily large and positive, the system may be arbitrarily unstable. Consequently, disasters of exponentially small probability may lead to exponentially large expected cost.

To prepare the way for the proof of Theorem \ref{thm: intro 2}, we first examine two regimes in which we can do almost as well as if we knew the value of $a$.

\subsection*{The Large $q$ Regime}

Suppose $q_0$ is very large. A glance at \eqref{eq: intro 1} suggests that the noise $dW(t)$ has only a small effect compared to that of the term $(aq+u)dt$. Therefore, after initially setting $u \equiv 0$ and observing $q(t)$ for small $t$, we quickly arrive at a guess $\bar{a}$ for the unknown $a$. That guess is probably highly accurate. Moreover, the larger $a$ is, the sooner we can arrive at the guess $\bar{a}$.

Once we have found $\bar{a}$, we can simply play the known-a strategy $\sigma_\opt(\bar{a})$ until the end of the game at time $T$. (If $\bar{a}$ is large positive, then in place of $\sigma_\opt(\bar{a})$, we use the strategy in which $u(t) = - 2 \bar{a}q(t)$, which is equivalent to $\sigma_\opt(\bar{a})$ asymptotically for $\bar{a} \gg 1$.)

This ``Large $q$ strategy'' incurs an expected cost almost as small as that of an opponent who knows $a$. 

A crucial point is that for large positive $a$, the probability of a significantly inaccurate guess $\bar{a}$ is $O(\exp(-cq_0^2 a))$, while the expected cost if such an error occurs is $O(q_0^2\exp(CTa))$. Hence, for large $q_0$, the exponentially tiny probability of disaster overwhelms the exponentially large resulting cost.

\subsection*{The Large $a$ Regime}
As in our discussion of the Large $q$ strategy, a glance at \eqref{eq: intro 1} suggests that the nosie $dW(t)$ will have negligibly small effect compared to that of the term $(aq+u)dt$, provided $a$ is large positive, say, $a \ge A$. This leads to a ``na\"{i}ve large $a$ strategy'' in which we initially set $u \equiv 0$, observe $q(t)$ for small $t$, arrive at a guess $\bar{a}$ for the unknown $a$, and then play the known-a strategy $\sigma_\opt(\bar{a})$ until the game ends at time $T$.

This time, however, the exponentially small probability of an error of the form $\bar{a} \ll a$ is dominated by the exponentially large cost of the ensuing disaster. Consequently, the na\"{i}ve large $a$ strategy fails. The cure is to pick some $q_0^*$, big but not too big, and execute the na\"{i}ve large $a$ strategy only until the first moment we encounter $|q(t)| = q_0^*$. At that moment, we switch over to the Large $q$ strategy. If we never encounter $|q(t)| = q_0^*$, then we continue with the na\"{i}ve large $a$ strategy until the end of the game.

The above modification limits the damage arising from the event $\bar{a} \ll a$. We have thus produced a ``Large $a$ strategy'' whose expected cost is close to that of the optimal known-a strategy $\sigma_\opt(a)$ whenever $a \ge A$. When $a < A$, the expected cost of our Large $a$ strategy is $O(A^2)$. That's bigger than we'd like, but it isn't exponentially large.

\subsection*{Strengthening the Given Strategy}

Next, we consider the strategy $\sigma$ given in the statement of Theorem \ref{thm: intro 2}. Recall that $\sigma$ is assumed to satisfy the estimate \eqref{eq: intro 6.5}. From that estimate we see at once that $\sigma$ disastrously undercontrols in case $a \gg A^{m_0}$, leading to exponentially large expected cost. We can remedy this defect, by modifying $\sigma$ the same way we modified the na\"{i}ve large $a$ strategy. We pick a $q_0^*$, large but not too large, and switch over from $\sigma$ to the Large $q$ strategy as soon as we encounter $|q(t)| = q_0^*$. Thus, we obtain a strategy $\tilde{\sigma}$ that performs almost as well as $\sigma$ for $ a\in [-A,A]$, and avoids exponentially large expected cost if $a > A$.

Like $\sigma$, the strategy $\tilde{\sigma}$ starts at position $q_0$. By rescaling $\tilde{\sigma}$ slightly, we arrive at a strategy $\bar{\sigma}$ with starting position $(1+\varepsilon)q_0$. Like $\tilde{\sigma}$, the strategy $\bar{\sigma}$ performs almost as well as $\sigma$ when $a \in [-A,A]$, and avoids exponentially large expected cost when $a > A$.

Armed with the Large $q$ and Large $a$ strategies, and the modified strategy $\bar{\sigma}$, we can now describe the strategy $\sigma_*$ whose existence is asserted by Theorem \ref{thm: intro 2}. Without loss of generality, we suppose that our nonzero starting position $q_0$ is positive.

In the strategy $\sigma_*$ there are two epochs, a Prologue and a Main Act. During the Prologue, we set $u \equiv 0$. The Prologue ends as soon as we encounter one of the three events
\begin{enumerate}
    \item[(a)] $q(t) = (1+\varepsilon)q_0$
    \item[(b)] $q(t) = - q_{\rare}$ for a suitable $q_\rare > 0$ (big but not too big)
    \item[(c)] the end of the game at time $T$ (in which case there is no Main Act).
\end{enumerate}

Event (b) occurs with small probability, regardless of the unknown $a$. If it does occur, then during the Main Act we play a slight variant of the Large $q$ strategy to bound our losses.

If instead we enter the Main Act via case (a), then by observing how long it took to pass from the initial position $q_0$ to the position $(1+\varepsilon)q_0$, we obtain a guess $\bar{a}$ for the unknown $a$. If $a$ is large positive, then as in our discussion of the Large $a$ strategy, our guess $\bar{a}$ is probably highly accurate. Otherwise, $\bar{a}$ is likely not so close to $a$, but at least $\bar{a}$ probably won't be large positive. Therefore, for large $A$, our guess $\bar{a}$ will at least tell us whether $a \gtrsim A$ or $a \ll A$. Accordingly, in case (a) we proceed as follows during the Main Act.
\begin{itemize}
    \item If $\bar{a} > A$, then during the Main Act we execute the Large $a$ strategy.
    \item If $\bar{a} \le A$, then during the Main Act we execute $\bar{\sigma}$, the improved version of the given strategy $\sigma$. 
\end{itemize}
The Main Act lasts until the end of the game at time $T$. 

This completes our description of the strategy $\sigma_*$. We hope the reader finds it plausible that our $\sigma_*$ satisfies conditions (A) and (B) in Theorem \ref{thm: intro 2}.

We again warn the reader that our discussion in this introduction is somewhat oversimplified. For instance, our basic stochastic ODE \eqref{eq: intro 1} isn't obviously well-defined if our strategy allows $u(t)$ to be a discontinuous function of $t$. We haven't even given a rigorous definition of a strategy. Our rigorous discussion starts from scratch in Section \ref{sec: setup}.

\subsection*{Survey of Prior Literature}

Literature that considers adaptive control of a simple linear system similar to the one considered in this paper commonly consists of one or more of the following features: (\emph{i}) unknown governing dynamics, (\emph{ii}) unknown cost function and (\emph{iii}) adversarial noise.  Examples of such work include \cite{kumar2022online,mania2019certainty,wagenmaker2020active,furieri2020learning,Cohen:2019,malik2019derivative,duchi2011adaptive} as well as our own prior work \cite{fefferman2021optimal}, \cite{carruth2022controlling}.

Initial work in obtaining regret bounds in the infinite time horizon for the related LQR (linear-quadratic regulator) problem was undertaken in \cite{abbasi2011regret}, which proved that under certain assumptions, the expected additive regret of the adaptive controller is bounded by $\tilde{O} (\sqrt{T})$. Further progress was made on this problem in \cite{chen2021black}. Assuming controllability of the system, the authors gave the first efficient algorithm capable of attaining sublinear additive regret in a single trajectory in the setting of online nonstochastic control. See also the related \cite{minasyan2021online}, which obtained sublinear adaptive regret bounds, a stronger metric than standard regret and more suitable for time-varying systems. Additional adaptive control approaches include \cite{dean2018regret,dean2019safely} using the system level synthesis. This expands on ideas in \cite{simchowitz2018learning}, which showed that the ordinary least-squares estimator learns a linear system nearly optimally in one shot. Other work uses Thompson Sampling \cite{abeille2017thompson,kargin2022thompson} or deep learning \cite{chen2023regret}.  Perhaps most related to the work performed in this study is \cite{jedra2022minimal}, which designed an online learning algorithm with sublinear expected regret that moves away from episodic estimates of the state dynamics (meaning that no boundedness or initially stabilizing control needed to be assumed).

In \cite{fefferman2021optimal}, the third and fourth authors of the present paper, along with B. Guill\'{e}n Pegueroles and M. Weber, found regret minimizing strategies for a problem with simple unknown dynamics (a particle moving in one-dimension at a constant, unknown velocity subject to Brownian motion). In \cite{gurevich2022optimal}, along with D. Goswami and D. Gurevich, they generalized these results to an analogous, higher-dimensional system with the addition of sensor noise. In \cite{fefferman2021optimal}, they also posed the problem of finding regret minimizing strategies for the more complicated dynamics \eqref{eq: intro 1}. In \cite{carruth2022controlling}, the authors of the present paper, along with M. Weber, took the first steps toward resolving this problem. Specifically, we exhibited a strategy for the dynamics \eqref{eq: intro 1} with bounded multiplicative regret.

Historically, significant work has been undertaken in the closely related ``multi-armed bandit'' problem; see, for instance, the classic papers \cite{robbins1952some,vermorel2005multi}. Recent work considering this paradigm includes \cite{wei2021non}, which used reinforcement learning to obtain dynamic regret whose order of magnitude is optimal, and \cite{faury2021regret}, which studied the more general Generalized Linear Bandits (GLBs) and obtains similar regret bounds.

We finally want to point out the parallel field of adversarial control, where the noise profile is chosen by an adversary instead of randomly. This includes \cite{martin2022safe}, which attained minimum dynamic regret and guaranteed compliance with hard safety constraints in the face of uncertain disturbance realizations using the system level synthesis framework, and \cite{goel2021competitive}, which studied the problem of competitive control.

As this list of references is by no means exhaustive and does not do justice to the wealth of studies in the literature, we point the reader to the book \cite{hazancontrol} and the references therein for a more thorough overview of online control.

We emphasize that our approach in \cite{fefferman2021optimal}, \cite{carruth2022controlling}, and the present paper differs from the other work cited above in that:
\begin{itemize}
    \item We seek strategies that minimize the worst-case regret for a fixed time horizon $T$, whereas the literature is mainly concerned with $T \rightarrow \infty$.
    \item \sloppy Typically in the literature one assumes either that the dynamics are bounded or that one is given a stabilizing control. We make no such assumptions and so we must control a system that is arbitrarily unstable.
    \item However, we achieve the above ambitious goals only for a simple model system.
\end{itemize}

We thank Amir Ali Ahmadi, Brittany Hamfeldt, Elad Hazan, Robert Kohn, Sam Otto, Allan Sly, and Melanie Weber for helpful conversations. We are grateful to the Air Force Office of Scientific Research, and Frederick Leve, for their generous support. The second named author thanks the Joachim Herz foundation for support.

\section{Setup}\label{sec: setup}

Except for where we explicitly state otherwise, we adopt the following conventions regarding constants throughout this paper. We write $C, C', C'', \dots$ to denote positive absolute constants. When we wish to specify that a positive absolute constant is smaller than 1, we write $c, c', c'', \dots$. If a constant depends on some quantity $X$, then we write $C_X, C'_X, c_X, c'_X, \dots$. The values of these constants may change from line to line. 

Whenever we say that a parameter/constant depends on other parameters/constants, we assume that the dependence is continuous unless we explicitly state otherwise.

We let $W(t)$ denote Brownian motion starting at $W(0) = 0$ and normalized so that $\E[(W(t))^2] = t$. We write $(\Omega, \cF, \prob)$ for the corresponding probability space and $\E[X]$ for the expected value of a random variable $X$. We write $\omega$ to denote an arbitrary element of $\Omega$. For $t \in [0,T]$, we write $\cF_t$ to denote the sigma algebra determined by the history of the Brownian motion from time 0 until time $t$.

We introduce a \emph{time horizon} $T>0$ and a \emph{starting position} $q_0 \ne 0$. We define a \emph{strategy} (for time horizon $T$ and starting position $q_0$) to be a collection of random variables $q(t,a)$, $u(t,a)$ defined for all $t \in [0,T]$ and $a\in \R$ and satisfying the following:
\begin{enumerate}[label=(S.\arabic*)]
\item For every $a \in \R$, $q(t,a)$ is a continuous function of $t$ with probability 1 and $u(t,a)$ is an $L^2$ function of $t$ with probability 1.
\item For every $a \in \R$ and $t \in [0,T]$ the maps $(s,\omega) \mapsto q(s,a,\omega)$ and $(s,\omega) \mapsto u(s,a,\omega)$, defined on $[0,t]\times \Omega$, are measurable as functions on $[0,t]\times (\Omega, \cF_t, \prob)$. Intuitively, this means that $q,u$ are determined by the past. \label{s4}
 \item For every $a \in \R$,
    \[
    \E\bigg[ \int_0^T ( q(t,a))^2 + (u(t,a))^2\ dt \bigg] < \infty.
    \]
    \item For almost all $\omega \in \Omega$, we have that for all $a,b\in \R$, and for all $t \in [0,T]$, if $q(s,a,\omega) = q(s,b,\omega)$ for all $s \in [0,t]$ then $u(s,a,\omega) = u(s,b,\omega)$ for almost all $s \in [0,t]$. This tells us that $u$ does not depend on the unknown $a$.
    \item For every $a \in \R$ and $t \in [0,T]$, we have
    \[
    q(t,a) = q_0 + W(t) + \int_0^t [aq(\tau,a) + u(\tau,a)]\ d \tau
    \]
    with probability 1.\label{s2}
\end{enumerate}

For a given $a \in \R$ we refer to $q(t,a)$ and $u(t,a)$, respectively, as the \emph{particle trajectory} and the \emph{control variable} for $a$ at time $t$.

We remark that any strategy for time horizon $T$ and starting position $q_0$ gives rise to a strategy for time horizon $T'$ and starting position $q_0$ for any $T' \in (0, T)$ (simply by restricting the time domain).

We will use $\sigma$ to denote an arbitrary strategy. We then write $q^\sigma$ and $u^\sigma$ to denote the families of particle trajectories and control variables associated with $\sigma$.

For any strategy $\sigma$ we define a random variable
\[
\cost(\sigma,a) = \int_0^T\big( (q^\sigma (t,a))^2 + (u^\sigma(t,a))^2\big) \ dt.
\]
This random variable is well-defined because with probability 1 $q(t,a)$ is a continuous function of $t$ and $u(t,a)$ is an $L^2$ function of $t$. We then define the \emph{expected cost} of $\sigma$ by
\[
\cJ(\sigma, a; T, q_0) = \E [ \cost(\sigma,a)].
\]

For any smooth function $v:[0,T] \rightarrow \R$ we define a strategy $\sigma_v$ by setting
\begin{align}
&q^{\sigma_v}(t,a) = q_0 + W(t) + \int_0^t(a - v(\tau)) q^{\sigma_v}(\tau,a)\ d \tau,\label{eq: sfs 1}\\
&u^{\sigma_v}(t,a) = - v(t) q^{\sigma_v}(t,a)\label{eq: sfs 2}.
\end{align}
We refer to $\sigma_v$ as a \emph{simple feedback strategy} with \emph{gain function} $v$.

For $\alpha \in \R$, $s \ge 0$ we define
\[
\kappa(s,\alpha) = \frac{\tanh(s\sqrt{\alpha^2+1})}{\sqrt{\alpha^2+1}-\alpha \tanh(s\sqrt{\alpha^2+1})}.
\]
We let $\sigma_\opt(\alpha)$ denote the simple feedback strategy with gain function $t \mapsto \kappa(T-t, \alpha)$. In Section \ref{sec: known a}, we will show that for any $a \in \R$, the strategy $\sigma_\opt(a)$ minimizes the quantity $\cJ(\sigma,a;T,q_0)$ over all strategies $\sigma$. We therefore refer to the family of strategies $(\sigma_\opt(\alpha))_{\alpha \in \R}$ as \emph{optimal known-$a$} (or just known-$a$) strategies. For ease of notation we define
\[
\cJ_0(a;T,q_0) = \cJ(\sigma_\opt(a),a;T,q_0);
\]
we refer to $\cJ_0(a;T,q_0)$ as the \emph{optimal expected cost for known $a$} (for time horizon $T$ and starting position $q_0$).

Fix a real number $C_0$ and an integer $m_0 \ge 1$. We say that a strategy $\sigma$ (for time horizon $\hat{T}$) is \emph{$A$-bounded} for some $A>0$ if
\[
|u^\sigma(t,a)| \le C_0 A^{m_0} [ |q^\sigma(t,a)|+1]\;\text{for all}\; a\in \R, t \in [0,\hat{T}].
\]
Recall that we assume that our starting position $q_0$ is nonzero.

\begin{thm}\label{thm: main simp}
    Let $\varepsilon>0$. Then for $A>0$ sufficiently large depending on $\varepsilon, T, q_0, C_0, m_0$, the following holds.

Let $\sigma$ be an $A$-bounded strategy for time horizon $T + \varepsilon$ and starting position $q_0$. Then the strategy $\sigma_*$ for time horizon $T$ and starting position $q_0$ specified in Section \ref{sec: def aos} satisfies the following.
\begin{enumerate}
    \item If $a \in [-A,A]$, then
    \[
    \cJ(\sigma_*,a;T,q_0) < \varepsilon + (1+\varepsilon) \cdot \sup\{ \cJ(\sigma,b;T+\varepsilon,q_0): |a-b| <  \varepsilon |a|\}.
    \]
    \item If $|a| > A$, then
    \[
    \cJ(\sigma_*,a;T,q_0) < \varepsilon + (1+\varepsilon)\cdot \cJ_0(a;T,q_0).
    \]
\end{enumerate}
\end{thm}

We now state a corollary of Theorem \ref{thm: main simp}; this is the variant of Theorem \ref{thm: intro 2} for starting position $q_0=0$ mentioned in the Introduction. The proof of this corollary is given in Section \ref{sec: q0}.

\begin{cor}\label{cor: q0}
    Let $\varepsilon>0$. Then for $A>0$ sufficiently large depending on $\varepsilon, T, C_0, m_0$, the following holds.

Let $\sigma$ be an $A$-bounded strategy for time horizon $T + \varepsilon$ and starting position $\varepsilon$. Then the strategy $\hat{\sigma}_*$ for time horizon $T$ and starting position $0$ specified in Section \ref{sec: q0} satisfies the following.
\begin{enumerate}
    \item If $a \in [-A,A]$, then
    \[
    \cJ(\hat{\sigma}_*,a;T,0) < \varepsilon + (1+\varepsilon) \cdot \sup\{ \cJ(\sigma,b;T+\varepsilon,\varepsilon): |a-b| <  \varepsilon |a|\}.
    \]
    \item If $|a| > A$, then
    \[
    \cJ(\hat{\sigma}_*,a;T,0) < \varepsilon + (1+\varepsilon)\cdot \cJ_0(a;T,0).
    \]
\end{enumerate}
\end{cor}

We remark that Corollary \ref{cor: q0} implies a variant of Corollary \ref{cor: hr} for $q_0 = 0$, but we don't state it here.

Our proof of Theorem \ref{thm: main simp} makes use of two additional strategies. These strategies, defined in Sections \ref{sec: lqs} and \ref{sec: las}, respectively, are almost optimal when $|q_0|$ and $a$ are large. Specifically, we prove the following theorems.

\begin{thm}\label{thm: lqs}
Let $\varepsilon > 0$. Then there exists $q_{\bg} \ge 1$ depending (continuously) on $\varepsilon,T$ such that the following is true. For any $q_0 \ge q_{\bg}$, let $\elqs$ be the strategy for time horizon $T$ and starting position $q_0$ defined in Section \ref{sec: lqs} depending on $\varepsilon, T,q_0$. Then
\[
\cJ(\elqs, a;T,q_0) \le (1+\varepsilon) \cdot \cJ_0(a;T,q_0)\;\emph{for any}\; a \in \R.
\]
\end{thm}

In our previous paper \cite{carruth2022controlling}, we exhibited a strategy $\br_0$ satisfying
\[
\cJ(\br_0,a;T,0) < C_{T} \cdot \cJ_0(a;T,0)\;\text{for any}\; a \in \R.
\]
By making a simple modification to the strategy $\lqs$ in Theorem \ref{thm: lqs} we can produce, for any $q_0\in \R$, a strategy $\br$ satisfying
\[
\cJ(\br;a,T,q_0) < C_{T,q_0} \cdot \cJ_0(a;T,q_0)\;\text{for any}\;a\in \R;
\]
we refer to this as a bounded regret strategy for starting position $q_0$.

\begin{thm}\label{thm: las}
Let $\varepsilon>0$ and let $\elas$ be the strategy for time horizon $T$ and starting position $q_0$ defined in Section \ref{sec: las} depending on $\varepsilon, T, q_0$. Then for any $A\ge A_{\min}(\varepsilon,T,q_0)$ we have
\[
\cJ(\elas, a; T,q_0) \le (1+\varepsilon) \cdot \cJ_0(a;T,q_0) \;\text{for any} \; a \ge A
\]
and
\[
\cJ(\elas, a;T,q_0) \le C_{T,q_0} \cdot A^2\;\text{for any}\; a \le A.
\]
\end{thm}

\section{Preliminary results on strategies}\label{sec: known a}

We begin this section by stating some properties of strategies.

\begin{rmk}\label{rmk: symmetry}
    By symmetry, any strategy $\sigma$ for time horizon $T$ and starting position $q_0$ gives rise to a strategy $\sigma_-$ for time horizon $T$ and starting position $-q_0$ defined by
\begin{align*}
q^{\sigma_-}(t,a) = - q^\sigma(t,a), \qquad u^{\sigma_-}(t,a) = - u^\sigma(t,a)
\end{align*}
and satisfying
\[
\cJ(\sigma, a;T,q_0) = \cJ(\sigma_-,a;T,-q_0).
\]
\end{rmk}

For the remainder of this paper we assume, without loss of generality, that $q_0>0$; this is justified by Remark \ref{rmk: symmetry}.

\begin{rmk}[Rescaling property]\label{rmk: rescale}
Let $\sigma$ be a strategy for time horizon $T$ and starting position $q_0$ with particle trajectories $q^\sigma(t,a)$ and control variables $u^\sigma(t,a)$. Let $\lambda >0$ and define the rescaled random variables
\begin{align*}
    &\tilde{q}(t,a) = \lambda \cdot q^\sigma\bigg(\frac{t}{\lambda^2},\lambda^2 a\bigg),\qquad \tilde{u}(t,a) = \frac{1}{\lambda} \cdot u^\sigma \bigg( \frac{t}{\lambda^2},\lambda^2 a\bigg),\\
    & \tilde{W}(t) = \lambda \cdot W\bigg(\frac{t}{\lambda^2}\bigg).
\end{align*}
Note that $\tilde{W}(t)$ is a standard Brownian motion and that for any $a \in \R$, $t \in [0,\lambda^2 T]$ we have
\[
\tilde{q}(t,a) = \lambda q_0 + \tilde{W}(t) + \int_0^t[a\tilde{q}(\tau,a) + \tilde{u}(\tau,a)]\ d\tau
\]
with probability 1. It follows that $\tilde{q}(t,a)$, $\tilde{u}(t,a)$ determine a strategy for time horizon $\lambda^2 T$ and starting position $\lambda q_0$. Denoting this strategy by $\sigma_\lambda$, we have
\begin{align*}
\cJ(\sigma_\lambda, a ; \lambda^2 T, \lambda q_0) &= \eE \bigg[ \int_0^{\lambda^2 T} \big(\tilde{q}^2(t,a) + \tilde{u}^2(t,a)\big)dt\bigg]\\
& = \eE\bigg[ \int_0^T \big( \lambda^4\cdot  (q^\sigma(t,\lambda^2 a))^2 +  (u^\sigma(t,\lambda^2 a))^2\big)dt \bigg]\\
& \le \max\{\lambda^4,1\} \cdot \cJ(\sigma, \lambda^2 a; T, q_0).
\end{align*}
We say that $\sigma_\lambda$ is a \emph{rescaling} of the strategy $\sigma$. Note that the rescaled strategy $\sigma_1$ is equal to the original strategy $\sigma$. 
\end{rmk}

\begin{rmk}[Branching strategies]
    We will often decide to switch from one strategy to another.

    To explain how we do that, recall that we denote strategies by $\sigma$, $\sigma'$, $\hat{\sigma}$, etc. We introduce the notion of a \emph{parametrized strategy}, denoted $\sigma(\cdot)$. For each parameter value $\alpha \in \R^d$ (some $d\ge 0$), $\sigma(\alpha)$ is a strategy. (When $d=0$, a parametrized strategy is just a strategy.) For instance, the strategy $\sigma_\emph{\opt}(\beta)$ for time horizon $T$ and starting position $q_0$ is a parametrized strategy; $\beta,T,q_0$ are the parameters. We remark that the parameter $\alpha$ will often include our guess for the unknown $a$. 

    Now suppose we are given parametrized strategies $\sigma_0(\cdot)$, $\sigma_1(\cdot)$,\dots, $\sigma_N(\cdot)$. We will combine the $\sigma$'s into a new parametrized strategy $\sigma^\#$ (a \emph{branching strategy}). Suppose we are given a parameter value $\alpha$ and that $\sigma_0(\alpha)$ is a strategy for some time horizon $T$ and starting position $q_0$. The strategy $\sigma^\#(\alpha)$ is then also a strategy for time horizon $T$ and starting position $q_0$.

    For each $a \in \R$, we pick a stopping time $\tau(a,\alpha) \le T$ depending on $\alpha$ and satisfying the following: for any $a,b \in \R$, $t \in [0,T]$, and $\omega \in \Omega$, if $q^{\sigma_0(\alpha)}(s,a,\omega)=q^{\sigma_0(\alpha)}(s,b,\omega)$ for all $s \in [0,t]$ then $\tau(a,\alpha,\omega)>t$ if and only if $\tau(b,
    \alpha,\omega)>t$. Note that this condition ensures that the stopping times are defined only in terms of $q$ and cannot use knowledge of the unknown $a$. The strategy $\sigma^\#(\alpha)$ proceeds as follows.
    
    Until time $\tau(a,\alpha)$, we execute the strategy $\sigma_0(\alpha)$. If $\tau(a,\alpha) = T$, then we are done.

    If $\tau(a,\alpha) < T$, then we will pick a new parametrized strategy $\hat{\sigma}(\cdot)$ and a new parameter $\hat{\alpha}$. The $\hat{\sigma}(\cdot)$ will be one of our given parametrized strategies $\sigma_1(\cdot)$, \dots, $\sigma_N(\cdot)$. Our choice of $\hat{\sigma}(\cdot)$ and $\hat{\alpha}$ is determined by $\alpha$, the history up to time $\tau(a,\alpha)$, and the requirement that $\hat{\sigma}(\hat{\alpha})$ is a strategy for starting position $\hat{q}_0 = q(\tau(a,\alpha))$ and some time horizon $\hat{T} \ge (T - \tau(a,\alpha))$.

    Once we have picked $\hat{\sigma}(\cdot)$ and $\hat{\alpha}$, we forget the past, regard $t = \tau(a,\alpha)$ as if we were at $t=0$, and execute the strategy $\hat{\sigma}(\hat{\alpha})$, starting at position $\hat{q}_0$. We stop playing at time $T$.

    Thus, we have combined our parametrized strategies $\sigma_0(\cdot)$, $\sigma_1(\cdot),$ \dots, $\sigma_N(\cdot)$ into a branching (parametrized) strategy $\sigma^\#(\cdot)$. 

    We may then combine our $\sigma^\#(\cdot)$ with additional parametrized strategies to form further branching strategies.

\end{rmk}

We note that the strategies $\sigma_*$ and $\tilde{\sigma}$ constructed in Section \ref{sec: almost optimal strategy}, the strategy $\lqs$ constructed in Section \ref{sec: lqs}, and the strategy $\las$ constructed in Section \ref{sec: las} are all examples of branching strategies.

\subsection{Optimal known-\texorpdfstring{$a$}{a} strategies}

Let $a \in \R$. In Section \ref{sec: setup}, we defined the optimal expected cost for known $a$ by
\[
\cJ_0(a;T,q_0) = \cJ(\sigma_\opt(a),a;T,q_0),
\]
where $\sigma_\opt(a)$ is the simple feedback strategy with gain function
\begin{equation}\label{eq: kappa def}
\kappa(T-t,\alpha) = \frac{\tanh((T-t)\sqrt{\alpha^2+1})}{\sqrt{\alpha^2+1}-\alpha\tanh((T-t)\sqrt{\alpha^2+1})};
\end{equation}
recall that we refer to $\sigma_\opt(a)$ as an optimal known-$a$ strategy. Observe that (for fixed $\alpha$) $\kappa$ solves a Riccati equation:
\[
\frac{d \kappa}{d t}(t,\alpha) = [1+2a\kappa(t,\alpha) - \kappa^2(t,\alpha)].
\]
For $t \in [0,T]$ we define
\begin{equation}\label{eq: K def}
K(t,\alpha) = \int_0^{t} \kappa(s,\alpha)\ ds.
\end{equation}
Clearly, we have (again for fixed $\alpha$)
\[
\frac{d K}{d t}(t,\alpha) = \kappa(t,\alpha).
\]
It is well-known (see, for example, \cite{astrom}) that
\begin{equation}\label{eq: known a 6.01}
\cJ_0(a;T,q_0) = \kappa(T,a)\cdot q_0^2 + K(T,a).
\end{equation}
This implies the following remark.
\begin{rmk}\label{cor: q_0s}
Let $0< q_0'< q_0''$. Then
\[
\cJ_0(a;T,q_0') < \cJ_0(a;T, q_0'') < \bigg(\frac{q_0''}{q_0'}\bigg)^2 \cJ_0(a;T, q_0') \;\text{for any}\; a \in \R.
\]
\end{rmk}

Combining \eqref{eq: kappa def} and \eqref{eq: K def} and evaluating the resulting integral, we get
\begin{multline*}
K(T,a) =  \big(a+\sqrt{a^2+1}\big)T+\log\bigg(\frac{\sqrt{a^2+1}-a}{2\sqrt{a^2+1}}\bigg)\\
+\log\bigg(1+e^{-2T\sqrt{a^2+1}}\bigg(\frac{\sqrt{a^2+1}+a}{\sqrt{a^2+1}-a}\bigg)\bigg).
\end{multline*}
Note that for any $\varepsilon>0$, there exists $A>0$ depending on $\varepsilon$ and the time horizon $T$ such that:
\begin{itemize}
\item For $a > A$, we have
\begin{align}
&|\kappa(T,a)-2a| < \varepsilon a,\label{eq: known a 6.2}\\
& |K(T,a)-2aT| < \varepsilon a T. \label{eq: known a 6.3}
\end{align}
\item For $a<-A$, we have
\begin{align}
&\bigg| \kappa(T,a) - \frac{1}{2|a|}\bigg| \le \frac{\varepsilon}{|a|},\label{eq: known a 6.4}\\
&\bigg| K(T,a) - \frac{T}{2|a|}\bigg| \le \frac{\varepsilon T}{|a|}.\label{eq: known a 6.5}
\end{align}
\end{itemize}

From \eqref{eq: known a 6.01}--\eqref{eq: known a 6.5}, we deduce that for any $\varepsilon>0$, there exists $A>0$ depending on $\varepsilon$ and $T$ such that
\begin{align}
&|\cJ_0(a;T,q_0) - 2a(q_0^2 + T)| < \varepsilon a (q_0^2+T) \;\text{when}\; a > A,\; \text{and}\label{eq: ka 1}\\
&\bigg| \cJ_0(a;T,q_0) - \frac{1}{2|a|}(q_0^2+T)\bigg| < \frac{\varepsilon}{|a|}(q_0^2+T)\;\text{when}\; a<-A.\label{eq: ka 2}
\end{align}

Now let $\varepsilon>0$ be arbitrary and introduce $\delta>0$ sufficiently small depending on $\varepsilon,T,q_0$. Suppose that $T'>0$ and $q_0' \in \R$ satisfy $|T-T'|, |q_0-q_0'|<\delta$. We claim that
\begin{multline}\label{eq: ka 3}
|\cJ_0(a';T',q_0') - \cJ_0(a;T,q_0)|\\ < \varepsilon \cdot \cJ_0(a;T,q_0)\; \text{for any} \; a \in \R, |a- a'|<|a| \delta.
\end{multline}
 By \eqref{eq: ka 1} and \eqref{eq: ka 2} above, there exists $A >0$ depending on $\varepsilon, T$ such that:
    \begin{enumerate}
    \item For any $a, a' > A$, we have
    \begin{align}
    &|\cJ_0(a'; T', q_0') - 2a'( (q_0')^2+T')| < \varepsilon a' ((q_0')^2+T') \label{eq: cont lem 1}\\ 
    &| \cJ_0(a; T, q_0) - 2a(q_0^2+T)| < \varepsilon a (q_0^2+T).\label{eq: cont lem 2}
    \end{align}
    \item For any $a, a' <- A$, we have
    \begin{align}
    &\bigg|\cJ_0(a';T',q_0') - \frac{1}{2|a'|}((q_0')^2+T')\bigg| < \frac{\varepsilon}{|a'|}((q_0')^2+T'),\label{eq: cont lem 3}\\
    &\bigg|\cJ_0(a;T,q_0) - \frac{1}{2|a|}(q_0^2+T)\bigg| < \frac{\varepsilon}{|a|}(q_0^2+T).\label{eq: cont lem 4}
    \end{align}
    \end{enumerate}
 Combining \eqref{eq: cont lem 1}, \eqref{eq: cont lem 2}, and using the assumptions $|T-T'|< \delta$ and $|q_0 - q_0'| < \delta$, we get
\begin{multline*}
|\cJ_0(a';T',q_0') - \cJ_0(a;T,q_0)| < C_{T,q_0}\cdot a \cdot( \varepsilon+ \delta)\\ \text{for any}\; a\ge 2A, \; |a-a'| < \delta a.
\end{multline*}
Taking $\delta$ sufficiently small depending on $\varepsilon$, we use \eqref{eq: cont lem 2} to deduce that (for $A$ sufficiently large depending on $\varepsilon$ and $T$) we have
\begin{equation}\label{eq: cont lem 5}
\begin{split}
|\cJ_0(a';T',q_0') - \cJ_0(a;T,q_0)| < C'_{T,q_0}\cdot \varepsilon \cdot &\cJ_0(a;T,q_0)\\ &\text{for any}\; a \ge 2A, \; |a-a'| < \delta a.
\end{split}
\end{equation}
 Similarly, we use \eqref{eq: cont lem 3}, \eqref{eq: cont lem 4} to deduce that
\begin{multline}\label{eq: cont lem 6}
|\cJ_0(a';T',q_0') - \cJ_0(a;T,q_0)| < C'_{T,q_0}\cdot \varepsilon \cdot \cJ_0(a;T,q_0)\\
\text{for any}\; a \le  - 2 A, |a - a' | < \delta |a|.
\end{multline}
Note that \eqref{eq: cont lem 5} and \eqref{eq: cont lem 6} imply that for a sufficiently large number $\tilde{A}>A$ depending on $\varepsilon$, $T$, $q_0$, we have
\begin{multline}
    |\cJ_0(a';T',q_0') - \cJ_0(a;T,q_0)| \\ < \varepsilon \cdot \cJ(a;T,q_0)\;\text{for any} \; |a| \ge 2 \tilde{A}, |a- a'| < \delta |a|.
\end{multline}

Next, we note that \eqref{eq: known a 6.01} implies that $\cJ_0(a;T,q_0)$ is of the form
\[
\cJ_0(a; T, q_0) = f_1(a,T)q_0^2 + f_2(a,T)
\]
for smooth functions $f_1, f_2: \R \times (0,\infty)\rightarrow (0, \infty)$ independent of $q_0$. Therefore (since $\tilde{A}$ is determined by $\varepsilon$, $T$, $q_0$), we have
\begin{multline*}
|\cJ_0(a';T',q_0') - \cJ_0(a;T,q_0)| < C_{\varepsilon, T, q_0} \cdot \delta \\ \text{for any}\; |a|, |a'| \le 3 \tilde{A}\;\text{and}\; |a-a'|< \delta |a|
\end{multline*}
and
\[
\cJ_0(a;T,q_0) >  c_{\varepsilon,T, q_0}\;\text{for any} \; |a| \le 3 \tilde{A}.
\]
Combining the last two inequalities gives
\begin{equation}\label{eq: cont lem 7}
\begin{split}
|\cJ_0(a';T',q_0') - \cJ_0(a;T,q_0)| < C_{\varepsilon, T, q_0}& \cdot \delta \cdot \cJ_0(a;T,q_0)\\ &\text{for any}\; |a| \le 2\tilde{A}, |a-a'|<\delta |a|.
\end{split}
\end{equation}
Combining \eqref{eq: cont lem 5}, \eqref{eq: cont lem 6}, \eqref{eq: cont lem 7}, and taking $\delta$ sufficiently small depending on $\varepsilon, T, q_0$ proves \eqref{eq: ka 3}.

We summarize the above discussion (specifically \eqref{eq: known a 6.01}, \eqref{eq: ka 1}, \eqref{eq: ka 2}, and \eqref{eq: ka 3}) as a lemma.

\begin{lem}\label{lem: OS asymp}
The optimal expected cost for known $a$, $\cJ_0$, has the following properties:
\begin{enumerate}
\item
\[
\cJ_0(a;T,q_0) = \kappa(T,a)\cdot q_0^2 + K(T,a).
\]
\item For any $\varepsilon>0$, there exists $A>0$ depending on $\varepsilon$ and $T$ such that
\begin{align}
&|\cJ_0(a;T,q_0) - 2a(q_0^2 + T)| < \varepsilon a (q_0^2+T) \;\text{when}\; a > A,\; \text{and}\\
&\bigg| \cJ_0(a;T,q_0) - \frac{1}{2|a|}(q_0^2+T)\bigg| < \frac{\varepsilon}{|a|}(q_0^2+T)\;\text{when}\; a<-A.
\end{align}
\item Let $\varepsilon>0$. Let $\delta>0$ be sufficiently small depending on $\varepsilon$, $T$, $q_0$, and suppose that $T'>0$ and $q_0' \in \R$ satisfy $|T-T'| < \delta$ and $|q_0 - q_0'| < \delta$. Then for any $a \in \R$ we have
\[
|\cJ_0(a';T',q_0') - \cJ_0(a;T,q_0)| < \varepsilon \cdot \cJ_0(a;T,q_0)\;\text{for any} \; |a- a'|<|a| \delta.
\]
\end{enumerate}
\end{lem}

The next lemma says that the strategy $\sigma_\opt(a)$ is indeed optimal for known $a$, i.e., for fixed $a \in \R$ the strategy $\sigma_\opt(a)$ minimizes the quantity $\cJ(\sigma,a;T,q_0)$.

\begin{lem}\label{lem: OS}
Let $\sigma$ be an arbitrary strategy for time horizon $T$ and starting position $q_0$. Then
\[
\cJ(\sigma,a;T,q_0) \ge \cJ_0(a;T,q_0)\;\text{for any}\; a \in \R.
\]
\end{lem}
\begin{proof}
We suppose that $N\gg 1$ is a sufficiently large integer depending on $a,T$; we write $c,c',C,C',\dots$ to denote constants determined by $a,T$. The symbols $c,c',C,C',\dots$ may denote different constants in different occurrences.

We set $h = T/N$ and for $\nu = 0,\dots, N$ we set $t_\nu = \nu h$.

By downward induction on $\nu$, we will show that
\begin{multline}\label{eq: os 1}
    \E\bigg[ \int_{t_\nu}^T((q(s))^2 + (u(s))^2)\ ds \bigg| \cF_{t_\nu}\bigg]\cdot (1 + h^{1/100})\\ \ge \kappa(T-t_\nu,a)(q(t_\nu))^2 + K(T-t_\nu,a) - h^{1/100}(T-t_\nu).
\end{multline}
Once we prove \eqref{eq: os 1}, we take $\nu = 0 $ and let $N \rightarrow \infty$ to derive the conclusion of the lemma. So our task is to prove \eqref{eq: os 1}. We begin our induction on $\nu$.

In the base case $\nu = N$, \eqref{eq: os 1} holds since $\kappa(0,a) = K(0,a) =0$.

For the induction step, we fix $\nu$ ($0 \le \nu < N)$, and assume the inductive hypothesis
\begin{multline}\label{eq: os 2}
    \E\bigg[ \int_{t_{\nu+1}}^T ((q(s))^2 + (u(s))^2)\ ds \bigg| \cF_{t_{\nu+1}}\bigg]\cdot (1 + h^{1/100})\\ \ge \kappa(T-t_{\nu+1},a)\cdot (q(t_{\nu+1}))^2 + K(T-t_{\nu+1},a) - h^{1/100}(T-t_{\nu+1}).
\end{multline}
Our goal is then to prove \eqref{eq: os 1} assuming \eqref{eq: os 2}.

Recall that our Brownian motion is denoted by $(W(t))_{t\ge0}$.

We set
\begin{flalign}
    &\Delta W_\nu = W(t_{\nu+1}) - W(t_\nu),\label{eq: os 3}&\\
    &\omega(\nu) = \sup \{ |W(t) - W(t_\nu)| : t \in [t_\nu, t_{\nu+1}]\},\label{eq: os 3.5}\\
    &||u||_\nu = \bigg(\int_{t_\nu}^{t_{\nu+1}} u^2\ ds\bigg)^{1/2},\label{eq: os 4}\\
    &\bar{u}_\nu = \frac{1}{h} \int_{t_\nu}^{t_{\nu+1}} u \ ds\label{eq: os 4.5}.
\end{flalign}
For the rest of the proof, we condition on $\cF_{t_\nu}$, and we write $\E[\cdots]$ to denote the expectation conditioned on $\cF_{t_\nu}$.

We have
\begin{equation}\label{eq: os 5}
    \E[\Delta W_\nu]=0,\quad \E[(\Delta W_\nu)^2] = h, \quad \E[(\omega(\nu))^2]\le C h,
\end{equation}
and
\begin{equation}\label{eq: os 6}
    h | \bar{u}_\nu|^2 \le ||u||_\nu^2;
\end{equation}
more generally,
\begin{equation}\label{eq: os 6.5}
\bigg| \int_{t_\nu}^t u \ ds \bigg| \le h^{1/2} ||u||_\nu\; \text{for}\; t \in [t_\nu, t_{\nu+1}].
\end{equation}
By the definition of a strategy, we have
\begin{multline}\label{eq: os 7}
    q(t) - q(t_\nu) = [W(t) - W(t_\nu)] + a q(t_\nu)(t-t_\nu) + \int_{t_\nu}^t a [ q(s) - q(t_\nu)]\ ds \\ + \int_{t_\nu}^t u\ ds\; \text{for} \; t \in [t_\nu, t_{\nu+1}].
\end{multline}

Setting
\begin{equation}\label{eq: os 8}
    \osc(\nu) = \sup\{ |q(t) - q(t_\nu)| : t \in [t_\nu, t_{\nu+1}]\},
\end{equation}
we deduce that
\[
\osc(\nu) \le \omega(\nu) + |a q(t_\nu)|h + |a| \osc(\nu)\cdot h + h^{1/2} ||u||_\nu.
\]
Since $h$ is less than a small enough constant determined by $a$, we may absorb the term $|a| \osc(\nu)h$ into the left-hand side above, to conclude that
\begin{equation}\label{eq: os 9}
    \osc(\nu) \le C \omega(\nu) + C  h\cdot |a q(t_\nu)| + C h^{1/2}||u||_\nu.
\end{equation}
From \eqref{eq: os 8}, \eqref{eq: os 9}, we see that
\[
|q(t) - q(t_\nu)| \le C \omega(\nu) + C h |q(t_\nu)| + C h^{1/2} ||u||_\nu\;\text{for} \; t \in [t_\nu, t_{\nu+1}],
\]
hence
\begin{equation}\label{eq: os 10}
    \int_{t_\nu}^{t_{\nu+1}} q^2\ ds \ge (1- Ch^{1/10})(q(t_\nu))^2h - C h^{9/10}(\omega(\nu))^2 - C h^{19/10}||u||_\nu^2.
\end{equation}
Also, \eqref{eq: os 7}, \eqref{eq: os 8}, \eqref{eq: os 9} yield
\begin{flalign}
& q(t_{\nu+1}) = (1+ah)q(t_\nu) + \Delta W_\nu + \bar{u}_\nu h + \text{ERR}(\nu),\; \text{with}&\label{eq: os 11}\\
& | \text{ERR}(\nu)| \le C h \omega(\nu) + C h^2 |q(t_\nu)| + C h^{3/2}||u||_\nu.\label{eq: os 12}
\end{flalign}
Using \eqref{eq: os 11}, \eqref{eq: os 12}, we estimate
\begin{multline}\label{eq: os 13}
    (q(t_{\nu+1}))^2 = [ (1+ah)q(t_\nu) + \Delta W_\nu + \bar{u}_\nu h ]^2 + (\text{ERR}(\nu))^2 \\+ 2 [(1+ah)q(t_\nu) + \Delta W_\nu + \bar{u}_\nu h] \cdot (\text{ERR}(\nu)).
\end{multline}
We have
\begin{equation}\label{eq: os 14}
    \begin{split}
        |[(1+a&h)q(t_\nu) + \Delta W_\nu + \bar{u}_\nu h ] \cdot (\text{ERR}(\nu))|\\ \le &C[|q(t_\nu)| + \omega(\nu) + h^{1/2}||u||_\nu] \cdot [h\omega(\nu) + h^2|q(t_\nu)| + C h^{3/2}||u||_\nu]\\
        \le & C h^2|q(t_\nu)|^2 + Ch(\omega(\nu))^2 + C h^2 ||u||_\nu^2 + C h^{3/2} |q(t_\nu)|\cdot ||u||_\nu \\&+ C h |q(t_\nu)| \omega(\nu) + C h^{3/2} ||u||_\nu \omega(\nu).
    \end{split}
\end{equation}
Into \eqref{eq: os 14} we substitute the estimates
\begin{flalign*}
    & h^{3/2} |q(t_\nu)|\cdot ||u||_\nu \le C h^{3/2} |q(t_\nu)|^2 + C h^{3/2}||u||_\nu^2,&\\
    &h |q(t_\nu)| \omega(\nu) \le  C h^{5/4} | q(t_\nu)|^2 + C h^{3/4} (\omega(\nu))^2,\;\text{and}\\
    &h^{3/2} ||u||_\nu \omega(\nu) \le C h^{3/2} ||u||_\nu^2 + C h^{3/2} (\omega(\nu))^2.
\end{flalign*}
We find that
\begin{multline*}
    |[(1+ah)q(t_\nu) + \Delta W_\nu + \bar{u}_\nu h ] \cdot \text{ERR}(\nu)|\\ \le C h^{5/4} |q(t_\nu)|^2 + C h^{3/2}||u||_\nu^2 + C h^{3/4}(\omega(\nu))^2.
\end{multline*}
Consequently, \eqref{eq: os 13} implies the estimate
\begin{equation}\label{eq: os 15}
\begin{split}
    (q(t_{\nu+1}))^2 \ge &[(1+ah)q(t_\nu) + \Delta W_\nu + \bar{u}_\nu h]^2 - C h^{5/4} |q(t_\nu)|^2 \\
    &- C h^{3/2} ||u||_\nu^2 - C h^{3/4}(\omega(\nu))^2\\
    = & \{ (1+ah)^2 - C h^{5/4}\} (q(t_\nu))^2 + (\Delta W_\nu)^2 + \bar{u}_\nu^2 h^2 \\
    &+ 2(1+ah) q(t_\nu) \Delta W_\nu + 2(1+ah) q(t_\nu)\bar{u}_\nu h + 2 \Delta W_\nu \bar{u}_\nu h\\
    & - C h^{3/2}||u||_\nu^2 - C h^{3/4}(\omega(\nu))^2.
\end{split}
\end{equation}
Since 
\begin{align*}
|\Delta W_\nu \bar{u}_\nu h | &\le |\Delta W_\nu | \cdot h^{1/2} ||u||_\nu \\
&\le C h^{1/2} | \Delta W_\nu|^2 + C h^{1/2} ||u||_\nu^2\\
& \le C h^{1/2} (\omega(\nu))^2 + C h^{1/2}||u||_\nu^2,
\end{align*}
estimate \eqref{eq: os 15} implies the following:
\begin{equation}\label{eq: os 16}
\begin{split}
(q(t_{\nu+1}))^2 \ge& \{(1+ah)^2 - C h^{5/4}\}(q(t_\nu))^2 + (\Delta W_\nu)^2 \\
&+ 2(1+ah)q(t_\nu)\Delta W_\nu + 2(1+ah)q(t_\nu)\bar{u}_\nu h - C h^{1/2}||u||_\nu^2 \\
&- C h^{1/2}(\omega(\nu))^2.
\end{split}
\end{equation}
Recall that we are conditioning on $\cF_{t_\nu}$, so that $q(t_\nu)$ is deterministic. From \eqref{eq: os 5} and \eqref{eq: os 16} we therefore learn that
\begin{equation}\label{eq: os 17}
    \begin{split}
        \E[(q(t_{\nu+1}))^2] \ge& (1+2ah - C h^{5/4}) (q(t_\nu))^2 + h + 2h(1+ah)q(t_\nu)\E[\bar{u}_\nu]\\
        & - C h^{1/2}\E[||u||_\nu^2] - C h^{3/2},
    \end{split}
\end{equation}
while \eqref{eq: os 5} and \eqref{eq: os 10} yield
\begin{equation}\label{eq: os 18}
\begin{split}
    \E\bigg[ \int_{t_\nu}^{t_{\nu+1}} q^2 \ ds\bigg] \ge & h (1- Ch^{1/10}) (q(t_\nu))^2 - C h^{19/10} \E[||u||_\nu^2]\\
    &- C h^{19/10}.
\end{split}
\end{equation}

We bring the inductive assumption \eqref{eq: os 2} into play. Thanks to \eqref{eq: os 2}, \eqref{eq: os 17}, \eqref{eq: os 18}, we have
\begin{equation*}
    \begin{split}
    \E\bigg[ \int_{t_\nu}^T &(q^2 + u^2)\ ds\bigg]\cdot (1+h^{1/100})\\
    = & \E\bigg[ \int_{t_\nu}^{t_{\nu+1}} q^2 \ ds\bigg] \cdot (1+h^{1/100}) + \E\bigg[ \int_{t_\nu}^{t_{\nu+1}} u^2\ ds\bigg] \cdot( 1+h^{1/100}) \\
    & + \E\bigg[ \E\bigg[ \int_{t_{\nu+1}}^T( q^2 + u^2)\ ds \bigg| \cF_{t_{\nu+1}}\bigg]\cdot (1+h^{1/100})\bigg]\\
    \ge & \{ h (1- Ch^{1/10})(1+h^{1/100}) (q(t_\nu))^2 - C h^{19/10} \E[||u||_\nu^2] - C h^{19/10}\}\\
    & + \{ \E[||u||_\nu^2](1+h^{1/100})\} \\
    &+ \{\kappa(T-t_{\nu+1},a) \E[(q(t_{\nu+1}))^2] + K(T-t_{\nu+1},a) - h^{1/100}(T-t_{\nu+1})\}\\
    \ge & h\Big(1+\frac{1}{2}h^{1/100}\Big)(q(t_\nu))^2 + \E\Big[||u||_\nu^2\Big(1+\frac{1}{2}h^{1/100}\Big)\Big] - C h^{19/10}\\
    & + \kappa(T-t_{\nu+1},a) \Big\{(1+2ah - C h^{5/4})(q(t_\nu))^2 + h + 2h(1+ah)q(t_\nu)\E[\bar{u}_\nu]\\
    &\qquad - C h^{1/2} \E[||u||_\nu^2] - Ch^{3/2}\Big\} + K(T-t_{\nu+1},a) - h^{1/100}(T-t_{\nu+1}),
    \end{split}
\end{equation*}
and consequently,
\begin{equation}\label{eq: os 19}
\begin{split}
    \E\bigg[ \int_{t_\nu}^T (q^2+&u^2)\ ds\bigg](1+h^{1/100}) \\
    \ge & \Big\{\kappa(T-t_{\nu+1},a) + h[2a\kappa(T-t_{\nu+1},a)+1]+\frac{1}{4}h^{101/100}\Big\} (q(t_\nu))^2\\
    & + \E[ ||u||_\nu^2 + 2h(1+ah)\kappa(T-t_{\nu+1},a)q(t_\nu)\bar{u}_\nu]\\
    & + \{ K(T-t_{\nu+1},a) - h^{1/100}(T-t_{\nu+1}) + \kappa(T-t_{\nu+1},a) h - Ch^{3/2}\}
\end{split}
\end{equation}
Now, recalling \eqref{eq: os 6}, we see that
\begin{equation*}
    \begin{split}
        ||u||_\nu^2 + 2h(1+ah)\kappa(T-t_{\nu+1},&a) q(t_\nu)\bar{u}_\nu\\
        \ge &\bar{u}_\nu^2 h + 2h(1+ah)\kappa(T-t_{\nu+1},a)q(t_\nu)\bar{u}_\nu\\
        = &h(\bar{u}_\nu + (1+ah)\kappa(T-t_{\nu+1},a)q(t_\nu))^2 \\
        &- h(1+ah)^2\kappa^2(T-t_{\nu+1},a)(q(t_\nu))^2\\
        \ge & - h(1+ah)^2\kappa^2(T-t_{\nu+1},a)(q(t_\nu))^2.
    \end{split}
\end{equation*}
Therefore, from \eqref{eq: os 19}, we have
\begin{equation}\label{eq: os 20}
    \begin{split}
        \E\bigg[\int_{t_\nu}^T &(q^2+u^2) \ ds\bigg] \cdot (1+h^{1/100})\\
        \ge& \Big\{ \kappa(T-t_{\nu+1},a) + h[1+2a\kappa(T-t_{\nu+1},a) \\
        &\qquad - (1+ah)^2 \kappa^2(T-t_{\nu+1},a)] + \frac{1}{4}h^{101/100}\Big\}(q(t_\nu))^2 \\
        &+ \{ K(T-t_{\nu+1},a) - h^{1/100}(T-t_{\nu+1}) \\
        &\qquad+ \kappa(T-t_{\nu+1},a)h - C h^{3/2}\}.
    \end{split}
\end{equation}
Recall that
\[
\frac{d}{dt}[\kappa(T-t,a)] = - [1+2a\kappa(T-t,a)-\kappa^2(T-t,a)]
\]
and that
\[
\frac{d}{dt}[K(T-t,a)] = - \kappa(T-t,a).
\]
Therefore the expressions in curly brackets on the right in \eqref{eq: os 20} are bounded below, respectively, by $\kappa(T-t_\nu,a)$ and by
\[
K(T-t_\nu) - h^{1/100}(T-t_{\nu+1}) - C h^{3/2}.
\]
Consequently, \eqref{eq: os 20} implies the estimate
\begin{equation}\label{eq: os 21}
\begin{split}
    \E\bigg[ \int_{t_\nu}^T (q^2 + &u^2)\ ds\bigg] \cdot (1+h^{1/100}) \\
    \ge & \kappa(T-t_\nu,a)\cdot(q(t_\nu))^2 + K(T-t_\nu,a)-h^{1/100}(T-t_{\nu+1}) - Ch^{3/2}\\
    \ge & \kappa(T-t_\nu,a)\cdot (q(t_\nu))^2 + K(T-t_\nu,a) - h^{1/100}(T-t_\nu).
\end{split}
\end{equation}
Recalling that $\E[\cdots]$ here denotes expectation conditioned on $\cF_{t_\nu}$, we see that \eqref{eq: os 21} is precisely our desired inequality \eqref{eq: os 1}.

This completes our downward induction on $\nu$, proving the lemma.
\end{proof}

\subsection{The expected cost of simple feedback strategies}
Let $v:[0,T]\rightarrow \R$ be a smooth function and let $\sigma_v$ denote the simple feedback strategy for time horizon $T$ and starting position $q_0$ with gain function $v: [0,T]\rightarrow \R$ (see Section \ref{sec: setup} for the definition of a simple feedback strategy). We fix some $ a\in \R$ and let $q$ denote the particle trajectory $q^{\sigma_v}(\cdot,a)$. Note that $q$ solves the stochastic ODE
\begin{equation}\label{eq: sfsr 1}
dq = (a - v)q dt + dW, \qquad q(0) = q_0
\end{equation}
and that
\begin{equation}\label{eq: sfsr 1.5}
\cJ(\sigma_v, a;T,q_0) = \E\bigg[\int_0^T q^2(\tau)\big(1+v^2(\tau)\big)\ d\tau \bigg].
\end{equation}

Define a (smooth) function $\phi:[0,T]\rightarrow \R$ by
\[
\phi(t)= \int_t^T\big(1+v^2(\tau)\big)\exp\bigg(2 \int_t^\tau \big(a-v(s)\big) \ ds\bigg) d\tau;
\]
$\phi$ solves the ODE
\begin{equation}\label{eq: ode add}
-\phi'(t) = 2 \phi(t)(a-v(t))+1+v^2(t), \qquad \phi(T) = 0.
\end{equation}
By It\^{o}'s Lemma we have
\[
d(q^2\phi) =  [\phi'q^2  +\phi] dt  + (2q \phi)  dq.
\]
Combining this with \eqref{eq: sfsr 1} and \eqref{eq: ode add} gives
\begin{equation}\label{eq: sfsr 2}
d(q^2\phi) = [-(1+v^2)q^2 +\phi] dt +(2\phi q) dW.
\end{equation}
With probability 1, we have
\[
\int_0^T d (q^2 \phi) = q^2(T) \phi(T) -  q^2(0) \phi(0) = - q_0^2 \phi(0);
\]
combining this with \eqref{eq: sfsr 2} gives
\begin{equation}
    \int_0^Tq^2(\tau)(1+v^2(\tau))\ d\tau = \phi(0) q_0^2 + \int_0^T \phi(\tau)\ d\tau + \int_0^T 2 \phi(\tau)q(\tau)\ d W_\tau.
\end{equation}
Taking expectation and using \eqref{eq: sfsr 1.5}, we get
\[
\cJ(\sigma_v,a;T,q_0) = \phi(0) q_0^2 + \int_0^T \phi(\tau)\ d\tau.
\]
We summarize the above result as a lemma.

\begin{lem}\label{lem: lin ctrl formula}
Let $v:[0,T]\rightarrow \R$ be a smooth function and let $\sigma_v$ be the simple feedback strategy with gain function $v$. Then
\[
\cJ(\sigma_v,a;T,q_0) = \phi(0) \cdot q_0^2 + \int_0^T \phi(s)ds,
\]
where
\[
\phi(t) = \int_t^T (1+v^2(s)) \exp\bigg(2\int_t^s(a-v(r))dr\bigg) ds.
\]
\end{lem}

Let $\alpha \ge 0$ be a real number. We define the \emph{constant gain strategy} $\cg(\alpha)$ to be the simple feedback strategy with the constant gain function $v(t) \equiv \alpha$. We remark that $\cg(\alpha)$ is independent of the time horizon $T$ and the starting position $q_0$.

Lemma \ref{lem: lin ctrl formula} then implies that for any $a \in \R$ we have
\begin{multline}\label{eq: cg 1}
\cJ(\text{\cg}(\alpha),a;T,q_0) = (1+\alpha^2)\Big[ \frac{1}{2(a-\alpha)} \Big[\Big(\frac{e^{2T(a-\alpha)}-1}{2(a-\alpha)}\Big)-T\Big] \\ + \Big(\frac{e^{2T(a-\alpha)}-1}{2(a-\alpha)}\Big)q_0^2 \Big]  \; \text{for any}\;  \alpha \ge 0, \alpha \ne a,
\end{multline}
and for any $a  \ge 0$ we have
\begin{equation}\label{eq: cg 2}
\cJ(\text{CG}(a),a;T,q_0) = (1+a^2)T\Big(q_0^2+\frac{1}{2}T\Big).
\end{equation}
Using \eqref{eq: cg 1} and \eqref{eq: cg 2} we deduce the following corollary.
\begin{cor}\label{cor: CG score}
For $\alpha \ge 0$ and $a \in \R$ we have
    \[
    \cJ(\emph{CG}(\alpha),a;T,q_0) \le \begin{cases}
        C_T (1+\alpha^2)(1+q_0^2) e^{2T(a-\alpha)} &\text{when}\; (a - \alpha)>1/T,\\
        C_T(1+\alpha^2)(1+q_0^2) &\text{when}\; |a - \alpha | \le 1/T,\\
        \frac{(1+\alpha^2)}{2|a-\alpha|}(q_0^2+T) &\text{when}\; (a - \alpha)< -1/T.
    \end{cases}
    \]
\end{cor}

For any $\alpha \in \R$ the strategy $\sigma_\opt(\alpha)$ is a simple feedback strategy with gain function $ \kappa(T-t,\alpha)$. Note that this strategy depends on $T$ but is independent of $q_0$. By Lemma \ref{lem: OS asymp}, and because $\kappa(T-t,a) \le C \max\{1,a\}$ for all $a \in \R$, $t \in [0,T]$, we have $\cJ_0(a;T,q_0) \le C \max\{1,a\}(q_0^2+T)$. We combine this with Lemma \ref{lem: lin ctrl formula} to deduce the following corollary.
\begin{cor}\label{cor: opt alph score}
For any $a \in \R$ the following holds.
\[
\cJ(\sigma_{\emph{opt}}(\alpha),a;T,q_0) \le \begin{cases}
C \max\{1,a\}(q_0^2+T) &\text{if}\; \alpha= a,\\
C_T(1+\alpha^2)e^{2|a|T}(1+q_0^2) &\text{for any}\; \alpha \in \R.
\end{cases}
\]
\end{cor}
We remark that the first upper bound in Corollary \ref{cor: opt alph score} (for $\alpha = a$) is sharp when $a$ is large. The second upper bound is far from sharp unless $a \gg \alpha \gg 1$.

\section{The uncontrolled system}\label{sec: stopping}
Let $T>0$, $q_0 > 0$, and let $a \in \R$ be arbitrary. In this section we consider the dynamics
\begin{equation}\label{eq: stopping 1}
dq = (aq) dt + dW_t,\qquad q(0) = q_0.
\end{equation}
We define a random process
\begin{equation}\label{eq: xt def}
X_t = e^{-at}q(t) - q_0 = \int_0^t e^{-as} dW_s.
\end{equation}
If $a=0$, then $X_t$ is standard Brownian motion. If $a \ne 0$, then $X_t$ is a normal random variable with
\begin{equation}\label{eq: xa}
\E[X_t] = 0, \qquad \var[X_t] = \frac{1-e^{-2at}}{2a}.
\end{equation}
For the remainder of this paper, we adopt the convention that $(1-e^{-2at})/2a$ is equal to $t$ when $a = 0$. With this convention in place, \eqref{eq: xa} holds for all $a \in \R$. As a consequence of the above, for any $\delta >0$ we have
\begin{equation}\label{eq: xb}
    \prob(|X_t| > \delta) \le
    \begin{cases}
    C \exp(-c \delta^2 a)\; \text{for any} \; a \ge 0, t\ge 0,\\
    C \exp(-c \delta^2 / t)\; \text{for any}\; |a| t \le 1/10.
    \end{cases}
\end{equation}
Note that these two cases are not mutually exclusive.

We also note that $X_t$ satisfies the reflection principle, i.e., for any $M>0$ we have
\begin{equation}\label{eq: xc}
\prob\left( \sup_{0 \le s \le t} X_s \ge M\right) = 2\cdot  \prob(X_t \ge M).
\end{equation}
This is true because (1) $X_t$ has almost surely continuous paths, (2) $X_t$ satisfies the strong Markov property, and (3) the random variable $(X_{t+t'}-X_t)$ is equal in distribution to $e^{-at}X_{t'}$, which is symmetric. Examining the proof of the Reflection Principle for Brownian motion in \cite{feller}, one sees that properties (1)--(3) are sufficient to prove \eqref{eq: xc}.

\subsection{Stopping times}
We let $\varepsilon>0$ be a given parameter. We define two stopping times.

\begin{itemize}
\item $\tau_+$ is equal to the first time $t \in(0,T)$ for which $q(t) = (1+\varepsilon)q_0$ if such a time exists and equal to $T$ if no such time exists.
\item  $\tau_-$ is equal to the first time $t \in (0,T)$ for which $q(t) = (1-\varepsilon)q_0$ if such a time exists and equal to $T$ if no such time exists.
\end{itemize}

We will make use of the following claim throughout this section.

\begin{claim}\label{clm: hit}
Let $a \ne 0$ and let $\eta \ge \eta_0$, where $\eta_0>1$ is a sufficiently large absolute constant. Define
\[
t_0 = \frac{\eta \varepsilon}{|a|}.
\]
Suppose that $t_0 < T$ and that $t_0 |a| < 1/10$. Then, provided $\varepsilon$ is sufficiently small depending on $T$, the following hold.
\begin{enumerate}[label=\emph{(\roman*)}]
\item If $a > 0$, then
\[
\eprob(\tau_+>t_0) \le C \exp(-c \varepsilon \eta q_0^2 a).
\]
\item If $a <0$, then
\[
\eprob(\tau_- > t_0) \le C \exp(-c \varepsilon \eta q_0^2 |a|).
\]
\end{enumerate}
\end{claim}
\begin{proof}
We will only prove (i); the proof of (ii) is nearly identical. Assume that $a>0$. By our assumption that $t_0 < T$, we have
\begin{equation}
\begin{split}
\prob(\tau_+ > t_0) =& \prob(q(t) < (1+\varepsilon)q_0 \;\text{for all}\; t \in [0,t_0])\\
\le & \prob(q(t_0) < (1+\varepsilon)q_0).
\end{split}
\end{equation}
Let $X_t$ be as in \eqref{eq: xt def}. If $q(t_0) < (1+\varepsilon)q_0$, then (since we assume $a t_0 = \eta \varepsilon<1/10$) we have
\begin{align*}
X_{t_0} &< e^{-a t_0}(1+\varepsilon)q_0 - q_0\\
& \le (1- c \eta \varepsilon)(1+\varepsilon)q_0 - q_0.\\
& \le - c' \eta\varepsilon q_0
\end{align*}
provided $\eta$ is larger than an absolute constant and $\varepsilon$ is smaller than an absolute constant. This shows that
\[
\prob(\tau_+ > t_0) \le \prob(X_{t_0} \le  -c' \eta \varepsilon q_0);
\]
applying \eqref{eq: xb} (and again using the assumption $t_0|a| < 1/10$) proves (i).
\end{proof}

\begin{lem}\label{lem: 2 side pro}
Let $c_0>0$ be a sufficiently small absolute constant. Define
\begin{align*}
t_\emx = c_0 \varepsilon^{1/2}, \qquad a_\etny = \varepsilon^{1/2}, \qquad a_\esml = \varepsilon^{1/4}.
\end{align*}
Then, provided $\varepsilon$ is sufficiently small depending on $T$, the following hold.
\begin{enumerate}[label=\emph{(\roman*)}]
\item For any $a \le a_\etny$,
\[
\eprob((\tau_+ < t_\emx) \;\emph{AND}\; (\tau_+ < \tau_-)) \le C \exp(-c_\varepsilon q_0^2(|a|+1)).
\]
\item For any $a \ge -a_\etny$, 
\[
\eprob((\tau_- < t_\emx)\;\emph{AND}\; (\tau_- < \tau_+)) \le  C \exp(-c_\varepsilon q_0^2 (|a|+1)).
\]
\item For any $|a| \ge  a_\esml$,
\[
\eprob((\tau_+ \ge t_\emx) \;\emph{AND}\; (\tau_- \ge t_\emx)) \le C \exp(-c_\varepsilon q_0^2 |a|).
\]
\item Define a stopping time $\tilde{\tau} = \min\{ \tau_+, \tau_-, t_\emx\}$. Then provided $q_0$ is sufficiently large depending on $\varepsilon$ and $T$ we have
\[
\emph{E}\bigg[ \int_0^{\tilde{\tau}} q^2(t)\ dt\bigg] < C \varepsilon^{1/2} \cJ_0(a;T,q_0)\;\text{for any}\; a \in \R.
\]
\end{enumerate}
\end{lem}
\begin{proof}
We note that by taking $\varepsilon$ sufficiently small depending on $T$ we can assume that $t_\mx <T$.

We first claim that
\begin{align}
&\prob(\tau_+ < t_\mx) \le C \exp(- c \varepsilon^{3/2} q_0^2)\;\text{for any}\; 0 \le a \le a_\tny,\label{eq: lqstr 1}\\
&\prob(\tau_- < t_\mx) \le C \exp(-c \varepsilon^{3/2}q_0^2)\;\text{for any}\; - a_\tny \le a \le 0.\label{eq: lqstr 2}
\end{align}
Assume that $0 \le a \le a_\tny$. Since $t_\mx <T$ we have
\[
\prob(\tau_+ < t_\mx) = \prob( \exists t \in (0,t_\mx) : q(t) \ge (1+\varepsilon)q_0).
\]
Note that we have $0 \le at \le c_0 \varepsilon$ and thus $e^{-at} \ge (1-c_0 \varepsilon)$ for any $t \in (0,t_\mx)$. Therefore, if $q(t) \ge (1+\varepsilon)q_0$ for any $t \in (0,t_\mx)$, then
\begin{align*}
X_t &= e^{-at}q(t) - q_0\\
&\ge (1-c_0\varepsilon)(1+\varepsilon)q_0-q_0\\
& > \varepsilon q_0/2
\end{align*}
provided $c_0$, $\varepsilon$ are smaller than certain absolute constants. We have shown that
\begin{align*}
\prob(\tau_+ < t_\mx)  \le \prob(\exists t \in (0,t_\mx) : X_t > \varepsilon q_0/2).
\end{align*}
Applying \eqref{eq: xc} and then \eqref{eq: xb} (we use that $t_\mx a \le c_0 \varepsilon <  1/10$) proves \eqref{eq: lqstr 1}. We omit the proof of \eqref{eq: lqstr 2} as it can be easily inferred from the proof of \eqref{eq: lqstr 1}.

We now claim that
\begin{align}
&\prob( \tau_- < t_\mx) \le C \exp(- c \varepsilon^{3/2} q_0^2)\;\text{for any} \; 0 \le a \le a_\sml,\label{eq: lqstr 3}\\
&\prob(\tau_+ < t_\mx) \le C \exp(-c \varepsilon^{3/2} q_0^2)\;\text{for any}\; - a_\sml \le a \le 0.\label{eq: lqstr 4}
\end{align}
Assume that $0 \le a \le a_\sml$. If $q(t) \le (1-\varepsilon)q_0$ for some $t \in (0,t_\mx)$, then 
\[
X_t \le e^{-at} (1-\varepsilon) q_0 - q_0 \le - \varepsilon q_0.
\]
This implies that
\[
\prob(\tau_- < t_\mx) \le \prob(\exists t \in (0,t_\mx) : X_t \le - \varepsilon q_0).
\]
Applying \eqref{eq: xc} and then \eqref{eq: xb} (note that $t_\mx a < c_0 \varepsilon^{3/4} < 1/10$ for any $0 \le a \le a_\sml$) proves \eqref{eq: lqstr 3}. The proof of \eqref{eq: lqstr 4} is very similar; we omit it.

Now assume that $a \ge a_\sml$. We claim that the following estimates hold:
\begin{align}
&\prob(\tau_+ \ge t_\mx)< C \exp(-c \varepsilon q_0^2 a), \label{eq: lqst 3}\\
&\prob((\tau_- < t_\mx) \;\text{AND}\; (\tau_- < \tau_+)) < C \exp(-c \varepsilon^2 q_0^2 a).\label{eq: lqst 4}
\end{align}
Let $\eta_0>0$ be the absolute constant in Claim \ref{clm: hit} and define
\[
\tilde{t}_0 = \frac{\eta_0 \varepsilon}{a}.
\]
By our assumption that $a \ge a_\sml$, we have 
\begin{equation}\label{eq: lqst 5}
\tilde{t}_0 \le \eta_0 \varepsilon^{3/4} < c_0 \varepsilon^{1/2} = t_\mx < T
\end{equation}
provided $\varepsilon$ is smaller than an absolute constant. Therefore
\begin{equation}\label{eq: lqst 6}
\prob(\tau_+ \ge t_\mx) \le \prob(\tau_+ \ge \tilde{t}_0).
\end{equation}
Provided $\varepsilon$ is smaller than an absolute constant we have $\tilde{t}_0 a = \eta_0 \varepsilon <1/10$. Together with Claim \ref{clm: hit},  \eqref{eq: lqst 6} implies \eqref{eq: lqst 3}.

Define
\[
t_0^* = \min\bigg\{ t_\mx, \frac{1}{10a}\bigg\}.
\]
Observe that $t_0^* < T$ and so we have
\[
\prob(\tau_- < t_0^*) = \prob(\exists t \in (0,t_0^*) : q(t) < (1-\varepsilon)q_0).
\]
If $q(t) < (1-\varepsilon)q_0$ for some $t \in (0,t_0^*)$, then (since $a > 0$)
\[
X_t < e^{-at}(1-\varepsilon)q_0-q_0 < - \varepsilon q_0.
\]
Therefore
\[
\prob(\tau_- < t_0^*) \le \prob(\exists t \in (0,t_0^*) : X_t < - \varepsilon q_0).
\]
Combining this with \eqref{eq: xc} gives
\begin{equation*}
\prob(\tau_- < t_0^*) \le  2\cdot \prob(X_{t_0^*} < - \varepsilon q_0)
\end{equation*}
Using \eqref{eq: xb} and the fact that $a t_0^* < 1/10$ gives
\begin{equation}\label{eq: lqst 6.5}
\prob(\tau_- < t_0^*)\le C \exp(-c \varepsilon^2 q_0^2/t_0^*).
\end{equation}

We claim that \eqref{eq: lqst 6.5} implies \eqref{eq: lqst 4} under the assumption that
\begin{equation}\label{eq: lqst 7}
a_\sml \le  a \le \frac{1}{10  c_0 \varepsilon^{1/2}}.
\end{equation}
To see this, note that \eqref{eq: lqst 7} implies that $t_0^* = t_\mx$, and therefore \eqref{eq: lqst 6.5} gives
\begin{equation}
\prob(\tau_- < t_\mx) \le C \exp(-c \varepsilon^{3/2} q_0^2).
\end{equation}
We then use \eqref{eq: lqst 7} to deduce that
\[
\prob(\tau_- < t_\mx) \le C \exp(- c \varepsilon^2 q_0^2 a).
\]
This proves \eqref{eq: lqst 4} assuming \eqref{eq: lqst 7}.

We now prove \eqref{eq: lqst 4} assuming that
\begin{equation}\label{eq: lqst 8}
a > \frac{1}{10 c_0 \varepsilon^{1/2}}.
\end{equation}
Note that this assumption implies that $t_0^* < t_\mx$. We have
\begin{equation}\label{eq: lqst 9}
\begin{split}
\prob((\tau_- < t_\mx) \;\text{AND}\; (\tau_- < \tau_+))  \le \prob(\tau_- < t_0^*)+ \prob(\tau_+ > t_0^*).
\end{split}
\end{equation}
By \eqref{eq: lqst 6.5} (and using that assumption \eqref{eq: lqst 8} implies that $t_0^* = (10a)^{-1}$), we have
\begin{equation}\label{eq: lqst 10}
\prob(\tau_- < t_0^*) \le C \exp(-c \varepsilon^2 q_0^2 a ).
\end{equation}
Observe that $t_0^* < T$ and that $t_0^* a = 1/10$. We can therefore apply Claim \ref{clm: hit} with $\eta = 1/(10\varepsilon)$ to get
\begin{equation}\label{eq: lqst 11}
\prob(\tau_+ > t_0^*) \le C \exp(-c q_0^2 a).
\end{equation}
Combining \eqref{eq: lqst 9}--\eqref{eq: lqst 11} proves \eqref{eq: lqst 4} under assumption \eqref{eq: lqst 8}. This completes the proof of \eqref{eq: lqst 4}.

By mirroring the proofs of \eqref{eq: lqst 3} and \eqref{eq: lqst 4}, we can show that for $a \le -a_\sml$ we have
\begin{align}
& \prob(\tau_- \ge t_\mx) < C \exp(-c \varepsilon q_0^2 |a|),\label{eq: lqstr 7}\\
& \prob((\tau_+ < t_\mx) \;\text{AND}\; (\tau_+ < \tau_-)) < C \exp(-c \varepsilon^2 q_0^2 |a|).\label{eq: lqstr 8}
\end{align}

Combining \eqref{eq: lqstr 1}, \eqref{eq: lqstr 4}, \eqref{eq: lqstr 8} proves (i). Similarly, combining \eqref{eq: lqstr 2}, \eqref{eq: lqstr 3}, \eqref{eq: lqst 4} proves (ii). Equations \eqref{eq: lqst 3} and \eqref{eq: lqstr 7} prove (iii).

We now prove (iv). Note that with probability 1 we have that $\tilde{\tau} \le t_\mx = c_0 \varepsilon^{1/2}$ and that $|q(t)| \le (1+\varepsilon)q_0$ for all $t \in (0,\tilde{\tau})$. Therefore
\begin{equation}\label{eq: lqst 12}
\E\bigg[ \int_0^{\tilde{\tau}} q^2(t)\ dt \bigg] < C q_0^2 \varepsilon^{1/2}.
\end{equation}

By Lemma \ref{lem: OS asymp}, there exists $\tilde{A}>0$ depending only on $T$ such that
\begin{equation}\label{eq: lqst 13}
\cJ_0(a;T,q_0) > \frac{q_0^2}{100|a|}\;\text{for any}\; a < - \tilde{A}
\end{equation}
and
\begin{equation}\label{eq: lqst 14}
\cJ_0(a;T,q_0) > c_T q_0^2 \;\text{for any}\; a > - \tilde{A}.
\end{equation}
Combining \eqref{eq: lqst 12} and \eqref{eq: lqst 14} gives
\begin{equation}\label{eq: lqst 15}
\E\bigg[ \int_0^{\tilde{\tau}} q^2(t)\ dt\bigg] < C_T \varepsilon^{1/2}\cJ_0(a;T,q_0)\;\text{for any}\; a > - \tilde{A}.
\end{equation}
This proves (iv) when $a > - \tilde{A}$.

Now assume that $a < - \tilde{A}$. Let $\eta_0$ be as in Claim \ref{clm: hit} and define
\[
\bar{t}_0 = \frac{\eta_0\varepsilon}{|a|}.
\]
By taking $\varepsilon$ smaller than an absolute constant we can ensure that $\bar{t}_0 < t_\mx$ and that $\bar{t}_0 |a| < 1/10$. Therefore we can apply Claim \ref{clm: hit} to get
\[
\prob(\tilde{\tau} > \bar{t}_0) \le \prob(\tau_- > \bar{t}_0) \le C \exp(-c \varepsilon q_0^2|a|).
\]
Recall that with probability 1 we have $|q(t)|\le (1+\varepsilon)q_0$ for all $t \in [0, \tilde{\tau}]$. Therefore
\begin{align*}
\E\bigg[ \int_0^{\tilde{\tau}} q^2(t)\ dt \bigg] &< C q_0^2 \bar{t}_0 + C q_0^2 t_\mx \prob(\tilde{\tau} > \bar{t}_0)\\
& < \frac{C q_0^2 \varepsilon}{|a|} + C q_0^2 t_\mx \exp(-c \varepsilon q_0^2 |a|).
\end{align*}
Taking $q_0$ sufficiently large depending on $\varepsilon$ and using \eqref{eq: lqst 13} gives
\begin{equation}\label{eq: lqst 16}
\E\bigg[ \int_0^{\tilde{\tau}} q^2(t)\ dt \bigg]  < C\varepsilon \cJ_0(a,T,q_0)\;\text{for any}\; a < - \tilde{A}.
\end{equation}
Combining \eqref{eq: lqst 15} and \eqref{eq: lqst 16} proves (iv).
\end{proof}

We now let $q_0^*>q_0$ be a real number and we define two additional stopping times.

\begin{itemize}
\item $\tau^*_+$ is equal to the first time $t \in(0,T)$ for which $q(t) = q_0^*$ if such a time exists and equal to $T$ if no such time exists.
\item $\tau^*_-$ is equal to the first time $t \in (0,T)$ for which $q(t) = -q_0^*$ if such a time exists and equal to $T$ if no such time exists.
\end{itemize}

\begin{lem}\label{lem: 1 side pro}
Provided $\varepsilon$ is sufficiently small depending on $T$, the following hold. 
\begin{enumerate}[label=\emph{(\roman*)}]
\item For any $a \ge -\varepsilon^{-1/2}$
\[
\eprob( \tau_-^*< \tau_+) \le C_{q_0} \varepsilon^{1/4}.
\]
\item Define a stopping time
\[
\bar{\tau} = \min \{ \tau_+, \tau_-^*\}.
\]
Then
\[
\emph{E} \bigg[ \int_0^{\bar{\tau}} q^2(t)\ dt \bigg] < C_{T,q_0,q_0^*} \varepsilon^{1/4} \;\text{for any}\; a \in \R.
\]
\end{enumerate}
\end{lem}
\begin{proof}
We first prove a claim. Define
\[
\tilde{t}_1 = \varepsilon \min \{1, 1/|a|\}
\]
(when $a=0$ we set $\tilde{t}_1 = \varepsilon$). We claim that
\begin{equation}\label{eq: tau est 2}
\prob(\tau_+ > \tilde{t}_1) \le C q_0 \varepsilon^{1/4}\;\text{for any}\; |a| \le \varepsilon^{-1/2}.
\end{equation}

Assume that $|a| \le \varepsilon^{-1/2}$. By taking $\varepsilon$ sufficiently small depending on $T$ we can assume that $\tilde{t}_1 <T$. We have
\begin{equation*}
\prob(\tau_+ > \tilde{t}_1) = \prob(q(t) < (1+\varepsilon)q_0\;\text{for all}\; t \in [0,\tilde{t}_1]).
\end{equation*}
Note that $|a| \tilde{t}_1 <  \varepsilon$, so if $q(t) < (1+\varepsilon)q_0$ for any $t \in [0,\tilde{t}_1]$, then 
\begin{equation}
\begin{split}
X_{t} &< e^{-a t} (1+\varepsilon)q_0 - q_0\\
& < C \varepsilon q_0.
\end{split}
\end{equation}
Therefore
\begin{equation}
\begin{split}
\prob(\tau_+>\tilde{t}_1) &\le \prob( X_t < C \varepsilon q_0 \;\text{for all}\; t \in [0,\tilde{t}_1])\\
& = 1 - \prob(\exists t \in [0,\tilde{t}_1] : X_t > C \varepsilon q_0)\\
&= 1 - 2 \cdot \prob(X_{\tilde{t}_1} > C \varepsilon q_0),
\end{split}
\end{equation}
where the last equality follows from \eqref{eq: xc}. Recall (see \eqref{eq: xa}) that $X_{\tilde{t}_1}$ is a normal random variable with mean 0 and variance $(1- e^{-2a\tilde{t}_1})/2a$. Also recall that $|a| \tilde{t}_1 <  \varepsilon$, so provided $\varepsilon$ is smaller than some absolute constant the variance of $X_{\tilde{t}_1}$ is bounded above and below by constant multiples of $\tilde{t}_1$. Therefore
\begin{equation}
\begin{split}
\prob(X_{\tilde{t}_1} > C \varepsilon q_0) &\ge \frac{1}{2} - C \int_0^{C \varepsilon q_0/\tilde{t}_1^{1/2}} e^{-t^2/2}\ dt\\
& \ge \frac{1}{2} - \frac{C \varepsilon q_0}{\tilde{t}_1^{1/2}},
\end{split}
\end{equation}
so
\[
\prob(\tau_+ > \tilde{t}_1) \le \frac{C\varepsilon q_0}{\tilde{t}_1^{1/2}}.
\]
Using the assumption $|a| \le \varepsilon^{-1/2}$ and the definition of $\tilde{t}_1$ implies \eqref{eq: tau est 2}.

We now prove (i). Note that for any $t^* \in (0,T)$ we have
\begin{equation}\label{eq: tau est 4}
\prob( \tau_-^*< \tau_+) \le \prob(\tau_-^* < t^*) + \prob(\tau_+>t^*).
\end{equation}

In the event that $q(t) = - q_0^*$ for some $ t\in (0,T)$, then
\[
X_t = - e^{-at}q_0^*-q_0 < - q_0.
\]
Therefore
\begin{align*}
\prob(\tau_-^* < t^*) &= \prob(\exists t \in (0,t^*) : q(t) = - q_0^*)\\
&\le \prob(\exists t \in (0, t^*) : X_t < - q_0).
\end{align*}
Equation \eqref{eq: xc} implies
\begin{equation}\label{eq: tau est 3}
\prob(\tau_-^* < t^*) \le 2 \cdot \prob(X_{t^*}<-q_0).
\end{equation}

We let $\tilde{t}_1$ be as above; note that by definition $\tilde{t}_1 |a| < \varepsilon$ and $1/\tilde{t}_1 > 1 /\varepsilon$. Combining \eqref{eq: tau est 3} and \eqref{eq: xb} gives
\begin{equation*}
\prob(\tau_-^* < \tilde{t}_1) \le C \exp(-c q_0^2 /\varepsilon)\;\text{for}\; |a| \le \varepsilon^{-1/2}.
\end{equation*}
Combining this with \eqref{eq: tau est 4} and \eqref{eq: tau est 2} proves (i) for $|a| \le \varepsilon^{-1/2}$.

Assume that $a \ge \varepsilon^{-1/2}$. Define
\[
\bar{t}_1 = \frac{1}{10a}.
\]
Let $\eta_0$ be as in Claim \ref{clm: hit}. Taking $\varepsilon$ sufficiently small depending on $T$ ensures that $\bar{t}_1 < T$ and that $(10\varepsilon)^{-1}>\eta_0$. We can therefore apply Claim \ref{clm: hit} with $\eta = (10\varepsilon)^{-1}$ to get
\begin{equation}\label{eq: tau est 5}
\prob(\tau_+ > \bar{t}_1) \le C \exp(-c q_0^2 a).
\end{equation}
Taking $\varepsilon$ smaller than an absolute constant gives $1/\bar{t}_1>1$. Combined with \eqref{eq: tau est 3}, \eqref{eq: xb} this gives
\[
\prob(\tau_-^*< \bar{t}_1) \le C \exp(-c q_0^2 a).
\]
By \eqref{eq: tau est 4}, we have therefore shown that
\[
\prob(\tau_-^*<\tau_+) \le C \exp(-cq_0^2 a)\;\text{for any}\; a > \varepsilon^{-1/2}.
\]
This implies (i) for $a > \varepsilon^{-1/2}$ (here we use the assumption that $q_0$ is nonzero).

We now prove (ii). We split the analysis into cases.

\noindent\underline{Case I}: (Large negative $a$)

Assume that $a < - \varepsilon^{-1/2}$. Note that the system \eqref{eq: stopping 1} satisfies
\[
q(t) = q_0 + W(t) + \int_0^t a q(\tau)\ d\tau,
\]
with probability 1. Therefore the expected value of the integral of $q^2$ from time 0 to time $T$ is equal to the expected cost incurred from time $0$ to time $T$ by the constant gain strategy $\cg(0)$, i.e., the strategy that sets $u=0$ (see Section \ref{sec: known a}). Applying Corollary \ref{cor: CG score} with $\alpha = 0$, taking $\varepsilon$ sufficiently small depending on $T$, and using that $\bar{\tau} \le T$ with probability 1,  we have
\[
\E\bigg[\int_0^{\bar{\tau}}q^2(t)\ dt\bigg] \le \frac{q_0^2+T}{2|a|}\;\text{for any}\; a< -\varepsilon^{-1/2},
\]
and therefore
\begin{equation}\label{eq: npro 1}
\E\bigg[\int_0^{\bar{\tau}}q^2(t)\ dt\bigg] < C_{T,q_0} \varepsilon^{1/2}\;\text{for any}\; a < -\varepsilon^{-1/2}.
\end{equation}

\noindent\underline{Case II}: (Large positive $a$)

Assume that $a > \varepsilon^{-1/2}$ and let $\bar{t}_1$ be as above. Since $\bar{\tau} < T$ and since $|q(t)| \le \min\{(1+\varepsilon)q_0,q_0^*\}$ for all $t \in (0, \bar{\tau})$ with probability 1, we have
\begin{equation}\label{eq: npro 8}
\E\bigg[\int_0^{\bar{\tau}}q^2(t)\ dt \bigg] \le C_{q_0,q_0^*}( \bar{t}_1 + T \cdot  \prob(\bar{\tau} > \bar{t}_1)).
\end{equation}
Note that our assumption that $a > \varepsilon^{-1/2}$ implies that $\bar{t}_1 < c\varepsilon^{1/2}$. The event $\bar{\tau} > \bar{t}_1$ implies the event $\tau_+ > \bar{t}_1$, and therefore we combine \eqref{eq: npro 8} with \eqref{eq: tau est 5} to get
\begin{equation}\label{eq: tau est 7}
\E\bigg[\int_0^{\bar{\tau}}q^2(t)\ dt \bigg] < C_{T,q_0,q_0^*} \varepsilon^{1/2}\;\text{for any}\; a > \varepsilon^{-1/2}
\end{equation}
(note that we also use the assumption that $q_0$ is nonzero).

\noindent\underline{Case III}: (Bounded $a$)
Let $\tilde{t}_1$ be as above. Then
\[
\E\bigg[\int_0^{\bar{\tau}}q^2(t)\ dt \bigg] < C_{q_0,q_0^*}( \tilde{t}_1 + T\cdot  \prob(\bar{\tau} > \tilde{t}_1))\; \text{for any}\;  |a| \le \varepsilon^{-1/2}.
\]
The event $\bar{\tau}>\tilde{t}_1$ implies the event $\tau_+>\tilde{t}_1$. Therefore, by \eqref{eq: tau est 2}, we have
\begin{equation}\label{eq: pro 16.5}
\E\bigg[\int_0^{\bar{\tau}}q^2(t)\ dt \bigg] < C_{T,q_0,q_0^*} \varepsilon^{1/4}\; \text{for any}\;  |a| \le \varepsilon^{-1/2}.
\end{equation}

Combining \eqref{eq: npro 1}, \eqref{eq: tau est 7}, and \eqref{eq: pro 16.5} proves (ii).
\end{proof}

\begin{lem}\label{lem: decay}
Assume $a < 0$ and define $\rho = q_0^*/q_0$ (recall that $q_0^*>q_0$ and so $\rho >1$). Then
\[
\emph{Prob}((\tau_+^*<T) \; \emph{OR}\; (\tau_-^* < T)) \le C_{T, \rho}(1+q_0^{-2})  \exp(-c_\rho q_0^2 |a|).
\]
\end{lem}
\begin{proof}
Let $N$ be the smallest positive integer such that
\[
10 |a|T \le N.
\]
Note that
\begin{equation}\label{eq: decay lem 1}
N < (1 + 10  |a|T).
\end{equation}
Define $\Delta t = \frac{T}{N}$ (observe that our choice of $N$ ensures that $\Delta t |a| \le \frac{1}{10}$) and $I_j = [j\Delta t, (j+1)\Delta t]$ for $j = 0, \dots, N-1$. Note that 
\begin{equation}\label{eq: decay lem 2}
\prob((\tau_+^*<T) \; \text{OR}\; (\tau_-^* < T)) \le \sum_{j=0}^{N-1} \prob (\exists t\in I_j:|q(t)|\ge q_0^*).
\end{equation}
We claim that
\begin{equation}\label{eq: decay lem 3}
\prob(\exists t \in I_j: |q(t)| \ge q_0^* ) \le C \exp(-c_\rho q_0^2 |a|)\;\text{for any}\; j = 0,\dots, N-1.
\end{equation}
Combining \eqref{eq: decay lem 1}-\eqref{eq: decay lem 3} proves the lemma. Therefore, it just remains to establish \eqref{eq: decay lem 3}.

Since $a< 0$, we have $e^{-at} \ge e^{-aj\Delta t}$ for all $t \in I_j$. Therefore, in the event that $q(t) \ge q_0^*$,
\begin{equation}\label{eq: decay 1}
    X_t = q(t)e^{-at} - q_0 \ge q_0^* e^{-at}\left( 1 - \frac{1}{\rho}e^{at}\right) \ge c_\rho q_0^* e^{-a j \Delta t}
\end{equation}
(we've used that $\rho>1$ and $e^{at} \le 1$). In the event that $q(t) \le - q_0^*$,
\begin{equation}\label{eq: decay 2}
    X_t = q(t) e^{-at} - q_0 \le - q_0^* e^{-a j \Delta t}.
\end{equation}
Combining \eqref{eq: decay 1} and \eqref{eq: decay 2} implies that
\begin{equation}\label{eq: decay lem 4}
    \prob( \exists t \in I_j: |q(t)| \ge q_0^*) \le \prob( \exists t \in I_j : |X_t|\ge c_\rho q_0^* e^{-aj\Delta t}).
\end{equation}
Since $I_j \subset [0, (j+1)\Delta t]$, we have
\begin{multline}\label{eq: decay lem 5}
\prob( \exists t \in I_j : |X_t| \ge c_\rho q_0^* e^{-aj\Delta t}) \\
\le \prob( \exists t \in [0, (j+1) \Delta t] : |X_t| \ge c_\rho q_0^* e^{-aj \Delta t}).
\end{multline}
Equation \eqref{eq: xc} implies that
\begin{multline}\label{eq: decay lem 6}
    \prob( \exists t \in [0, (j+1) \Delta t] : |X_t| \ge c_\rho q_0^* e^{-aj \Delta t}) \\
    = 2\cdot \prob\left(|X_{(j+1)\Delta t}| \ge c_\rho q_0^* e^{-aj\Delta t}\right).
\end{multline}
Recall that $X_{(j+1)\Delta t}$ is a normal random variable with mean $0$ and standard deviation
\[
\left( \frac{e^{2|a|(j+1)\Delta t }-1}{2|a|}\right)^{1/2}
\]
Recall also that our choice of $N$ ensures that $\Delta t |a| < \frac{1}{10}$. Therefore $c < e^{-|a|\Delta t} < 1$, and we have
\[
\frac{c_\rho q_0^* e^{|a|j\Delta t } |a|^{1/2}}{(e^{2|a|(j+1)\Delta t}-1)^{1/2}} = \frac{c_\rho q_0^* e^{-|a| \Delta t}|a|^{1/2}}{(1 - e^{-2|a|(j+1)\Delta t})^{1/2}} \ge c_\rho' q_0 |a|^{1/2}.
\]
Thus, when $|X_{(j+1)\Delta t}| \ge c_\rho q_0^* e^{-aj\Delta t}$, the normal random variable $X_{(j+1)\Delta t}$ is at least $c_\rho' q_0 |a|^{1/2}$ standard deviations from its mean of 0. Therefore
\begin{equation}\label{eq: decay lem 7}
\prob\left( |X_{(j+1)\Delta t}| \ge c_\rho q_0^* e^{-aj \Delta t} \right) \le C \exp( - c_\rho' q_0^2 |a|). 
\end{equation}
Combining \eqref{eq: decay lem 4} - \eqref{eq: decay lem 7} proves \eqref{eq: decay lem 3}, finishing the proof of the lemma.
\end{proof}

\subsection{Estimating the parameter \texorpdfstring{$a$}{a}}\label{sec: abar}
We remain in the setting of the previous section.

In the event that $\tau_+<T$ we define a random variable
\[
\bar{a}_+ = \frac{\log(1+\varepsilon)}{\tau_+}.
\]
Note that $\bar{a}_+$ is an estimate for the parameter $a$ in \eqref{eq: stopping 1}. Similarly, in the event that $\tau_-<T$ we define a random variable
\[
\bar{a}_- = \frac{\log(1-\varepsilon)}{\tau_-}.
\]

Let $X \subset (0,\infty)$ be a subset. We will be interested in estimating the probability of events of the form ``$\tau_+<T$ and $\bar{a}_+ \in X$''. We will abbreviate such events to ``$\bar{a}_+ \in X$''; note that when $\tau_+ = T$ the random variable $\bar{a}_+$ is undefined and so the event $\bar{a}_+ \in X$ is meaningless. Similarly, for $Y \subset (-\infty,0)$ we let ``$\bar{a}_- \in Y$'' denote the event ``$\tau_- < T$ and $\bar{a}_- \in Y$''.

\begin{lem}\label{lem: Asharp}
Let $A^\#> 1$. Then provided $\varepsilon$ is sufficiently small depending on $T$ we have
\[
\eprob(\bar{a}_+ > A^\#) \le C \exp(-c_\varepsilon q_0^2 A^\#)
\]
for any $ |a| \le A^\# /10$.
\end{lem}
\begin{proof}
Define
\begin{equation}\label{eq: rv 1}
t_2 = \frac{\log(1+\varepsilon)}{A^\#}.
\end{equation}
By taking $\varepsilon$ sufficiently small depending on $T$ we ensure that $t_2 < T$. Therefore
\begin{equation}\label{eq: Asharp 1}
    \prob(\bar{a}_+ > A^\#) = \prob(\tau_+ < t_2).
\end{equation}

Suppose that $0 \le a \le A^\#/10$. Note that for any $t \in [0,t_2)$ we have
\[
ta < \frac{a}{A^\#}\log(1+\varepsilon) < \frac{\varepsilon}{10},
\]
and thus
\[
e^{-at}  \ge (1 - \varepsilon/10).
\]
Consequently, if $q(t) > (1+\varepsilon)q_0$ for any $t \in [0, t_2)$, then
\begin{align*}
X_t &= e^{-at}q(t) - q_0\\
& > e^{-at}(1+\varepsilon)q_0 - q_0\\
& \ge c \varepsilon q_0.
\end{align*}

Now suppose that $- A^\# / 10 \le a < 0$. If $q(t) > (1+\varepsilon)q_0$ for some $t \in [0, t_2)$, then
\[
X_t = e^{|a|t}q(t) - q_0 > \varepsilon q_0.
\]

We have shown that
\[
\prob(\tau_+ < t_2) \le \prob( \exists t \in [0, t_2) : X_t \ge c \varepsilon q_0)\;\text{for any}\; |a| \le \frac{A^\#}{10}.
\]
Combining this with \eqref{eq: Asharp 1} and \eqref{eq: xc}, we deduce that
\begin{equation}\label{eq: Asharp 1.5}
\begin{split}
\prob(\bar{a}_+ \ge A^\#) &\le \prob( \exists t \in [0, t_2) : X_t \ge c \varepsilon q_0)\\
& = 2 \cdot \prob( X_{t_2} \ge c \varepsilon q_0).
\end{split}
\end{equation}
Since $t_2|a| < C \varepsilon$, we apply \eqref{eq: xb} to get
\[
\prob(X_{t_2} \ge c \varepsilon q_0) \le C \exp( -c \varepsilon^2 q_0^2 / t_2) \le \exp( - c_\varepsilon q_0^2 A^\#).
\]
Combining this with \eqref{eq: Asharp 1.5} proves the lemma.
\end{proof}

\begin{lem}\label{lem: abar 1}
Let $\delta>0$. Then provided $\varepsilon$ is sufficiently small depending on $T$ the following hold.
\begin{enumerate}[label=\emph{(\roman*)}]
\item $\eprob(|a-\bar{a}_+|>\delta a) \le C \exp(-c_{\varepsilon, \delta} q_0^2 a)$ for any $a \ge \varepsilon^{1/2}$.
\item $\eprob(|a-\bar{a}_-|>\delta |a|) \le C \exp(-c_{\varepsilon, \delta} q_0^2 |a|)$ for any $a \le -\varepsilon^{1/2}$.
\end{enumerate}
\end{lem}
\begin{proof}
We will just prove (i), as essentially the same argument can be used to prove (ii).

Without loss of generality we assume that $\delta < 1/10$. Note that
\begin{equation}\label{eq: abar 1}
\prob(|a - \bar{a}_+| > \delta a) = \prob( \bar{a}_+ > (1+\delta)a)
+ \prob( 0 < \bar{a}_+ < (1-\delta)a).
\end{equation}
Now define
\[
t_3 = \frac{\log(1+\varepsilon)}{(1+\delta)a}.
\]
By taking $\varepsilon$ sufficiently small depending on $T$ we ensure that $t_3 < T$, and therefore
\begin{equation*}
\prob(\bar{a}_+ > (1+\delta) a)= \prob(\tau_+ < t_3).
\end{equation*}
For any $t \in [0, t_3)$, we have
\[
at \le a t_3 \le \frac{\log(1+\varepsilon)}{1+\delta}.
\]
Therefore, in the event that $q(t)>(1+\varepsilon)q_0$ for some $t \in [0, t_3)$, we have
\[
X_t = e^{-at}q(t) - q_0 \ge q_0\big((1+\varepsilon)^{1- 1/(1+\delta)} - 1\big) = c_{\varepsilon, \delta} q_0.
\]
This implies that
\[
\prob(\bar{a}_+ > (1+\delta) a) \le \prob(\exists t \in (0, t_3) : X_t > c_{\varepsilon, \delta} q_0).
\]
We use \eqref{eq: xc} to deduce that
\[
\prob(\bar{a}_+ > (1+\delta) a) \le 2\cdot \prob(X_{t_3} > c_{\varepsilon,\delta}q_0);
\]
we then note that $t_3 a < C \varepsilon^{1/2} < 1/10$ and apply \eqref{eq: xb} to get
\begin{equation}\label{eq: abar 4}
\prob(\bar{a}_+ > (1+\delta)a) \le C \exp(-c_{\varepsilon, \delta} q_0^2 a).
\end{equation}

Now define
\[
t_4 = \frac{\log(1+\varepsilon)}{(1-\delta)a}.
\]
By taking $\varepsilon$ sufficiently small depending on $T$ we ensure that $t_4 < T$, and therefore
\begin{equation*}
\prob(0 < \bar{a}_+ < (1-\delta)a) \le \prob( \tau_+ > t_4).
\end{equation*}

If $q(t_4) < (1+\varepsilon)q_0$, then
\begin{align*}
X_{t_4} &< e^{-a t_4} (1+\varepsilon)q_0 - q_0\\
& = q_0 ( (1+\varepsilon)^{-\delta/(1-\delta)}-1)\\
& < - c_{\varepsilon, \delta} q_0.
\end{align*}
Therefore
\[
\prob(0 < \bar{a}_+ < (1-\delta)a) \le \prob(X_{t_4} < - c_{\varepsilon,\delta} q_0) ;
\]
applying \eqref{eq: xb} gives
\begin{equation}\label{eq: abar 7}
\prob(0 < \bar{a}_+ < (1-\delta)a) \le  C \exp(-c_{\varepsilon, \delta} q_0^2 a).
\end{equation}

Combining \eqref{eq: abar 1}, \eqref{eq: abar 4}, \eqref{eq: abar 7} proves (i).
\end{proof}

\section{The almost optimal strategy}\label{sec: almost optimal strategy}

Throughout this section we fix a time horizon $T>0$ and a starting position $q_0 > 0$.

We fix constants $C_0, m_0$ as in the definition of an $A$-bounded strategy in Section \ref{sec: setup}. Recall that a strategy $\sigma$ (for time horizon $T$) is $A$-bounded for some $A>0$ if
\[
|u^\sigma(t,a)| \le C_0 A^{m_0}[|q^\sigma(t,a)| + 1]\;\text{for all} \; a \in \R, t\in [0,T].
\]
Throughout this section, we allow all constants and parameters to depend on $C_0,m_0$.  

In this section we prove the following theorem.

\begin{thm}\label{thm: main comp}
    Let $\varepsilon>0$. Then for $\varepsilon_0>0$ sufficiently small depending on $\varepsilon$ and $A>0$ sufficiently large depending on $\varepsilon, \varepsilon_0$, the following holds.

Let $\sigma$ be an $A$-bounded strategy for time horizon $T + \varepsilon_0$ and starting position $q_0$. Then the strategy $\sigma_*$ for time horizon $T$ and starting position $q_0$ specified in Section \ref{sec: def aos} satisfies the following.
\begin{enumerate}
    \item If $a \in [-A,A]$, then
    \[
    \cJ(\sigma_*,a;T,q_0) < \varepsilon + (1+\varepsilon) \cdot \sup\{ \cJ(\sigma,b;T+\varepsilon_0,q_0): |a-b| <  \varepsilon |a|\}.
    \]
    \item If $|a| > A$, then
    \[
    \cJ(\sigma_*,a;T,q_0) < \varepsilon + (1+\varepsilon)\cdot \cJ_0(a;T,q_0).
    \]
\end{enumerate}
\end{thm}

We now show that Theorem \ref{thm: main comp} implies Theorem \ref{thm: main simp}. Let $\varepsilon>0$. Let $\varepsilon_0$ be sufficiently small (depending on $\varepsilon$) and let $A$ be sufficiently large (depending on $\varepsilon$,$\varepsilon_0$) so that the conclusion of Theorem \ref{thm: main comp} holds. Note that we can assume that $\varepsilon_0 < \varepsilon$.

Let $\sigma$ be an $A$-bounded strategy for time horizon $T + \varepsilon$ and starting position $q_0$. Since $\varepsilon_0 < \varepsilon$, $\sigma$ is also an $A$-bounded strategy for time horizon $T + \varepsilon_0$ and starting position $q_0$. Moreover,
\begin{equation}\label{eq: mr 1}
\cJ(\sigma, a; T + \varepsilon_0, q_0) \le \cJ(\sigma, a; T + \varepsilon, q_0)\;\text{for any}\; a \in \R.
\end{equation}
Let $\sigma_*$ denote the strategy defined in Section \ref{sec: def aos}. Observe that Part (B) of Theorem \ref{thm: main simp} follows immediately from Part (B) of Theorem \ref{thm: main comp}. To prove Part (A) of Theorem \ref{thm: main simp} we simply combine equation \eqref{eq: mr 1} with Part (B) of Theorem \ref{thm: main comp}. This proves Theorem \ref{thm: main simp}.

The remainder of Section \ref{sec: almost optimal strategy} is devoted to proving Theorem \ref{thm: main comp}. For the remainder of Section \ref{sec: almost optimal strategy} we assume that we are given $\varepsilon>0$ smaller than some constant. We let $\varepsilon_0>0$ be sufficiently small depending on $\varepsilon$, and $A>0$ be sufficiently large depending on $\varepsilon,\varepsilon_0$.

In Section \ref{sec: def aos} we define the strategy $\sigma_*$. First, in Section \ref{sec: prelim aos}, we establish some preliminary results. In Section \ref{sec: proof aos} we prove Theorem \ref{thm: main comp}.

\subsection{Preliminaries}\label{sec: prelim aos}
\subsubsection{Averting disaster when \texorpdfstring{$a$}{a} is large}\label{sec: mbas}

Let $\sigma$ be an $A$-bounded strategy for time horizon $T+\varepsilon_0$ and starting position $q_0$. Because $\sigma$ is $A$-bounded, if $a \gg A$ then with high probability the control $u^\sigma$ will fail to control the particle and so the expected cost of $\sigma$ will be much larger than $\cJ_0(a;T,q_0)$ (see \cite{boundeda}). We now show how to use Theorem \ref{thm: lqs} to obtain a strategy $\tilde{\sigma}$ that performs almost as well as $\sigma$ when $|a| \le A$ and much better than $\sigma$ when $a \gg A$.

By Theorem \ref{thm: lqs},  there exists a number $q_0^*> \max\{1, 2  |q_0|\}$ depending on $\varepsilon, \varepsilon_0$, $T$, and $q_0$, but independent of $A$, such that for any $T' \in [\varepsilon_0, 2T]$ there exists a strategy $\lqs(T')$ for time horizon $T'$ and starting position $q_0^*$ satisfying 
\begin{equation}\label{eq: avert 1}
\cJ(\lqs(T'),a;T',q_0^*) < (1+\varepsilon) \cdot \cJ_0(a; T',q_0^*)\;\text{for any}\; a \in \R.
\end{equation}

We now define the strategy $\tilde{\sigma}$ for time horizon $T$ and starting position $q_0$.

We execute the strategy $\sigma$ from time 0 until time $t^*$, where $t^*$ is equal to the first time $t \in (0,T)$ for which $|q^\sigma(t)|=q_0^*$ if such a time exists and $t^*$ is equal to $T$ if no such time exists. Note that $t^*$ is a stopping time.

If $t^* = T$, then $\tilde{\sigma}$ exercises the same control variable as $\sigma$ at all times $t \in [0,T]$.

If $t^* < T$, then we execute the strategy $\lqs(T+\varepsilon_0- t^*)$ (from \eqref{eq: avert 1}) starting from time $t^*$ and position $|q^{\tilde{\sigma}}(t^*)|=q_0^*$. For any $s \in (0,T)$, we have $(T + \varepsilon_0 - s) > \varepsilon_0$. By taking $\varepsilon_0$ sufficiently small depending on $T$ we ensure that $(T + \varepsilon_0)  < 2 T$ and therefore that the strategy $\lqs(T + \varepsilon_0 - s)$ is well-defined for any $s \in (0,T)$. 

This concludes the definition of the strategy $\tilde{\sigma}$. We now estimate its expected cost.

If we never encounter $|q^{\tilde{\sigma}}| = q_0^*$, i.e., if $t^*  = T$, then the strategies $\sigma$ and $\tilde{\sigma}$ exercise the same control from time 0 to time $T$, hence they incur the same cost.

Suppose instead that we first encounter $|q^{\tilde{\sigma}}| = q_0^*$ at some time $t^* < T$. Until time $t^*$, the strategies $\sigma$, $\tilde{\sigma}$ exercise the same control and incur the same cost. Let $S$ denote the expected cost of $\sigma$ starting from position $q_0^*$ and time $t^*$ and continuing until time $T+\varepsilon_0$ and let $\tilde{S}$ denote the expected cost of $\tilde{\sigma}$ starting from position $q_0^*$ and time $t^*$ and continuing until time $T$. Note that
\begin{align*}
\tilde{S} &= \cJ(\lqs(T+\varepsilon_0-t^*),a;T-t^*,q_0)\\
&< \cJ(\lqs(T+\varepsilon_0-t^*),a;T+\varepsilon_0-t^*,q_0).
\end{align*}
Therefore, by \eqref{eq: avert 1}, 
\[
\tilde{S} < (1+\varepsilon)\cdot \cJ_0(a;T+\varepsilon_0-t^*,q_0).
\]
By Lemma \ref{lem: OS},
\[
\cJ_0(a;T+\varepsilon_0-t^*,q_0) \le S,
\]
and therefore $\tilde{S} \le(1+\varepsilon)\cdot S$.

As a consequence of the above discussion, we have
\begin{equation}\label{eq: mbas 1.3}
\cJ(\tilde{\sigma},a; T,q_0) < (1+\varepsilon)\cdot \cJ(\sigma,a;T+\varepsilon_0,q_0)\;\text{for any}\; a \in \R.
\end{equation}

We now derive another estimate on the expected cost of $\tilde{\sigma}$---this one is much sharper than \eqref{eq: mbas 1.3} when $|a| \gg A$.

From time $0$ until time $t^*$, we execute the strategy $\sigma$ and, by construction, have $|q^{\tilde{\sigma}}(t)| \le q_0^* $ for all $t \in [0,t^*]$. Since $\sigma$ is $A$-bounded, we therefore have
\[
|u^{\tilde{\sigma}}(t)| < C_0 A^{m_0}(|q(t)| + 1)\;\text{for any}\; t \in [0,t^*).
\]
\sloppy Therefore, from time $0$ until time $t^*$ we incur a cost of at most $C_T A^{2m_0} (q_0^*)^2$ with probability 1.

In the event that $t^*< T$, then from time $t^*$ until time $T$ we execute the strategy $\lqs(T+\varepsilon_0-t^*)$ (starting from position $\pm q_0^*$). Therefore, from time $t^*$ until time $T$ we incur an expected cost that is at most $(1+\varepsilon)$ times the expected cost of the optimal known-$a$ strategy $\sigma_\opt(a)$ for time horizon $(T+\varepsilon_0-t^*)$ and starting position $q_0^*$. By Corollary \ref{cor: opt alph score}, this is at most $C_T (q_0^*)^2 \max\{a,1\}$. We therefore have
\begin{equation*}
\cJ(\tilde{\sigma},a;T,q_0) < C_T (q_0^*)^2 \max\{A,a\}^{2m_0} \;\text{for any}\;a\in \R.
\end{equation*}
Since $q_0^*$ is determined by $\varepsilon$, $\varepsilon_0$, $T$, and $q_0$ we have
\begin{equation}\label{eq: mbas 3}
\cJ(\tilde{\sigma},a;T,q_0) < C_{T,q_0,\varepsilon,\varepsilon_0} \max\{A,a\}^{2m_0}\;\text{for any}\; a\in \R.
\end{equation}

\subsubsection{Rescaling strategies}\label{sec: rescaling}

Remark \ref{rmk: rescale} implies the following: Given $\lambda>1$ and a strategy $\sigma$ for time horizon $T$ and starting position $q_0$, we can define a strategy $\sigma_{\lambda}$ for time horizon $\lambda^2 T$ and starting position $\lambda q_0$ that satisfies 
\begin{equation}\label{eq: rs 1}
\cJ(\sigma_{\lambda},a;\lambda^2 T,\lambda q_0)\le \lambda^4\cdot  \cJ(\sigma, \lambda^2 a;T,q_0)\;\text{for any}\; a \in\R.
\end{equation}
Moreover, since $T < \lambda^2 T$, $\sigma_{\lambda}$ is also a strategy for time horizon $T$ and starting position $\lambda q_0$, and we have
\begin{equation}\label{eq: rs 2}
\cJ(\sigma_{\lambda},a; T,\lambda q_0)
\le   \cJ(\sigma_{\lambda},  a; \lambda^2 T,\lambda q_0)\;\text{for any}\; a \in\R.
\end{equation}
Combining \eqref{eq: rs 1}, \eqref{eq: rs 2}, we deduce that
\begin{equation}
\cJ(\sigma_{\lambda},a; T,\lambda q_0)
\le \lambda^4\cdot  \cJ(\sigma, \lambda^2 a;T,q_0)\;\text{for any}\; a \in\R.
\end{equation}

Let $\sigma$ be an $A$-bounded strategy for time horizon $T+\varepsilon_0$ and starting position $q_0$.

By the results of the previous section (specifically, see \eqref{eq: mbas 1.3} and \eqref{eq: mbas 3}),  there exists a strategy $\tilde{\sigma}$ for time horizon $T$ and starting position $q_0$ satisfying
\begin{equation}\label{eq: ing 18}
    \cJ(\tilde{\sigma},a; T,q_0) < (1+\varepsilon)\cdot \cJ(\sigma,a;T+\varepsilon_0,q_0)\;\text{for any}\; a \in \R
\end{equation}
and
\begin{equation}\label{eq: ing 19}
\cJ(\tilde{\sigma},a;T,q_0) < C_{T,q_0,\varepsilon,\varepsilon_0} \cdot \max\{A, a\}^{2m_0}\;\text{for any}\; a \in \R.
\end{equation}

By the discussion above, there then exists a strategy $\tilde{\sigma}_{1+\varepsilon_0}$ for time horizon $T$ and starting position $(1+\varepsilon_0) q_0$ satisfying
\begin{multline}\label{eq: ing 20}
\cJ(\tilde{\sigma}_{1+\varepsilon_0},a;T,(1+\varepsilon_0)q_0) 
< (1+C \varepsilon)  \cdot \cJ(\sigma, (1+\varepsilon_0)^2 a; T+\varepsilon_0, q_0)\\ \text{for any}\; a \in \R
\end{multline}
(recall that $\varepsilon_0$ is sufficiently small depending on $\varepsilon$; in particular, to deduce \eqref{eq: ing 20} we assume that $\varepsilon_0<\varepsilon$) and
\begin{equation}\label{eq: ing 21}
\cJ(\tilde{\sigma}_{1+\varepsilon_0},a;T,(1+\varepsilon_0)q_0) < C_{T,q_0,\varepsilon,\varepsilon_0} \cdot \max\{A,a\}^{2m_0}\; \text{for any} \; a \in \R .
\end{equation}
From equation \eqref{eq: ing 20}, we deduce that
\begin{multline}\label{eq: ing 22}
\cJ(\tilde{\sigma}_{1+\varepsilon_0},a;T,(1+\varepsilon_0)q_0) 
\\ < (1+C \varepsilon)  \cdot \sup\{\cJ(\sigma, b; T+\varepsilon_0, q_0): |a-b| < \varepsilon |a|\}\; \text{for any}\; a \in \R
\end{multline}
(again taking $\varepsilon_0$ sufficiently small depending on $\varepsilon$).

\subsection{Definition of the almost optimal strategy}\label{sec: def aos}

Let $\sigma$ be an $A$-bounded strategy for time horizon $T+\varepsilon_0$ and starting position $q_0$.

To prove the Main Theorem, we must exhibit a strategy $\sigma_*$ for time horizon $T$ and starting position $q_0$ satisfying
\begin{multline}\label{eq: pmt 2}
\cJ(\sigma_*,a;T,q_0) < \varepsilon + (1+C\varepsilon)\cdot \sup \{ \cJ(\sigma, b;T+\varepsilon_0,q_0): |b-a| < \varepsilon |a|\}\\ \text{for any}\; a \in [-A,A]
\end{multline}
and
\begin{equation}\label{eq: pmt 2.5}
    \cJ(\sigma_*,a;T,q_0) < \varepsilon + (1+C\varepsilon)\cdot \cJ_0(a;T,q_0)\;\text{for any}\; |a| > A.
\end{equation}
The definition of $\sigma_*$ requires a bit of setup.

By Theorem \ref{thm: las}, provided $A$ is sufficiently large depending on $\varepsilon,T,q_0$ there exists a strategy $\las$ for time horizon $T$ and starting position $(1+\varepsilon_0)q_0$ such that
\begin{equation}\label{eq: las pre 1}
\begin{split}
\cJ(\las, a; T, (&1+\varepsilon_0)q_0) \\
&< (1+\varepsilon) \cdot \cJ_0(a;T,(1+\varepsilon_0)q_0)\;\text{for any}\; a > \frac{1}{100}A,
\end{split}
\end{equation}
and
\begin{equation}\label{eq: las 2}
\cJ(\las, a; T, (1+\varepsilon_0)q_0) < C_{T,q_0} A^2 \;\text{for any}\; a<\frac{1}{100}A.
\end{equation}
By Lemma \ref{lem: OS asymp}, provided $\varepsilon_0$ is sufficiently small depending on $\varepsilon,T,q_0$ we have
\begin{equation}\label{eq: m1}
| \cJ_0(a; T, (1+\varepsilon_0)q_0) - \cJ_0(a;T,q_0)| < \varepsilon \cdot \cJ_0(a;T,q_0)\;\text{for any}\;a\in\R.
\end{equation}
 Combining \eqref{eq: las pre 1} and \eqref{eq: m1} gives
\begin{equation}\label{eq: las 1}
\begin{split}
\cJ(\las, a; T, (1+\varepsilon_0&)q_0)
\\
&< (1+C\varepsilon) \cdot \cJ_0(a;T,q_0)\;\text{for any}\; a > \frac{1}{100}A.
\end{split}
\end{equation}

By Section \ref{sec: rescaling} (see equations \eqref{eq: ing 21}, \eqref{eq: ing 22}), provided $\varepsilon_0$ is sufficiently small depending on $\varepsilon,T,q_0$ there exists a strategy $\bar{\sigma}$ for time horizon $T$ and starting position $(1+\varepsilon_0)q_0$ such that
\begin{multline}\label{eq: bas 1}
\cJ(\bar{\sigma},a;T,(1+\varepsilon_0)q_0) 
\\< (1+C \varepsilon)  \cdot \sup\{\cJ(\sigma, b; T+\varepsilon_0, q_0): |a-b| < \varepsilon |a|\}\; \text{for any}\; a \in \R
\end{multline}
and
\begin{equation}\label{eq: bas 2}
\cJ(\bar{\sigma},a;T,(1+\varepsilon_0)q_0) < C_{T,q_0,\varepsilon, \varepsilon_0} \max\{A,a\}^{2m_0}\;\text{for any}\;a \in \R.
\end{equation}

Last, we introduce a constant $q_\rare>q_0$ depending on $T,q_0$ and a strategy $\br$ (``BR'' for bounded regret) for time horizon $T$ and starting position $-q_\rare$ satisfying
\[
\cJ(\br, a; T, - q_\rare) < C \cdot \cJ_0(a;T,-q_\rare)\;\text{for any}\; a \in \R.
\]
This is possible due to Theorem \ref{thm: lqs}; see, in particular, the discussion of bounded regret strategies following the statement of Theorem \ref{thm: lqs}. We then observe, by Remark \ref{cor: q_0s}, that
\[
\cJ_0(a; T, -q_\rare) < C_{T,q_0} \cdot \cJ_0(a; T, q_0)\;\text{for any} \; a \in \R.
\]
This implies that the strategy $\br$ satisfies
\begin{equation}\label{eq: br}
\cJ(\br,a;T,-q_\rare) < C_{T,q_0} \cdot \cJ_0(a; T, q_0)\;\text{for any} \; a \in \R.
\end{equation}

We fix the strategies $\las,\bar{\sigma}, \br$ for the remainder of Section \ref{sec: almost optimal strategy}. We are now ready to define the strategy $\sigma_*$.

The strategy $\sigma_*$ consists of two epochs: the Prologue and the Main Act. Only the Prologue is guaranteed to occur. In fact, for large negative $a$ we do not expect the Main Act to occur.

For the remainder of Section \ref{sec: almost optimal strategy} we write $q,u$ to denote, respectively, the particle trajectories $q^{\sigma_*}$ and the control variables $u^{\sigma_*}$.

\underline{\Prologue}: During the Prologue we set $u=0$. We define a stopping time $\tau_1$ by setting $\tau_1$ equal to the first time $t \in (0,T)$ for which either $q(t) = (1+\varepsilon_0)q_0$ or $q(t) = - q_\rare$ if such a time exists and setting $\tau_1$ to be equal to $T$ if no such time exists. The Prologue lasts from time $0$ until time $\tau_1$. If $\tau_1 = T$ then at time $\tau_1$ the game ends along with the Prologue. If $\tau_1 < T$, then at time $\tau_1$ we enter the Main Act.

\underline{\Main}: Suppose that we enter the Main Act at some time $\tau_1 \in (0,T)$. We either have $q(\tau_1) = (1+\varepsilon_0)q_0$ or $q(\tau_1) = -q_\rare$; we proceed differently in each case. If $q(\tau_1) = (1+\varepsilon_0)q_0$, then we define
\[
\bar{a} = \frac{\log(1+\varepsilon_0)}{\tau_1}.
\]
The value of $\bar{a}$ then determines our strategy during the Main Act. Note that in the event that $\bar{a}$ is defined (i.e., $\tau_1<T$ and $q(\tau_1) = (1+\varepsilon_0)q_0$) we have $\bar{a}>0$ with probability 1. If $\bar{a}\ge A/10$, then we play strategy $\las$ beginning at time $\tau_1$ and position $q(\tau_1) = (1+\varepsilon_0)q_0$ until the game ends at time $T$. If $\bar{a}<A/10$, then we play strategy $\bar{\sigma}$ beginning at time $\tau_1$ and position $q(\tau_1) = (1+\varepsilon_0)q_0$ until the game ends at time $T$. This completes the description of our strategy during the Main Act in the event that $q(\tau_1) = (1+\varepsilon_0)q_0$. If $q(\tau_1) = - q_\rare$, then we play strategy $\br$ beginning at time $\tau_1$ and position $q(\tau_1) = - q_\rare$ until the end of the game at time $T$.

This concludes the definition of the strategy $\sigma_*$. 

\subsection{Proof of Theorem \ref{thm: main comp}}\label{sec: proof aos}

Recall that the Prologue ends at time $\tau_1$ and that during the Prologue we set $u=0$. We define a random variable $\cost_P$ to be the cost incurred during the Prologue for a given realization of the noise, i.e., 
\[
\cost_P= \int_0^{\tau_1}q^2(t)\ dt.
\]
Note that $\cost_P$ depends on $a$.

Taking $\varepsilon_0$ sufficiently small depending on $T$, we see from Lemma \ref{lem: 1 side pro} that
\begin{equation}\label{eq: npro main}
\E[\cost_P] < C_{T,q_0}  \varepsilon_0^{1/4}\;\text{for any} \; a \in \R
\end{equation}
(recall that $q_\rare$ is determined by $T,q_0$) and that
\begin{equation}\label{eq: main 0.4}
\prob(q(\tau_1) = - q_\rare) \le C_{q_0} \varepsilon_0^{1/4} \;\text{for any} \;  a \ge -{\varepsilon_0^{-1/2}}.
\end{equation}

Recall that, if it occurs, the Main Act begins at time $\tau_1 <T$, in which case we have either $q(\tau_1) = (1+\varepsilon_0)q_0$ or $q(\tau_1) = -q_\rare$.  In the event that the Main Act occurs and $q(\tau_1) = (1+\varepsilon_0)q_0$, we define the random variable
\[
\bar{a} = \frac{\log(1+\varepsilon_0)}{\tau_1};
\]
the value of $\bar{a}$ then determines our strategy during the Main Act. We will consider events of the form ``$\bar{a} \in X$'' for $X \subset \R$. This is shorthand for the event ``$\tau_1 < T, q(\tau_1) = (1+\varepsilon_0)q_0$, and $\bar{a} \in X$'' (recall that $\bar{a}$ is not defined unless $\tau_1 < T$ and $q(\tau_1) = (1+\varepsilon_0)q_0$).

Let $\alpha > 1$. We claim that, provided $\varepsilon_0$ is sufficiently small depending on $T$, we have the following estimates:
\begin{flalign}
&\text{If}\; a > 10\alpha\; \text{then}\; \prob(\bar{a} < \alpha) < C \exp(-c_{\varepsilon_0} q_0^2 a).&\label{eq: p1}\\
&\text{If}\; |a| < \alpha/10, \;\text{then}\; \prob(\bar{a}>\alpha) < C \exp(-c_{\varepsilon_0} q_0^2 \alpha).\label{eq: p2}\\
&\text{If}\; a < 0, \;\text{then}\; \prob(\bar{a} > 0) < C_{T,q_0,\varepsilon_0} \exp(-c_{\varepsilon_0} q_0^2 |a|).\label{eq: p3}
\end{flalign}
These estimates follow, respectively, from Lemmas \ref{lem: abar 1}, \ref{lem: Asharp}, and \ref{lem: decay} in Section \ref{sec: stopping}.

We define a random variable $\cost_M$ to be the cost incurred during the Main Act for a given realization of the noise, i.e.,
\[
\cost_M = \int_{\tau_1}^{T}( (q(t))^2 + (u(t))^2) \ dt.
\]
We use the remainder of this section to prove the estimates
\begin{multline}\label{eq: main main}
\E[\cost_M] <C\varepsilon +  (1+C\varepsilon)\cdot \sup\{\cJ(\sigma,b;T+\varepsilon_0,q_0): |b-a|< \varepsilon |a|\}\\ \text{for any}\; |a| \le A
\end{multline}
and
\begin{equation}\label{eq: 4.27.5}
    \E[\cost_M] < C\varepsilon + (1+C\varepsilon)\cdot \cJ_0(a;T,q_0)\;\text{for any}\; |a|>A.
\end{equation}
Combining \eqref{eq: main main}, \eqref{eq: 4.27.5} with the fact that
\[
\cJ(\sigma_*,a;T,q_0) = \E[\cost_P]+\E[\cost_M]
\]
and the estimate \eqref{eq: npro main} above, and taking $\varepsilon_0$ sufficiently small depending on $\varepsilon,T,q_0$, we deduce \eqref{eq: pmt 2} and \eqref{eq: pmt 2.5}, which in turn imply the conclusion of the Main Theorem.

Note that
\begin{equation}\label{eq: costm decomp}
\E[\cost_M] = \E[\cost_M \cdot \mathbbm{1}_{\bar{a} > 0}] + \E[\cost_M \cdot \mathbbm{1}_{q(\tau_1) = - q_\rare}]
\end{equation}

In the event that we enter the Main Act at some time $\tau_1 \in (0,T)$ for which $q(\tau_1) = - q_\rare$, we execute the strategy $\br$ starting from time $\tau_1$ and position $-q_\rare$. Combining \eqref{eq: br} and \eqref{eq: main 0.4} gives
\begin{equation}\label{eq: main 2.5}
\E[\cost_M \cdot \mathbbm{1}_{q(\tau_1) = - q_\rare}] < C_{T,q_0} \cdot \varepsilon_0^{1/4} \cdot \cJ_0(a;T,q_0)\;\text{for any} \; a \ge - {\varepsilon_0^{-1/2}}.
\end{equation}
Provided $\varepsilon_0$ is sufficiently small (depending on $T$), Lemma \ref{lem: OS asymp} implies that
\[
\cJ_0(a;T,q_0) < C_{T,q_0} \varepsilon_0^{1/2}\;\text{for any}\; a < - {\varepsilon_0^{-1/2}};
\]
combining this with \eqref{eq: br} gives
\begin{equation}\label{eq: main 2.5.5}
\E[\cost_M \cdot \mathbbm{1}_{q(\tau_1) = - q_\rare}]< C_{T,q_0} \varepsilon_0^{1/2}\;\text{for any}\; a < - {\varepsilon_0^{-1/2}}.
\end{equation}
We combine \eqref{eq: main 2.5}, \eqref{eq: main 2.5.5}, and take $\varepsilon_0$ sufficiently small depending on $\varepsilon,T,q_0$ to get
\begin{equation}\label{eq: main 2.5.6}
\begin{split}
\E[\cost_M \cdot \mathbbm{1}_{q(\tau_1) = - q_\rare}]< \varepsilon + \varepsilon \cdot \cJ_0(a;T,q_0) \;\text{for any}\; a \in \R.
\end{split}
\end{equation}

Throughout the remainder of this section we will make use of the fact that
\begin{equation}\label{eq: 4.31.25}
\cJ_0(a;T,q_0) < \cJ(\sigma,a;T+\varepsilon_0,q_0)\;\text{for any}\; a \in \R;
\end{equation}
this is a consequence of Lemma \ref{lem: OS}. We note that \eqref{eq: main 2.5.6} and \eqref{eq: 4.31.25} imply that
\begin{multline}\label{eq: 4.31.5}
\E[\cost_M \cdot \mathbbm{1}_{q(\tau_1) = - q_\rare}]\\ < \varepsilon + \varepsilon \cdot \sup\{ \cJ(\sigma,b;T+\varepsilon_0,q_0): |b-a|<\varepsilon|a|\}\;\text{for any}\; a \in \R.
\end{multline}

It remains to control the expected value of $\cost_M$ when $\bar{a}>0$. We will consider 4 cases.

\noindent \underline{Case I: $a$ is large, positive}

Suppose that $ a > A$. In the event that $\bar{a} \ge A/10$, we execute the strategy $\las$ during the Main Act. Therefore, by \eqref{eq: las 1}, we have
\begin{equation}\label{eq: main 1}
\begin{split}
\E[\cost_M \cdot \mathbbm{1}_{\bar{a}\ge A/10}] &< \cJ(\las,a;T,(1+\varepsilon_0)q_0)\\ &< (1+C\varepsilon) \cdot \cJ_0(a;T,q_0).
\end{split}
\end{equation}

In the event that $0 < \bar{a} < A/10$, we execute strategy $\bar{\sigma}$ during the Main Act. Therefore, by \eqref{eq: bas 2} and \eqref{eq: p1}, we have
\begin{equation}\label{eq: main 2}
\begin{split}
\E[\cost_M \cdot \mathbbm{1}_{0 < \bar{a}<A/10}] &< \cJ(\bar{\sigma},a;T,(1+\varepsilon_0)q_0)\cdot \prob(\bar{a} < A/10) \\
&< C_{T,q_0,\varepsilon, \varepsilon_0} a^{2m_0}\exp(-c_{\varepsilon_0}q_0^2 a).
\end{split}
\end{equation}

Combining \eqref{eq: main 1} and \eqref{eq: main 2} and taking $A$ sufficiently large depending on $T,q_0,\varepsilon, \varepsilon_0$ gives
\begin{equation}\label{eq: main 3}
\E[\cost_M \cdot \mathbbm{1}_{\bar{a}>0}] < (1+C\varepsilon) \cdot \cJ_0(a;T,q_0) + \varepsilon\;\text{for any}\; a > A.
\end{equation}

\noindent\underline{Case II: $a$ is positive, medium-sized}

Suppose that $ \frac{1}{100}A < a < A$.

If the event $\bar{a} \ge  A/10$ occurs, then during the Main Act we execute the strategy $\las$. By \eqref{eq: las 1}, we then have
\begin{equation}\label{eq: main 5}
\begin{split}
\E[\cost_M \cdot \mathbbm{1}_{\bar{a}>A/10}] &< \cJ(\las,a;T,(1+\varepsilon_0)q_0) \cdot \prob(\bar{a}\ge A/10) \\
&< (1+C\varepsilon) \cdot \cJ_0(a;T,q_0) \cdot \prob(\bar{a}\ge A/10).
\end{split}
\end{equation}

In the event that $0<\bar{a}<A/10$, we execute the strategy $\bar{\sigma}$ during the Main Act. By \eqref{eq: bas 1}, we then have
\begin{equation}\label{eq: main 6}
\begin{split}
\E[\cost_M \cdot &\mathbbm{1}_{0<\bar{a}<A/10}] \\
&< \cJ(\bar{\sigma},a;T,(1+\varepsilon_0)q_0) \cdot \prob(\bar{a}<A/10)\\
& < (1+C \varepsilon)  \cdot \sup\{\cJ(\sigma, b; T+\varepsilon_0, q_0): |a-b| < \varepsilon |a|\}\\
& \qquad \cdot \prob(\bar{a}<A/10).
\end{split}
\end{equation}
Combining \eqref{eq: main 5}, \eqref{eq: main 6}, and \eqref{eq: 4.31.25} we get
\begin{multline}\label{eq: main 7}
    \E[\cost_M\cdot \mathbbm{1}_{\bar{a}>0}] < (1+C \varepsilon)  \cdot \sup\{\cJ(\sigma, b; T+\varepsilon_0, q_0): |a-b| < \varepsilon |a|\}
    \\ \text{for any}\; \frac{1}{100}A < a <  A.
\end{multline}

\noindent \underline{Case III: $a$ is bounded}

Suppose that $|a| < \frac{1}{100} A$.

In the event that $0 < \bar{a} < A/10$, we execute the strategy $\bar{\sigma}$ during the Main Act. By \eqref{eq: bas 1}, we have
\begin{equation}\label{eq: main 8}
\begin{split}
    \E[\cost_M \cdot &\mathbbm{1}_{0 <\bar{a}<A/10}]\\
    &< \cJ(\bar{\sigma},a;T,(1+\varepsilon_0)q_0) \\
    &< (1+C \varepsilon)  \cdot \sup\{\cJ(\sigma, b; T+\varepsilon_0, q_0): |a-b| < \varepsilon |a|\}.
\end{split}
\end{equation}

In the event that $\bar{a}\ge A/10$, we execute the strategy $\las$ during the Main Act. By \eqref{eq: las 2} and \eqref{eq: p2}, we therefore have
\begin{equation}\label{eq: main 9}
    \E[\cost_M \cdot \mathbbm{1}_{\bar{a}>A/10}] < C_{T,q_0} A^2 \exp(-c_{\varepsilon_0} q_0^2 A).
\end{equation}

Combining \eqref{eq: main 8} and \eqref{eq: main 9}, and taking $A$ sufficiently large depending on $T,q_0,\varepsilon$, $\varepsilon_0$ gives
\begin{multline}\label{eq: main 10}
\E[\cost_M\cdot \mathbbm{1}_{\bar{a}>0}]<\varepsilon + (1+C \varepsilon)  \cdot \sup\{\cJ(\sigma, b; T+\varepsilon_0, q_0): |a-b| < \varepsilon |a|\}\\ \text{for any}\; |a|< \frac{A}{100}.
\end{multline}

\noindent \underline{Case IV: $a$ is large, negative}

Suppose that $a < - \frac{1}{100} A$.

In the event that $\bar{a}>0$ we execute either the strategy $\las$ or the strategy $\bar{\sigma}$ during the Main Act. Combining \eqref{eq: las 2}, \eqref{eq: bas 2}, and \eqref{eq: p3} gives
\[
\E[\cost_M\cdot \mathbbm{1}_{\bar{a}>0}] < C_{T,q_0,\varepsilon, \varepsilon_0} |a|^{2m_0} \exp(-c_{\varepsilon_0} q_0^2 |a|).
\]
Taking $A$ sufficiently large depending on $T,q_0,\varepsilon, \varepsilon_0$ then gives
\begin{equation}\label{eq: main 11}
\E[\cost_M\cdot \mathbbm{1}_{\bar{a}>0}] < \varepsilon\;\text{for any}\; a < - \frac{A}{100}.
\end{equation}

Combining \eqref{eq: main 3}, \eqref{eq: main 7}, \eqref{eq: main 10}, and \eqref{eq: main 11} with \eqref{eq: costm decomp}, \eqref{eq: main 2.5.6}, and \eqref{eq: 4.31.5} proves \eqref{eq: main main} and \eqref{eq: 4.27.5}.

\subsection{Proof of Corollary \ref{cor: q0}}\label{sec: q0}
Let $\varepsilon>0$ be given and let $A$ be sufficiently large so that the conclusion of Theorem \ref{thm: intro 2} holds. Let $\sigma$ be an $A$-bounded strategy for time horizon $T+\varepsilon$ and starting position $\varepsilon$.

Then the strategy $\sigma_*$ (as in Theorem \ref{thm: intro 2}) for time horizon $T$ and starting position $\varepsilon$ satisfies
\begin{multline}\label{eq: q0 1}
    \cJ(\sigma_*,a;T,\varepsilon) < \varepsilon + (1+\varepsilon) \cdot \sup\{ \cJ(\sigma, b; T+\varepsilon, \varepsilon): |a-b| < \varepsilon|a|\}\\ \text{for any}\; a \in [-A,A]
\end{multline}
and
\begin{equation}\label{eq: q0 2}
    \cJ(\sigma_*,a;T,\varepsilon) < \varepsilon + (1+\varepsilon) \cdot \cJ_0(a;T,\varepsilon)\;\text{for any}\; |a| > A.
\end{equation}
Combining \eqref{eq: q0 2} with Lemma \ref{lem: OS asymp} (specifically, Parts 1 and 2 of Lemma \ref{lem: OS asymp}), we deduce that, provided $A$ is sufficiently large depending on $\varepsilon$ and $T$, we have
\begin{equation}\label{eq: q0 3}
\cJ(\sigma_*,a;T,\varepsilon) < \varepsilon + (1+C\varepsilon) \cdot \cJ_0(a;T,0)\;\text{for any}\; |a| > A.
\end{equation}
By Remark \ref{rmk: symmetry}, $\sigma_*$ gives rise to a strategy for time horizon $T$ and starting position $-\varepsilon$; we also denote this strategy by $\sigma_*$ and recall that
\[
\cJ(\sigma_*,a;T,\varepsilon) = \cJ(\sigma_*,a;T,-\varepsilon)
\]

We now construct the strategy $\hat{\sigma}_*$ for time horizon $T$ and starting position $0$. We set $u^{\hat{\sigma}_*}(t) = 0$ from time 0 until time $\hat{t}$, where $\hat{t}$ is equal to the first time $t \in (0,T)$ for which $|q^{\hat{\sigma}_*}(t)| = \varepsilon$ if such a time exists and $\hat{t}$ is equal to $T$ if no such time exists.

If $\hat{t} =T$, then we set $u^{\hat{\sigma}_*} \equiv 0$ until the game ends at time $T$. In this case, we incur a cost of at most $T\varepsilon$ with probability 1.

If $\hat{t} < T$, then we execute the strategy $\sigma_*$ starting from time $\hat{t}$ and position $|q(\hat{t})|=\varepsilon$ until the game ends at time $T$. In this case, our expected cost from time 0 until time $T$ is at most $T\varepsilon + \cJ(\sigma_*,a;T,\varepsilon)$.

We have defined the strategy $\hat{\sigma}^*$ and shown that
\[
\cJ(\hat{\sigma}^*,a;T,0) < T \varepsilon + \cJ(\sigma_*,a;T,\varepsilon).
\]
Combining this with \eqref{eq: q0 1} and \eqref{eq: q0 3} implies the conclusion of the corollary.

\section{The Large \texorpdfstring{$q$}{q} Strategy}\label{sec: lqs}

In this section we prove Theorem \ref{thm: lqs}.

Fix a time horizon $T>0$. Let $\varepsilon>0$ be given. Note that it suffices to prove Theorem \ref{thm: lqs} under the assumption that $\varepsilon$ is sufficiently small depending on $T$.

We let $t_{\mx}\in(0,T)$ denote a real number that we will choose in Section \ref{sec: t max} depending on $\varepsilon$.

Let $A_1 \ge 1$ be a sufficiently large real number depending on $\varepsilon,T$.

We let $q_\bg \ge 1$ be a sufficiently large real number depending on $\varepsilon,A_1,T$. Note that since $A_1$ is determined by $\varepsilon$ and $T$, $q_\bg$ depends only on $\varepsilon$ and $T$. We fix a starting position $q_0 \ge q_\bg$.

For $\alpha \in \R$, recall that $\sigma_\opt(\alpha)$ denotes the known-$a$ strategy for time horizon $T$ (recall that this strategy is independent of the starting position; see Section \ref{sec: setup}). The corresponding control variable is
\[
u^{\sigma_\opt(\alpha)}(t,a) = - \kappa(T-t,\alpha) q^{\sigma_\opt(\alpha)}(t,a)
\]
for any $a \in \R$, $t\in [0,T]$, where
\[
\kappa(T-t,\alpha) = \frac{\tanh((T-t)\sqrt{\alpha^2+1})}{\sqrt{\alpha^2+1} - \alpha \tanh((T-t)\sqrt{\alpha^2+1})}.
\]
We will occasionally refer to the known-$a$ strategy for an arbitrary time horizon $T' \in (0,T]$; we denote this strategy by $\sigma_\opt(\alpha;T')$ for $\alpha \in \R$. We note that $\kappa(T'-t,\alpha)$ is the gain function corresponding to $\sigma_\opt(\alpha;T')$ and we remark that
\begin{equation}\label{eq: kappa 1}
    0 \le \kappa(T'-t,\alpha) \le C \cdot \max\{\alpha,1\}\; \text{for all} \; t \in [0,T'], \alpha \in \R
\end{equation}
with the constant $C$ independent of $T'$.

We are now ready to define the strategy $\lqs$. For the remainder of Section \ref{sec: lqs} we write $q,u$ to denote, respectively, the particle trajectories $q^{\lqs}$ and the control variables $u^{\lqs}$.

We note that the strategy $\lqs$ consists of two epochs: the Prologue and the Main Act. Both epochs are guaranteed to occur.

\underline{\Prologue}: We define a stopping time $\tau_M$ by setting $\tau_M$ equal to the first time $t \in (0, t_{\mx})$ for which $q(t) \notin ((1-\varepsilon)q_0, (1+\varepsilon)q_0)$ if such a time exists, and $\tau_M = t_{\mx}$ if no such time exists.

The Prologue lasts from time $0$ to time $\tau_M$. During the Prologue we exercise no control, i.e. we set $u \equiv 0$. At time $\tau_M$, we enter the Main Act.

\underline{\Main}: The Main Act lasts from time $\tau_M$ until the end of the game at time $T$. Our strategy during the Main Act depends on what happens during the Prologue. We define some events that capture the possible outcomes of the Prologue:
\begin{align*}
&\pro_+ = \{ q(\tau_M) = (1+\varepsilon)q_0\},\\
&\pro_- = \{ q(\tau_M) = (1-\varepsilon)q_0\},\\
& \pro_{\mx} = \{ \tau_M = t_\mx\}.
\end{align*}
The events $\pro_+$, $\pro_-$, $\pro_\mx$ partition our probability space.

If the event $\pro_+$ occurs with $\tau_M = s$ for some $s \in (0,t_\mx)$, then we define
\[
\bar{a} = \frac{\log(1+\varepsilon)}{s}.
\]
If the event $\pro_-$ occurs with $\tau_M = s$ for some $ s \in (0,t_\mx)$, then we define
\[
\bar{a} = \frac{\log(1-\varepsilon)}{s}.
\]
If $\pro_+$ (respectively $\pro_-$) occurs, then we have $\bar{a}>0$ (resp. $\bar{a} < 0$) with probability 1.

If the event $\pro_\mx$ occurs, then we leave $\bar{a}$ undefined.

For $X \subset (0,\infty)$ we will sometimes write ``$\bar{a} \in X$'' as shorthand for the event ``$\bar{a} \in X$ and $\pro_+$''. Similarly, for $Y \subset (-\infty, 0)$ we write ``$\bar{a}\in Y$'' for the event ``$\bar{a} \in Y$ and $\pro_-$''.

We now specify our strategy during the Main Act.

\noindent\underline{Case I}: (We believe $a$ is small)

If the event $\pro_\mx$ occurs, then we set
\[
u(t) = - \kappa(T - t,0)\cdot q(t)
\]
during the Main Act. In this case, the expected cost incurred during the Main Act is equal to the expected cost of the known-$a$ strategy $\sigma_\opt(0;T-t_\mx)$ for time horizon $(T-t_\mx)$ and starting position $q(t_\mx)$.

\noindent \underline{Case II}: (We believe $a$ is large, positive)

If the event $\pro_+$ occurs and $\bar{a} \ge A_1$, then we set
\[
u(t) = - 2 \bar{a} q(t)
\]
during the Main Act. Note that this is the control variable corresponding to the constant gain strategy $\cg(2\bar{a})$ (see Section \ref{sec: known a}) with starting position $(1+\varepsilon)q_0$.

\noindent \underline{Case III}: (We believe that $a$ is bounded, positive)

If the event $\pro_+$ occurs and $0 < \bar{a} < A_1$ and $\tau_M = s$ for some $s \in (0,t_\mx)$, then we set
\[
u(t) = - \kappa(T - t, \bar{a})\cdot  q(t)
\]
during the Main Act. In this case, the expected cost incurred during the Main Act is equal to the expected cost of the optimal known-$a$ strategy $\sigma_\opt(\bar{a};T-s)$ for time horizon $T - s$ and starting position $(1+\varepsilon)q_0$.

\noindent \underline{Case IV}: (We believe that $a$ is large, negative)

If the event $\pro_-$ occurs and $\bar{a} \le -A_1$, then we set 
\[
u(t) = 0
\]
during the Main Act. Note that this is the control variable corresponding to the constant gain strategy $\cg(0)$ with starting position $(1-\varepsilon)q_0$.

\noindent \underline{Case V}: (We believe that $a$ is bounded, negative)

If the event $\pro_-$ occurs and $-A_1 < \bar{a} < 0$ and $\tau_M = s$ for some $s \in (0,t_\mx)$, then we set
\[
u(t) = - \kappa(T - s,\bar{a})\cdot q(t)
\]
during the Main Act. In this case, the expected cost incurred during the Main Act is equal to the expected cost of the optimal known-$a$ strategy $\sigma_\opt(\bar{a};T-s)$ for time horizon $T-s$ and starting position $(1-\varepsilon)q_0$.

This concludes the definition of the strategy $\lqs$.

\subsection{The parameter \texorpdfstring{$t_\mx$}{t max}}\label{sec: t max}

We will now choose the parameter $t_\mx$. We set
\[
t_\mx = c_0 \varepsilon^{1/2},
\]
where $c_0$ is a sufficiently small absolute constant. We define
\[
a_\tny = \varepsilon^{1/2}, \qquad a_\sml = \varepsilon^{1/4}.
\]
Provided $\varepsilon$ is sufficiently small depending on $T$, we apply Lemma \ref{lem: 2 side pro} to deduce the following:
\begin{align}
&\prob(\pro_+) \le  C \exp(-c_\varepsilon q_0^2 (|a|+1))\;\text{for any}\; a \le a_\tny,\label{eq: tmax 2}\\
&\prob(\pro_-) \le C \exp(-c_\varepsilon q_0^2 (|a|+1))\;\text{for any}\; a \ge - a_\tny,\label{eq: tmax 3}\\
&\prob(\pro_\mx) \le  C \exp(-c_\varepsilon q_0^2|a|)\;\text{for any} \; |a| \ge a_\sml.\label{eq: tmax 1}
\end{align}

\subsection{Bounding the expected cost}\label{sec: main}

Define random variables
\begin{align*}
&\cost_P = \int_0^{\tau_M}q^2(t) \ dt,\\
& \cost_M = \int_{\tau_M}^{T} \left( (q(t))^2 + (u(t))^2)\right) dt.
\end{align*}
$\cost_P$ and $\cost_M$ are, respectively, the costs incurred during the Prologue and Main Act for a given realization of the noise. Note that both random variables depend on $a$.

Provided $q_\bg$ is sufficiently large depending on $\varepsilon$ and $T$, Lemma \ref{lem: 2 side pro} implies that
\begin{equation}\label{eq: tmax 4}
\E[\cost_P] < C \varepsilon^{1/2} \cJ_0(a;T,q_0)\;\text{for any}\; a \in \R.
\end{equation}

Our goal in the remainder of this section is to control the expected value of $\cost_M$. Recall that we play a different strategy during the Main Act depending on which one of 5 events occur during the Prologue. We analyze each of these cases separately in Sections \ref{sec: main 1}--\ref{sec: main 5}. In Section \ref{sec: main thm} we prove Theorem \ref{thm: lqs}.

\subsubsection{\texorpdfstring{$\cost_M$}{Cost M} when we believe that \texorpdfstring{$a$}{a} is small}\label{sec: main 1}
Recall that if the event $\pro_\mx$ occurs, then the Main Act begins at time $t_\mx$ and we set
\[
u(t) = - \kappa(T-t, 0)\cdot q(t)
\]
during the Main Act. In this case our expected cost during the Main Act is equal to the expected cost of the optimal known-$a$ strategy $\sigma_\opt(0;T-t_\mx)$ for time horizon $T-t_\mx$ starting from position $q(t_\mx) = \tilde{q}_0$ for some $\tilde{q}_0 \in ((1-\varepsilon)q_0,(1+\varepsilon)q_0)$. By Corollary \ref{cor: opt alph score}, we thus have
\begin{equation}\label{eq: main 1.2}
\E[\cost_M | \pro_\mx] \le C_T q_0^2 e^{2|a|T} \;\text{for any}\;a\in\R.
\end{equation}
Combining \eqref{eq: tmax 1} and \eqref{eq: main 1.2} gives
\[
\E[\cost_M \cdot \mathbbm{1}_{\pro_\mx}] < C_T q_0^2 \cdot \exp(2|a|T - c_\varepsilon q_0^2 |a|)\; \text{for any}\; |a| \ge a_\sml.
\]
Taking $q_\bg$ sufficiently large depending on $\varepsilon,T$, we get
\begin{equation}\label{eq: 1.5.4}
\E[\cost_M \cdot \mathbbm{1}_{\pro_\mx}] < \varepsilon\cdot  \min\{1,1/|a|\} \;\text{for any}\; |a| \ge a_\sml.
\end{equation}
By Lemma \ref{lem: OS asymp}, there exists a constant $\tilde{A}>1$ depending only on $T$ such that
\begin{align}
    &\cJ_0(a;T,q_0) > \frac{c_T}{|a|}\;\text{for any}\; a < - \tilde{A},\label{eq: 1.5.1}\\
    & \cJ_0(a;T,q_0) > c_{T} \;\text{for any} \; a > \tilde{A}.\label{eq: 1.5.2}
\end{align}
Again by Lemma \ref{lem: OS asymp},
\begin{equation}\label{eq: 1.5.3}
    \cJ_0(a;T,q_0) \ge \int_0^T \kappa(t,a) \ dt > c_{T}\;\text{for any}\; |a| \le \tilde{A}.
\end{equation}
Combining \eqref{eq: 1.5.4}--\eqref{eq: 1.5.3} gives
\begin{equation}\label{eq: main 1.11}
    \E[\cost_M\cdot \mathbbm{1}_{\pro_\mx}] \le \varepsilon C_{T} \cJ_0(a;T,q_0) \;\text{for any}\; |a| \ge a_\sml.
\end{equation}

By Lemma \ref{lem: lin ctrl formula}, we have
\[
\cJ(\sigma_\opt(\alpha),a;T,q_0) = h(\alpha,a)q_0^2 + j(\alpha,a)\;\text{for any}\; \alpha \in \R, a\in \R
\]
where $h,j: \R^2 \rightarrow (0,\infty)$ are smooth functions depending on $T$ but not on $q_0$. This implies that
\[
\E[\cost_M | \pro_\mx] \le h(0,a) (1+\varepsilon)^2 q_0^2 + j(0,a)
\]
and
\begin{equation}\label{eq: main 1.8}
\cJ_0(a;T,q_0) = h(a,a)q_0^2 + j(a,a).
\end{equation}
Therefore
\begin{equation*}
\begin{split}
    \E[\cost_M | \pro_\mx]  \le& \cJ_0(a;T,q_0) + |j(a,a) - j(0,a)| + h(a,a)C\varepsilon q_0^2\\
    & + C q_0^2 |h(a,a) - h(0,a)|.
\end{split}
\end{equation*}
Since $h, j$ are smooth functions (depending on $T$), we have
\begin{multline}\label{eq: main 1.9}
    \E[\cost_M | \pro_\mx]\\ \le \cJ_0(a;T,q_0) + C_T q_0^2 |a| + h(a,a) C \varepsilon q_0^2\;\text{for any}\; |a| \le a_\sml.
\end{multline}
We combine this with \eqref{eq: main 1.8} to get
\[
\E[\cost_M | \pro_\mx] < (1+C \varepsilon) \cdot \cJ_0(a;T,q_0) + C_{T} q_0^2 \varepsilon^{1/4}\;\text{for any}\; |a| \le a_\sml.
\]
Equation \eqref{eq: main 1.8} also implies that
\[
\cJ_0(a;T,q_0) \ge c_{T} q_0^2 \;\text{for any}\; |a| \le a_\sml.
\]
We have therefore shown that
\begin{multline}\label{eq: main 1.10}
    \E[\cost_M \cdot \mathbbm{1}_{\pro_\mx}] \\
    \le \cJ_0(a;T,q_0)\cdot (1+C_T \varepsilon^{1/4})\cdot \prob(\pro_\mx)\;\text{for any}\; |a| \le a_\sml.
\end{multline}

\subsubsection{\texorpdfstring{$\cost_M$}{Cost M} when we believe that \texorpdfstring{$a$}{a} is large, positive}\label{sec: main 2}
Let $\alpha \ge 0$ be a positive real number. Recall that $\cg(\alpha)$ denotes the simple feedback strategy with constant gain function $\alpha$ (see Section \ref{sec: known a}).

Suppose that the event $\pro_+$ occurs and that we have $\bar{a} = \tilde{a}$ for some $\tilde{a}\ge A_1$. In this case we set $u = - 2\tilde{a}q$ during the Main Act. We therefore have
\begin{equation}\label{eq: lar pos 1}
\E[\cost_M | \bar{a} = \tilde{a}] \le \cJ(\cg(2\tilde{a}),a;T,(1+\varepsilon)q_0)\;\text{for any}\; \tilde{a}\ge A_1, a \in \R.
\end{equation}
Define
\[
A^* = \max\{|a|, A_1\}.
\]
Combining \eqref{eq: lar pos 1} with Corollary \ref{cor: CG score} implies the following estimates:
\begin{multline}\label{eq: big a n4}
    \E[\cost_M | 10^n A^* < \bar{a} < 10^{n+1} A^*] \\ < C_T 10^{2n} (A^*)^2 q_0^2\;\text{for any}\; a \in \R, n\ge 0,
\end{multline}
and
\begin{equation}\label{eq: big a n3}
    \E[ \cost_M | A_1 \le \bar{a} < 10 |a| ] < C_T |a|^2 q_0^2\;\text{for any}\;a < - A_1/10.
\end{equation}
Now let $0 < \delta \ll 1$ be a sufficiently small parameter depending on $\varepsilon$. Using Corollary \ref{cor: CG score} again, we get
\begin{multline}\label{eq: big a n2}
\E[\cost_M | (A_1 \le \bar{a} < 10 a) \;\text{AND}\; (|a - \bar{a}|>\delta a)] \\ < C_T a^2 e^{2aT}q_0^2\; \text{for any}\; a \ge A_1/10,
\end{multline}
and
\begin{multline}\label{eq: big a n1}
    \E[\cost_M | (\bar{a} \ge A_1) \;\text{AND}\; (|\bar{a} - a | < \delta a)] \\ < (1+4a^2(1+\delta)^2)\frac{q_0^2+T}{2a(1-C\delta)}\;\text{for any}\; a \ge A_1/10.
\end{multline}

We assume for now that $a > A_1/10$. Taking $A_1$ sufficiently large and $\delta$ sufficiently small (both depending on $\varepsilon$) in \eqref{eq: big a n1} gives
\begin{equation*}
\E[\cost_M | (\bar{a} \ge A_1) \;\text{AND}\; (|\bar{a} - a | < \delta a)] < 2a (q_0^2+T)(1+\varepsilon).
\end{equation*}
Combining this with Lemma \ref{lem: OS asymp} gives
\[
\E[\cost_M | (\bar{a} \ge A_1 )\;\text{AND}\;(|\bar{a} - a| < \delta a)] < (1+C\varepsilon) \cJ_0(a,T,q_0)
\]
provided $A_1$ is sufficiently large depending on $\varepsilon$ and $T$. This implies
\begin{multline}\label{eq: main 2.3}
    \E[\cost_M \cdot \mathbbm{1}_{\bar{a} \ge A_1} \cdot \mathbbm{1}_{|a - \bar{a}|<\delta a} ] \\
    < (1+C\varepsilon)\cdot \cJ_0(a;T,q_0)\cdot \prob(\bar{a} \ge A_1)\;\text{for any}\; a > A_1/10.
\end{multline}

We continue to assume that $a > A_1/10$. Taking $\varepsilon$ sufficiently small depending on $T$, we apply Lemma \ref{lem: abar 1} to get
\begin{equation}\label{eq: big a n5}
    \prob(|a - \bar{a} | > \delta a) \le C \exp(-c_{\varepsilon, \delta} q_0^2 a).
\end{equation}
Combining \eqref{eq: big a n5} and \eqref{eq: big a n2} gives
\[
\E[\cost_M \cdot \mathbbm{1}_{ A_1 \le \bar{a} < 10 a} \cdot \mathbbm{1}_{|\bar{a} - a|> \delta a}] \le C_T a^2 q_0^2 \exp(2a(T - c_{\varepsilon, \delta} q_0^2)).
\]
Taking $q_\bg$ sufficiently large depending on $\varepsilon,T$ (recall that $\delta$ is determined by $\varepsilon$) we get
\begin{equation}\label{eq: main 2.6}
    \E[\cost_M \cdot \mathbbm{1}_{ A_1 \le \bar{a} < 10 a} \cdot \mathbbm{1}_{|\bar{a} - a|> \delta a}] < \varepsilon\;\text{for any} \; a \ge A_1/10.
\end{equation}

We continue to assume that $a > A_1/10$. Provided $\varepsilon$ is sufficiently small depending on $T$, Lemma \ref{lem: Asharp} gives
\[
\prob( \bar{a} > 10^n a) \le C \exp(- c_\varepsilon q_0^2 10^n a)\;\text{for any} \; n \ge 1.
\]
Combining this with \eqref{eq: big a n4} gives
\begin{multline*}
\E[\cost_M \cdot\mathbbm{1}_{10^{n+1}a > \bar{a} > 10^n a} \cdot \mathbbm{1}_{\bar{a} \ge A_1}] \\ \le C_T 10^{2n} a^2 q_0^2 \exp(-c_\varepsilon q_0^2 10^n a)\;\text{for any}\; n \ge 1.
\end{multline*}
Taking $q_\bg$ to be sufficiently large depending on $\varepsilon$ and $T$ and summing over $n$ gives
\begin{equation}\label{eq: main 2.9}
    \E[\cost_M \cdot\mathbbm{1}_{\bar{a} > 10 a  } \cdot \mathbbm{1}_{\bar{a} \ge A_1}] <  \varepsilon\;\text{for any}\; a > A_1/10.
\end{equation}
Combining \eqref{eq: main 2.3}, \eqref{eq: main 2.6}, \eqref{eq: main 2.9}, and \eqref{eq: 1.5.2}, we get
\begin{multline}\label{eq: main 2.10}
\E[\cost_M \cdot  \mathbbm{1}_{\bar{a}\ge A_1}] <  \cJ_0(a;T,q_0)\cdot (\prob(\bar{a}>A_1) + C_T \varepsilon) \\ \text{for any}\; a > A_1/10.
\end{multline}

We now assume that $|a| \le A_1 /10$. Lemma \ref{lem: Asharp} implies that
\begin{equation}\label{eq: main 2 +6}
\prob(\bar{a} \ge 10^n A_1) \le C \exp(- c_\varepsilon q_0^2 10^n A_1)\;\text{for any}\; n \ge 0
\end{equation}
provided $\varepsilon$ is sufficiently small depending on $T$. We combine \eqref{eq: main 2 +6} and \eqref{eq: big a n4} to get
\begin{equation}
    \begin{split}
        \E[\cost_M \cdot \mathbbm{1}_{\bar{a}\ge A_1}] = \sum_{n\ge 0} \E[\cost_M \cdot \mathbbm{1}_{10^{n+1}A_1 > \bar{a} \ge 10^n A_1 }]\\
        \le C_T \sum_{n \ge 0} 10^{2n} A_1^2 q_0^2 \exp(-c_\varepsilon q_0^2 10^n A_1).
    \end{split}
\end{equation}
Taking $q_\bg$ sufficiently large depending on $\varepsilon$ and $T$ gives
\begin{equation}\label{eq: big a n6}
\E[\cost_M \cdot \mathbbm{1}_{\bar{a} \ge A_1}] \le \frac{\varepsilon}{A_1}\; \text{for any}\; |a| \le A_1/10.
\end{equation}
Lemma \ref{lem: OS asymp} implies that there exists $\tilde{A}$ depending only on $T$ such that
\begin{align}
    &\cJ_0(a;T,q_0) > c_T |a|^{-1}\;\text{for any}\; a < - \tilde{A},\label{eq: big a n7}\\
    & \cJ_0(a;T,q_0) > c_T \;\text{for any}\; a >- \tilde{A}.\label{eq: big a n8}
\end{align}
Combining \eqref{eq: big a n6}--\eqref{eq: big a n8} gives
\begin{equation}\label{eq: main 2.11}
\E[\cost_M \cdot \mathbbm{1}_{\bar{a} \ge A_1}] \le C_T \varepsilon \cJ_0(a;T,q_0)\;\text{for any}\;  |a| \le A_1/10.
\end{equation}

We now assume that $a < - A_1/10$. Combining \eqref{eq: big a n4} with Lemma \ref{lem: Asharp}, we have
\[
\E[\cost_M \cdot \mathbbm{1}_{10^{n+1}|a| \ge \bar{a} > 10^n |a|}] \le C_T 10^{2n} |a|^2 q_0^2 \exp(- c_\varepsilon q_0^2 10^n |a|) \; \text{for any}\; n \ge 1.
\]
Taking $q_\bg$ sufficiently large depending on $\varepsilon$ and $T$ and summing over $n\ge 1$ gives
\begin{equation}\label{eq: big a n9}
\E[\cost_M \cdot \mathbbm{1}_{\bar{a}>10 |a|}] \le  \frac{\varepsilon}{|a|}.
\end{equation}

Combining \eqref{eq: big a n3} with Lemma \ref{lem: 2 side pro} gives
\[
\E[\cost_M \cdot \mathbbm{1}_{A_1 \le \bar{a} < 10 |a|}] < C_{T,\varepsilon}|a|^2 q_0^2 \exp(-c_\varepsilon q_0^2 |a|).
\]
Taking $q_\bg$ sufficiently large depending on $\varepsilon$, $T$ gives
\begin{equation}\label{eq: big a n10}
\E[\cost_M \cdot \mathbbm{1}_{A_1 \le \bar{a} < 10 |a|}] < \frac{\varepsilon}{|a|}.
\end{equation}
Combining \eqref{eq: big a n7}, \eqref{eq: big a n9}, \eqref{eq: big a n10} gives
\begin{equation}\label{eq: main 2.12}
\E[\cost_M \cdot \mathbbm{1}_{\bar{a}\ge A_1}] < C_T \varepsilon \cJ_0(a; T,q_0)\;\text{for any}\; a < - A_1/10
\end{equation}
provided $A_1$ is sufficiently large depending on $T$.
Combining \eqref{eq: main 2.11} and \eqref{eq: main 2.12} gives
\begin{equation}\label{eq: main add 1}
\E[\cost_M \cdot \mathbbm{1}_{\bar{a}\ge A_1}] < C_T \varepsilon \cJ_0(a;T,q_0)\;\text{for any}\; a \le A_1 /10.
\end{equation}

\subsubsection{\texorpdfstring{$\cost_M$}{Cost M} when we believe that \texorpdfstring{$a$}{a} is bounded, positive}\label{sec: bdd pos}

Suppose that the event $\pro_+$ occurs and that $\bar{a} = \tilde{a}$ for some $0 < \tilde{a} < A_1$. This determines the time at which the Main Act begins; denote this time by $s$ and note that $q(s) = (1+\varepsilon)q_0$. In this case we set as our control variable
\[
u(t) = - \kappa(T-t,\tilde{a})\cdot q(t)
\]
during the Main Act (which lasts from time $s$ until time $T$), and thus
\begin{multline}\label{eq: mid a 2}
\E[\cost_M | \bar{a} = \tilde{a} ]\\ \le \cJ(\sigma_\opt(\tilde{a}),a;T,(1+\varepsilon)q_0)\;\text{for any}\; a \in \R, 0 < \tilde{a} < A_1.
\end{multline}
By Lemma \ref{lem: lin ctrl formula}, we have
\[
\cJ(\sigma_\opt(\tilde{a}),a;T,(1+\varepsilon)q_0) = h(\tilde{a},a)(1+\varepsilon)^2q_0^2 + j(\tilde{a},a)\;\text{for any}\; a, \tilde{a} \in \R
\]
and
\begin{equation}\label{eq: mid a 1}
\cJ_0(a;T,q_0) = h(a,a)q_0^2 + j(a,a)\;\text{for any}\; a \in \R,
\end{equation}
where $h,j$ are smooth, positive functions (depending on $T$). We therefore have
\begin{equation*}
    \begin{split}
        \cJ(\sigma_\opt(\tilde{a}),a;T,(1+\varepsilon)q_0) \le &\cJ_0(a;T,q_0) + |j(\tilde{a},a) - j(a,a)| + C h(a,a) \varepsilon q_0^2 \\
        &+ Cq_0^2|h(\tilde{a},a) - h(a,a)|.
    \end{split}
\end{equation*}
Since $h,j$ are smooth functions we have
\begin{multline*}
\cJ(\sigma_\opt(\tilde{a}),a;T,(1+\varepsilon)q_0) \le \cJ_0(a;T,q_0) + C_{A_1,T} q_0^2 |a-\tilde{a}| + C \varepsilon q_0^2 h(a,a) \\ \text{for any}\; a, \tilde{a} \in [0,10A_1]
\end{multline*}
and
\[
h(a,a) > c_{A_1,T}\;\text{for any} \; a \in [0,10A_1].
\]
Taking $0<\delta\ll 1$ to be sufficiently small depending on $\varepsilon, A_1, T$ gives
\begin{multline*}
\cJ(\sigma_\opt(\tilde{a}),a;T,(1+\varepsilon)q_0) < \cJ_0(a;T,q_0) + C \varepsilon q_0^2h(a,a)\\ \text{for any}\; |a - \tilde{a}| < \delta a, \; a \in [a_\tny,10A_1].
\end{multline*}
Combining this with \eqref{eq: mid a 2}, \eqref{eq: mid a 1} gives
\begin{equation}
\begin{split}
    \E[\cost_M | 0 < \bar{a} < A_1 \;\text{AND}\; |a-\bar{a}| < \delta a] \le (1+C \varepsilon) \cdot \cJ_0(a;T,q_0)\\
    \text{for any}\; a \in [a_\tny,10A_1].
\end{split}
\end{equation}
We have therefore shown
\begin{multline}\label{eq: main 3.4.5}
\E[ \cost_M \cdot \mathbbm{1}_{0 < \bar{a} < A_1} \cdot \mathbbm{1}_{|a - \bar{a}| < \delta a} ] \\
< (1+C \varepsilon) \cJ_0(a;T,q_0) \prob(0 < \bar{a} \le A_1) \;\text{for any}\; a \in[a_\tny, 10A_1].
\end{multline}

We combine \eqref{eq: mid a 2} and Corollary \ref{cor: opt alph score} to get
\begin{equation}\label{eq: mid a 3}
\E[\cost_M | \bar{a} = \tilde{a}] \le  C_T A_1^2 e^{2|a| T} q_0^2\;\text{for any}\; a \in \R, 0 < \tilde{a} < A_1.
\end{equation}
By Lemma \ref{lem: abar 1}, we have
\begin{align}
    &\prob(|a - \bar{a}| > \delta a) \le C \exp(-c_{\varepsilon,\delta} q_0^2 a)\;\text{for any}\; a > a_\tny,\label{eq: mid a 4}\\
    &\prob(0 < \bar{a} < A_1) \le C \exp(-c_\varepsilon q_0^2 a)\;\text{for any}\; a \ge 10A_1.\label{eq: mid a 5}
\end{align}
By \eqref{eq: tmax 2}, we have
\begin{equation}\label{eq: main 3.6}
\begin{split}
\prob(0 < \bar{a} \le A_1) &\le \prob(\pro_+) \\
&\le C \exp(- c_\varepsilon q_0^2(|a|+1))\;\text{for any}\; a \le a_\tny.
\end{split}
\end{equation}
Combining \eqref{eq: mid a 3}, \eqref{eq: mid a 5}, and \eqref{eq: main 3.6} gives
\begin{multline*}
\E[\cost_M \cdot \mathbbm{1}_{0 < \bar{a}< A_1}]\\  < C A_1^2 q_0^2 \exp(2|a|T-c_\varepsilon q_0^2(|a|+1)) \; \text{for any} \; a \notin [a_\tny,10A_1].
\end{multline*}
Combining \eqref{eq: mid a 3} and \eqref{eq: mid a 4} gives
\begin{multline*}
\E[\cost_M \cdot \mathbbm{1}_{0 < \bar{a} < A_1} \cdot \mathbbm{1}_{|a - \bar{a}|> \delta a }] \\ < C A_1^2 q_0^2 \exp(2aT-c_{\varepsilon, \delta} q_0^2 a)\;\text{for any}  \; a \in [a_\tny, 10A_1].
\end{multline*}
Taking $q_\bg$ to be sufficiently large depending on $\varepsilon, T$ (recall that $A_1$ is determined by $\varepsilon, T$ and $\delta$ is determined by $\varepsilon, A_1, T$) we get
\begin{align}
    &\E[\cost_M \cdot \mathbbm{1}_{0 < \bar{a} < A_1}] < \varepsilon \cdot \min\{1, |a|^{-1}\}\; \text{for any}\; a \notin [a_\tny, 10A_1],\label{eq: mid a 6} \\
    & \E[\cost_M \cdot \mathbbm{1}_{0 < \bar{a} < A_1} \cdot \mathbbm{1}_{|a - \bar{a}| > \delta a}] < \varepsilon \;\text{for any}\; a \in [a_\tny, 10A_1].\label{eq: mid a 7}
\end{align}
As a consequence of Lemma \ref{lem: OS asymp} we have
\begin{align}
&\min\{|a|^{-1}, 1\} < C_{T} \cdot \cJ_0(a;T,q_0)\;\text{for any}\; a \in \R,\label{eq: j0 1}\\
&1 < C_T\cdot  \cJ_0(a;T,q_0)\;\text{for any}\; a\ge 0.
\end{align}
Therefore \eqref{eq: mid a 6} implies 
\begin{equation}\label{eq: main 3.8}
\E[\cost_M \cdot \mathbbm{1}_{0 < \bar{a} < A_1}] <  C_T  \varepsilon \cJ_0(a;T,q_0)\;\text{for any}\; a \notin [a_\tny, 10 A_1],
\end{equation}
and \eqref{eq: main 3.4.5}, \eqref{eq: mid a 7} imply
\begin{multline}\label{eq: main 3.5.5}
\E[\cost_M \cdot \mathbbm{1}_{0 < \bar{a} < A_1}] < \cJ_0(a;T,q_0)(\prob(0 < \bar{a} < A_1)+C_T \varepsilon)\\ \text{for any}\; a \in [a_\tny, 10A_1].
\end{multline}

\subsubsection{\texorpdfstring{$\cost_M$}{Cost M} when we believe that \texorpdfstring{$a$}{a} is large, negative}
Recall that $\cg(0)$ denotes the simple feedback strategy with constant gain function $0$. This is simply the strategy in which we set the control variable equal to zero for the entire game.

In the event that $\bar{a} = \tilde{a}$ for some $\tilde{a} \le - A_1$, we play the strategy $\cg(0)$ during the Main Act. We therefore have
\begin{equation}\label{eq: big neg a 1}
\E[\cost_M | \bar{a} \le - A_1] \le \cJ(\cg(0),a;T,q_0) \;\text{for any} \; a \in \R.
\end{equation}
Taking $A_1$ to be sufficiently large depending on $T$, we apply Corollary \ref{cor: CG score} to get
\begin{equation}\label{eq: main 4.3}
\E[\cost_M | \bar{a} \le - A_1] \le \frac{1}{2|a|}(q_0^2+T)\;\text{for any}\; a \le -A_1 / 10.
\end{equation}
By Lemma \ref{lem: OS}, 
\begin{equation}\label{eq: main 4.5}
    \frac{1}{2|a|}(q_0^2 + T)\le (1+\varepsilon)\cJ_0(a;T,q_0)\;\text{for any}\; a \le -A_1 / 10
\end{equation}
provided $A_1$ is sufficiently large depending on $\varepsilon$ and $T$. Combining \eqref{eq: main 4.3} and \eqref{eq: main 4.5} gives
\begin{multline}\label{eq: main 4.6}
\E[\cost_M \cdot \mathbbm{1}_{\bar{a}\le -A_1}] \le (1+\varepsilon)\cdot  \cJ_0(a;T,q_0)\cdot \prob(\bar{a} \le - A_1)\\ \text{for any}\; a \le -A_1 / 10.
\end{multline}

Combining \eqref{eq: big neg a 1} and Corollary \ref{cor: CG score} gives
\begin{equation}\label{eq: main 4.4}
\E[\cost_M | \bar{a} \le - A_1 ] \le C_T e^{2|a|T}q_0^2\;\text{for any}\; a > -A_1/10.
\end{equation}
Combining Lemma \ref{lem: abar 1} and \eqref{eq: tmax 3} gives
\begin{equation}\label{eq: main 4.7}
\prob(\bar{a} \le - A_1) \le C \exp(-c_\varepsilon q_0^2 (|a|+1))\;\text{for any} \;  a > - A_1/10.
\end{equation}
Inequalities \eqref{eq: main 4.4} and \eqref{eq: main 4.7} imply that
\[
\E[\cost_M \cdot \mathbbm{1}_{\bar{a}\le -A_1}] \le C_T q_0^2 \exp(2|a|T - c_\varepsilon q_0^2(|a|+1))\;\text{for any}\; a > - A_1/10 .
\]
Taking $q_\bg$ sufficiently large depending on $\varepsilon,T$, we get
\begin{equation*}
    \E[\cost_M \cdot \mathbbm{1}_{\bar{a} \le - A_1}] \le \varepsilon  \min \left\{1, |a|^{-1}\right\} \; \text{for any}\; a > - A_1/10.
\end{equation*}
Combining this with \eqref{eq: j0 1} gives
\begin{equation}\label{eq: main 4.12}
    \E[\cost_M \cdot \mathbbm{1}_{\bar{a}\le -A_1} ] \le C_T \varepsilon \cJ_0(a;T,q_0)\;\text{for any} \; a > -A_1/10.
\end{equation}

\subsubsection{\texorpdfstring{$\cost_M$}{Cost M} when we believe that \texorpdfstring{$a$}{a} is bounded, negative}\label{sec: main 5}
Proceeding as in the proofs of equations \eqref{eq: main 3.8} and \eqref{eq: main 3.5.5} in Section \ref{sec: bdd pos}, we can show that
\begin{equation}\label{eq: main 5.2}
    \E[\cost_M \cdot \mathbbm{1}_{-A_1 < \bar{a} < 0}] < C_T \varepsilon \cJ_0(a;T,q_0) \;\text{for any}\; a \notin [-10A_1, -a_\tny]
\end{equation}
and
\begin{multline}\label{eq: main 5.1}
    \E[\cost_M \cdot \mathbbm{1}_{- A_1 < \bar{a}<0} ] 
    \le \cJ_0(a;T,q_0)(\prob(-A_1 \le \bar{a} < 0) + C_T\varepsilon)\\ \text{for any} \; a \in [-10A_1, -a_\tny].
\end{multline}

\subsubsection{Proof of Theorem \ref{thm: lqs}}\label{sec: main thm}
We collect the estimates \eqref{eq: main 1.10}, \eqref{eq: main 1.11}, \eqref{eq: main 2.10}, \eqref{eq: main add 1}, \eqref{eq: main 3.5.5}, \eqref{eq: main 3.8}, \eqref{eq: main 4.6}, \eqref{eq: main 4.12}, \eqref{eq: main 5.1}, and \eqref{eq: main 5.2} from Section \ref{sec: main}. We assume that $q_\bg$ and $A_1$ are large enough depending on $\varepsilon$ and $T$ for all of these estimates to hold.

\begin{multline*}
     \E[\cost_M \cdot \mathbbm{1}_{\pro_\mx}]  \\ \le \cJ_0(a;T,q_0) ( \prob(\pro_\mx) + C_T \varepsilon^{1/4})\;\text{for any}\; |a| \le a_\sml,
\end{multline*}
\begin{flalign*}
    & \quad \E[\cost_M \cdot \mathbbm{1}_{\pro_\mx}] \le \varepsilon C_T \cJ_0(a;T,q_0) \;\text{for any}\; |a| \ge a_\sml,&
\end{flalign*}
\begin{multline*}
    \E[\cost_M \cdot \mathbbm{1}_{\bar{a}\ge A_1}]
    \le \cJ_0(a;T,q_0) (\prob(\bar{a}>A_1) + C_T \varepsilon)\;\text{for any}\; a > \frac{1}{10}A_1,
\end{multline*}
\begin{flalign*}
    &\quad \E[\cost_M \cdot \mathbbm{1}_{\bar{a} \ge A_1}] \le C_T \varepsilon \cJ_0(a;T,q_0)\;\text{for any}\; a \le \frac{1}{10}A_1,&
\end{flalign*}
\begin{multline*}
    \E[\cost_M \cdot \mathbbm{1}_{0 < \bar{a} < A_1}] \\ \le \cJ_0(a;T,q_0) (\prob(0 < \bar{a} \le A_1) + C_T \varepsilon)\;\text{for any}\; a \in [a_\tny, 10A_1],
\end{multline*}
\begin{flalign*}
&\quad \E[\cost_M \cdot \mathbbm{1}_{0 < \bar{a} < A_1}] \le C_T \varepsilon \cJ_0(a;T,q_0) \; \text{for any}\; a \notin [a_\tny, 10A_1],&
\end{flalign*}
\begin{multline*}
    \E[\cost_M \cdot \mathbbm{1}_{\bar{a} \le - A_1}] \\\le  \cJ_0(a;T,q_0) (\prob(\bar{a} < - A_1) + \varepsilon)\; \text{for any}\; a \le - \frac{1}{10}A_1,
\end{multline*}
\begin{flalign*}
&\quad\E[\cost_M \cdot \mathbbm{1}_{\bar{a} \le - A_1}] \le C_T \varepsilon \cJ_0(a;T,q_0) \;\text{for any}\; a > - \frac{1}{10}A_1,&
\end{flalign*}
\begin{multline*}
    \E[\cost_M \cdot \mathbbm{1}_{-A_1 < \bar{a}<0}]\\
    \le \cJ_0(a;T,q_0) (\prob(-A_1 \le \bar{a} < 0) + C_T\varepsilon) \; \text{for any} \; a \in [-10A_1, - a_\tny],
\end{multline*}
\begin{flalign*}
&\quad \E[\cost_M \cdot \mathbbm{1}_{-A_1 < \bar{a} < 0}] < C_T \varepsilon \cJ_0(a;T,q_0) \;\text{for any}\; a \notin [-10A_1, - a_\tny].&
\end{flalign*}
Combining these estimates proves
\begin{equation}\label{eq: prf 1}
\E[\cost_M ] \le \cJ_0(a;T,q_0) ( 1 + C_T\varepsilon^{1/4}) \;\text{for any}\; a \in \R.
\end{equation}
Combining \eqref{eq: prf 1} with \eqref{eq: tmax 4} and taking $q_\bg$ to be large enough so that \eqref{eq: tmax 4} holds proves that
\begin{align*}
\cJ(\lqs, a;T,q_0) &= \E[\cost_P] + \E[\cost_M] \\
& \le \cJ_0(a;T,q_0) (1 + C_T \varepsilon^{1/4})\;\text{for any}\; a \in \R.
\end{align*}
This proves Theorem \ref{thm: lqs}.

\section{The Large \texorpdfstring{$a$}{a} Strategy}\label{sec: las}

In this section we prove Theorem \ref{thm: las}.

We fix a time horizon $T>0$ and a starting position $q_0>0$. 

Let $\varepsilon >0$ be given and let $A\ge 1$ be a sufficiently large number depending on $\varepsilon,T,q_0$. Without loss of generality we assume that $\varepsilon$ is sufficiently small depending on $T,q_0$.

By Theorem \ref{thm: lqs}, there exists a number $q_0^* \ge 2 q_0$ depending on $q_0$ and $T$ and a strategy $\br$ (here ``$\br$'' stands for Bounded Regret) such that
\begin{equation*}
\cJ(\br, a;T, q_0^*) \le 2\cdot  \cJ_0(a;T, q_0^*)\;\text{for any}\; a \in \R.
\end{equation*}
By Remark \ref{cor: q_0s}, we have
\begin{equation}\label{eq: br las}
\cJ(\br, a;T, q_0^*) \le C_{T,q_0} \cdot \cJ_0(a;T, q_0)\;\text{for any}\; a \in \R.
\end{equation}
The strategy $\br$ for time horizon $T$ and starting position $q_0^*$ gives rise to a strategy $\br_-$ for time horizon $T$ and starting position $-q_0^*$ satisfying
\[
\cJ(\br_-,a;T,-q_0^*) = \cJ(\br,a;T,q_0^*)
\]
(see Section \ref{sec: setup} for details). For the remainder of Section \ref{sec: las} we will write $\br$ to denote both of these strategies; it will be clear from context which strategy we are referring to.

We now define the strategy $\las$. For the remainder of Section \ref{sec: las} we write $q,u$ to denote $q^{\las}$, $u^{\las}$.

\underline{\textsc{Testing Epoch}}: We let $\tau_C$ denote the first time $t \in (0,T)$ for which $q(t) = (1+\varepsilon)q_0$ or $q(t) = -q_0^*$, if such a time exists. If no such time exists we set $\tau_C = T$. The Testing Epoch ends when we reach time $\tau_C$. During the Testing Epoch we set $u=0$. If $\tau_C = T$, then the game is over when we reach time $\tau_C$. If $\tau_C <T$, then at time $\tau_C$ we enter the Control Epoch.

\underline{\textsc{Control Epoch}}: In a moment, we will define a stopping time $\tau_D$. With probability 1 we have $\tau_D \in (\tau_C, T]$. If the Control Epoch occurs, then it lasts from time $\tau_C$ until time $\tau_D$.

In the event that we enter the Control Epoch at some time $\tau_C \in (0,T)$ for which $q(\tau_C) = -q_0^*$, we set $\tau_D = T$ so that the Control Epoch begins at time $\tau_C$ and position $-q_0^*$ and lasts until the game ends at time $T$. In this case, during the Control Epoch we execute the strategy $\br$.

In the event that we enter the Control Epoch at some time $\tau_C \in (0,T)$ for which $q(\tau_C) = (1+\varepsilon)q_0$, we define a random variable
\[
\bar{a} = \frac{\log(1+\varepsilon)}{\tau_C}.
\]
We then set
\[
u(t) = - 2\bar{a} q(t)
\]
during the Control Epoch. In this case, we define the stopping time $\tau_D$ to be equal to the first time $t \in (\tau_C, T)$ for which $|q(t)| = q_0^*$ if such a time exists and equal to $T$ if no such time exists. 

In either case, if $\tau_D = T$, then the Control Epoch ends along with the game at time $T$. If $\tau_D<T$, then at time $\tau_D$ we enter the Disaster Mitigation Epoch.

\underline{\textsc{Disaster Mitigation Epoch}}: If we enter the Disaster Mitigation Epoch at some time $\tau_D \in (\tau_C, T)$, then we have $|q(\tau_D)|=q_0^*$. We then execute the strategy $\br$ from time $\tau_D$ until the game ends at time $T$.

This concludes the definition of the strategy $\las$.

We define three random variables:
\[
\cost_T = \int_0^{\tau_C} q^2(t)\  dt,
\]
\[
\cost_C = \int_{\tau_C}^{\tau_D} (q^2(t) + u^2(t))\ dt,
\]
and
\[
\cost_D = \int_{\tau_D}^T (q^2(t) + u^2(t))\ dt.
\]
These are, respectively, the costs incurred during the Testing Epoch, the Control Epoch, and the Disaster Mitigation Epoch for a given realization of the noise. Note that they all depend on $a$.

Let $X \subset (0,\infty)$ be an arbitrary subset. We will be interested in events of the form ``$\tau_C < T$ and $q(\tau_C) = (1+\varepsilon)q_0$ and $\bar{a} \in X$''. We will abbreviate this by writing just ``$\bar{a} \in X$''; note that $\bar{a}$ is undefined if either $\tau_C = T$ or if $q(\tau_C) = - q_0^*$.

By Lemma \ref{lem: OS asymp}, provided $A$ is sufficiently large depending on $T$ we have
\begin{equation}\label{eq: las j0}
c_{T,q_0} \le \cJ_0(a; T,q_0)\;\text{for any}\; a \ge A.
\end{equation}
We will make use of this fact throughout this section.

\subsection{The cost during the Testing Epoch}
 By Lemma \ref{lem: 1 side pro}, provided $\varepsilon$ is sufficiently small depending on $T$ we have
\begin{equation}\label{eq: test 1}
\E[\cost_T] <  C_{T,q_0}\cdot \varepsilon^{1/4} \;\text{for any}\; a \in \R
\end{equation}
(recall that $q_0^*$ is determined by $T,q_0$). Combining this with \eqref{eq: las j0} gives
\begin{equation}\label{eq: test 2}
\E[\cost_T] < C_{T,q_0} \cdot \varepsilon^{1/4}\cdot \cJ_0(a;T,q_0)\;\text{for any}\; a \ge A.
\end{equation}

\subsection{The cost during the Control Epoch}\label{sec: las control}
Our goal in this section is to control the expected value of the random variable $\cost_C$, defined above.

Observe that
\begin{equation}\label{eq: ctrl 1}
\begin{split}
\E[\cost_C] = \E[\cost_C \cdot \mathbbm{1}_{\bar{a}>0}] + \E[\cost_C \cdot \mathbbm{1}_{q(\tau_C)=-q_0^*}].
\end{split}
\end{equation}

In the event that we enter the Control Epoch at some time $\tau_C \in (0,T)$ for which $q(\tau_C) = - q_0^*$, we play the strategy $\br$ starting from position $-q_0^*$ and time $\tau_C$ until the game ends at time $T$. By \eqref{eq: br las} we therefore have
\begin{equation}\label{eq: ctrl 1.1}
\E[\cost_C | (q(\tau_C) = - q_0^*)] \le C_{T,q_0} \cdot \cJ_0(a;T,q_0)\;\text{for any}\; a \in \R.
\end{equation}
Provided $A$ is sufficiently large depending on $\varepsilon$, and $\varepsilon$ is sufficiently small depending on $T$, Lemma \ref{lem: 1 side pro} implies
\begin{equation*}
\prob(q(\tau_C) =  - q_0^*) \le C_{q_0}\cdot  \varepsilon^{1/4}\;\text{for any}\; a \ge A.
\end{equation*}
Therefore
\begin{equation}\label{eq: ctrl 2}
\E[\cost_C \cdot \mathbbm{1}_{q(\tau_C) = - q_0^*}] \le C_{T,q_0}\cdot   \varepsilon^{1/4}   \cdot \cJ_0(a;T,q_0)\;\text{for any}\; a \ge A.
\end{equation}
By Corollary \ref{cor: opt alph score},
\[
\cJ_0(a;T,q_0) \le C_{T,q_0}\cdot A\;\text{for any}\; a \le A.
\]
Combining this with \eqref{eq: ctrl 1.1} gives
\begin{equation}\label{eq: ctrl 3}
\E[\cost_C \cdot \mathbbm{1}_{q(\tau_C) = - q_0^*}] \le C_{T,q_0}\cdot A \; \text{for any}\; a \le A.
\end{equation}

We will make use of the following two observations throughout the remainder of this section. First, we note that
\begin{equation}\label{eq: ctrl 4.1}
\E[\cost_C | \bar{a} = \tilde{a}] \le C_{T,q_0}(\tilde{a}^2 + 1)\;\text{for any}\; \tilde{a}>0, a \in \R.
\end{equation}
This follows by observing that in the event that we enter the Control Epoch we have $|q| \le q_0^*$ and $u = - 2 \bar{a} q$ during the Control Epoch with probability 1. Second, we define
\[
A^* = \max\{A, |a|\}.
\]
Taking $\varepsilon$ sufficiently small depending on $T$, Lemma \ref{lem: Asharp} implies that
\begin{equation}\label{eq: Astar}
\prob(\bar{a} > 10^n A^*) \le C \exp(-c_\varepsilon q_0^2 10^n A^*)\;\text{for any}\; n \ge 1, a \in \R.
\end{equation}

We now introduce a parameter $1 \gg \delta>0$ that will be chosen below to be sufficiently small depending on $\varepsilon$. Note that since $\delta$ is determined by $\varepsilon$ we are allowed to take $A$ sufficiently large depending on $\delta$. 

Assume that $ a \ge A$. Observe that
\begin{equation}\label{eq: ctrl 4}
\begin{split}
\E[\cost_C \cdot \mathbbm{1}_{\bar{a}>0}] =& \E[\cost_C \cdot \mathbbm{1}_{|a - \bar{a}|< \delta a}]
+ \E[\cost_C \cdot \mathbbm{1}_{\bar{a} < (1-\delta)a}]\\
&+ \E[\cost_C \cdot \mathbbm{1}_{\bar{a} > (1+\delta)a}].
\end{split}
\end{equation}
In the event that $\bar{a} = \tilde{a}$ for some $\tilde{a}>0$, then during the Control Epoch we play the constant gain strategy $\cg(2\tilde{a})$ (see Section \ref{sec: known a}, specifically the discussion before Corollary \ref{cor: CG score}, for the definition of the strategy $\cg$). Therefore, by Corollary \ref{cor: CG score},
\[
\E[\cost_C | (|a-\bar{a}|<\delta a)] < (1+4a^2(1+\delta)^2) \frac{q_0^2 + T}{2a(1-C\delta)}.
\]
Taking $A$ sufficiently large and $\delta$ sufficiently small, both depending on $\varepsilon$, we have
\[
\E[\cost_C | (|a-\bar{a}|<\delta a)] < 2a(1+\varepsilon)(q_0^2+T).
\]
By Lemma \ref{lem: OS asymp}, provided $A$ is sufficiently large depending on $\varepsilon$ and $T$ we have
\[
\cJ_0(a;T,q_0) \ge 2a (1-\varepsilon)(q_0^2+T).
\]
We deduce that
\begin{equation}\label{eq: ctrl 4.5}
\E[\cost_C \cdot \mathbbm{1}_{|a - \bar{a}|< \delta a}] < (1+ C \varepsilon) \cdot \cJ_0(a;T,q_0)\;\text{for any} \; a \ge A.
\end{equation}

We continue to assume that $a \ge A$. By \eqref{eq: ctrl 4.1},  we have
\[
\E[\cost_C | \bar{a} < (1-\delta)a ] \le C_{T,q_0}\cdot a^2.
\]
Provided $A$ is sufficiently large depending on $\varepsilon$, and $\varepsilon$ is sufficiently small depending on $T$, Lemma \ref{lem: abar 1} implies that
\begin{equation}\label{eq: las add 1}
\prob(|\bar{a}-a| \ge \delta a) \le C \exp(-c_{\varepsilon,\delta} q_0^2 a).
\end{equation}
Therefore
\[
\E[\cost_C \cdot \mathbbm{1}_{\bar{a} < (1-\delta)a}] < C_{T,q_0} \cdot a^2 \cdot \exp(-c_{\varepsilon, \delta} q_0^2 a).
\]
Taking $A$ sufficiently large depending on $\varepsilon, \delta, q_0, T$ gives
\begin{equation}\label{eq: ctrl 5}
\E[\cost_C \cdot \mathbbm{1}_{\bar{a} < (1-\delta)a}] < \varepsilon\;\text{for any}\; a \ge A.
\end{equation}

We continue to assume that $a \ge A$. By \eqref{eq: ctrl 4.1}, we have
\begin{align}
   &\E[\cost_C | (1+\delta)a < \bar{a} < 10 a] < C_{T,q_0} \cdot a^2, \label{eq: ctrl 5.2}\\
   &\E[\cost_C | 10^n a < \bar{a} < 10^{n+1}a] < C_{T,q_0}\cdot  10^{2n} a^2\;\text{for any}\; n \ge 1.\label{eq: ctrl 5.3}
\end{align}
Since
\begin{equation*}
\begin{split}
\E[\cost_C \cdot \mathbbm{1}_{\bar{a} > (1+\delta)a}] =& \E[\cost_C \cdot \mathbbm{1}_{(1+\delta)a<\bar{a}<10 a}]\\
&+\sum_{n=1}^\infty \E[\cost_C \cdot \mathbbm{1}_{10^na < \bar{a} < 10^{n+1}a}],
\end{split}
\end{equation*}
equations \eqref{eq: Astar} (our assumption that $a \ge A$ implies that $A^* = a$), \eqref{eq: las add 1}, \eqref{eq: ctrl 5.2}, and \eqref{eq: ctrl 5.3} imply that
\begin{equation*}
\begin{split}
\E[\cost_C \cdot \mathbbm{1}_{\bar{a} > (1+\delta)a}] <& C_{T,q_0}\cdot a^2 \exp(-c_{\varepsilon, \delta}q_0^2 a) \\
&+ \sum_{n=1}^\infty C_{T,q_0}\cdot 10^{2n}a^2 \exp(-c_\varepsilon q_0^2 10^n a).
\end{split}
\end{equation*}
Taking $A$ sufficiently large depending on $\varepsilon,\delta,T,q_0$ gives
\begin{equation}\label{eq: ctrl 6}
\E[\cost_C \cdot \mathbbm{1}_{\bar{a}>(1+\delta) a}] < \varepsilon\;\text{for any}\; a \ge A.
\end{equation}

Combining \eqref{eq: ctrl 4}, \eqref{eq: ctrl 4.5}, \eqref{eq: ctrl 5}, \eqref{eq: ctrl 6} gives
\begin{equation*}
\E[\cost_C \cdot \mathbbm{1}_{\bar{a}>0}] < (1+C\varepsilon) \cdot \cJ_0(a;T,q_0) + C' \varepsilon\;\text{for any}\; a \ge A;
\end{equation*}
\eqref{eq: las j0} then implies that
\begin{equation}\label{eq: ctrl 7}
\E[\cost_C \cdot \mathbbm{1}_{\bar{a}>0}] < (1+C_{T,q_0}\cdot\varepsilon) \cdot \cJ_0(a;T,q_0) \;\text{for any}\; a \ge A.
\end{equation}

We now suppose that $|a| < A$. Note that
\begin{equation}\label{eq: ctrl 7.2}
\begin{split}
\E[\cost_C \cdot \mathbbm{1}_{\bar{a}>0}]  =& \E[\cost_C \cdot \mathbbm{1}_{\bar{a}< 10 A}]\\
&+ \sum_{n=1}^\infty \E[\cost_C \cdot \mathbbm{1}_{10^n A < \bar{a} < 10^{n+1}A}].
\end{split}
\end{equation}
Equation \eqref{eq: ctrl 4.1} implies
\begin{align}
&\E[\cost_C | \bar{a} < 10 A] < C_{T,q_0} A^2 \label{eq: ctrl 7.3}, \\
&\E[\cost_C | 10^n A< \bar{a} < 10^{n+1}A] \le C_{T,q_0} 10^{2n} A^2\;\text{for any}\; n \ge 1. \label{eq: ctrl 7.4}
\end{align}
We combine \eqref{eq: ctrl 7.2}-- \eqref{eq: ctrl 7.4} with \eqref{eq: Astar} (our assumption $|a| < A$ implies that $A^* = A$) to deduce that
\begin{equation}
\begin{split}
\E[\cost_C \cdot \mathbbm{1}_{\bar{a}>0}] \le C_{T,q_0} A^2 + \sum_{n=1}^\infty C_{T,q_0}  10^{2n} A^2 \exp(-c_\varepsilon q_0^2 10^n A).
\end{split}
\end{equation}
Taking $A$ sufficiently large depending on $\varepsilon, q_0$ gives
\begin{equation}\label{eq: ctrl 8}
\E[\cost_C \cdot \mathbbm{1}_{\bar{a}>0}] < C_{T,q_0}\cdot A^2\;\text{for any} \; |a| < A.
\end{equation}

We now assume that $a < -A$. Note that
\begin{equation*}
\prob(\bar{a}<10 |a|) = \prob(0 < \bar{a} <10 |a|) \le \prob(\bar{a}>0).
\end{equation*}
Applying Lemma \ref{lem: decay}, we get
\begin{equation}\label{eq: ctrl 8.1}
\prob(\bar{a}<10 |a|) \le C_{\varepsilon,T,q_0} \cdot \exp(-c_\varepsilon q_0^2 |a|).
\end{equation}
Note that
\begin{equation*}
\begin{split}
\E[\cost_C \cdot \mathbbm{1}_{\bar{a}>0}] =& \E[\cost_C \cdot \mathbbm{1}_{\bar{a}< 10 |a|}]\\
&+\sum_{n=1}^\infty \E[\cost_C \cdot \mathbbm{1}_{10^n|a| < \bar{a} < 10^{n+1}|a|}].
\end{split}
\end{equation*}
We combine this with \eqref{eq: ctrl 4.1}, \eqref{eq: Astar}, and \eqref{eq: ctrl 8.1} to get
\begin{multline*}
\E[\cost_C \cdot \mathbbm{1}_{\bar{a}>0}]\\ \le C_{\varepsilon,T,q_0} |a|^2 \exp(-c_\varepsilon q_0^2 |a|)+ C_{T,q_0} \sum_{n=1}^\infty 10^{2n}|a|^2 \exp(-c_\varepsilon q_0^2 10^n |a| ).
\end{multline*}
Taking $A$ sufficiently large depending on $\varepsilon, T, q_0$ gives 
\begin{equation}\label{eq: ctrl 9}
\E[\cost_C \cdot \mathbbm{1}_{\bar{a}>0}] < \varepsilon \;\text{for any}\; a < -A.
\end{equation}

We combine \eqref{eq: ctrl 1}, \eqref{eq: ctrl 2}, and \eqref{eq: ctrl 7} to get that
\begin{equation}\label{eq: ctrl 10}
\E[\cost_C] < (1+C_{T,q_0}\cdot \varepsilon^{1/4}) \cJ_0(a;T,q_0) \;\text{for any} \; a \ge A,
\end{equation}
and we combine \eqref{eq: ctrl 1}, \eqref{eq: ctrl 3}, \eqref{eq: ctrl 8}, and \eqref{eq: ctrl 9} to get
\begin{equation}\label{eq: ctrl 11}
\E[\cost_C] < C_{T,q_0} \cdot A^2\;\text{for any}\; a \le A.
\end{equation}

\subsection{The cost during the Disaster Mitigation Epoch}\label{sec: las disaster}
Our goal in this section is to control the expected value of the random variable $\cost_D$.

Let $\cD$ denote the event that we reach the Disaster Mitigation Epoch, i.e., $\cD$ is the event that $\tau_C<T$, $q(\tau_C) = (1+\varepsilon)q_0$, and $\tau_D < T$. For a subset $X \subseteq (0,\infty)$, we let $\cD(X)$ denote the event that $\cD$ occurs and $\bar{a} \in X$. Note that $\cD = \cD((0,\infty))$.

We remark that
\[
\E[\cost_D] = \E[\cost_D \cdot \mathbbm{1}_{\cD}].
\]

Suppose that we enter the Disaster Mitigation Epoch at some time $\tau_D \in (0,T)$ and suppose that $\bar{a}= \tilde{a}$ for some $\tilde{a}>0$. Starting from position $|q| = q_0^*$ and time $\tau_D$ until time $T$ we execute the strategy $\br$. Therefore, by \eqref{eq: br las}, we have
\begin{equation}\label{eq: dme 1}
\E[\cost_D | \cD \;\text{AND}\; \bar{a} = \tilde{a} ] \le C_{T,q_0} \cdot \cJ_0(a;T,q_0)\;\text{for any}\; a \in \R, \tilde{a} > 0.
\end{equation}
We remark that the RHS of \eqref{eq: dme 1} is independent of $\tilde{a}$, and so we have
\[
\E[\cost_D | \cD(X)] \le C_{T,q_0}\cdot \cJ_0(a;T,q_0)\;\text{for any}\; a \in \R, X \subset (0,\infty).
\]
Combining this with Corollary \ref{cor: opt alph score} gives
\begin{equation}\label{eq: dme 1.1}
\E[\cost_D | \cD(X) ] \le C_{T,q_0}\cdot  \max\{a,1\} \; \text{for any}\; a \in \R, X\subset (0,\infty).
\end{equation}
In particular, taking $X = (0,\infty)$ implies that
\begin{equation}\label{eq: dme 4}
\E[\cost_D ] \le C_{T,q_0}\cdot A \;\text{for any}\; a \le A.
\end{equation}

We now show that the probability of reaching the Disaster Mitigation Epoch is small when $a \ge A$.

Assume $a \ge A$. Observe that
\[
\prob(\cD((0,3a/4))) \le \prob( 0 < \bar{a} < 3a/4).
\]
Provided $A$ is sufficiently large depending on $\varepsilon$, and $\varepsilon$ is sufficiently small depending on $T$, Lemma \ref{lem: abar 1} implies that
\begin{equation}\label{eq: dme 2}
\prob(\cD ( (0, 3a/4))) \le  C\exp(-c_\varepsilon q_0^2 a).
\end{equation}

Now suppose that the event $\bar{a} = \tilde{a}$ occurs for some $\tilde{a}> 3a /4$. In this case, the probability that we enter the Disaster Mitigation Epoch is less than or equal to
\[
\prob(\exists t \in [0,T]: |\bar{q}(t)| \ge q_0^*),
\]
where $\bar{q}:[0,T]\rightarrow \R$ is governed by
\[
d\bar{q} = (a - 2\tilde{a})\bar{q} dt + dW_t, \qquad \bar{q}(0) = (1+\varepsilon)q_0.
\]
Since $a - 2 \tilde{a} < 0$, we apply Lemma \ref{lem: decay} to get
\[
\prob(\exists t \in [0,T]: |\bar{q}(t)| \ge q_0^*) \le C_{T,q_0}\cdot \exp(-c_{T,q_0} a).
\]
This holds for any $\tilde{a} > 3a/4$; we deduce that
\begin{equation}\label{eq: dme 3}
\prob(\cD(( 3a/4, \infty))) \le C_{T,q_0}\cdot \exp(-c_{T,q_0} a).
\end{equation}

Combining \eqref{eq: dme 1.1}, \eqref{eq: dme 2}, \eqref{eq: dme 3} gives
\begin{equation*}
\begin{split}
\E[\cost_D ] &= \E[\cost_D \cdot \mathbbm{1}_{\cD((0, 3a/4))}] + \E[\cost_D \cdot \mathbbm{1}_{ \cD (( 3a/4, \infty))}]\\
& \le C_{T,q_0}\cdot  a\cdot \exp(-c_\varepsilon q_0^2 a) + C_{T,q_0}\cdot  a\cdot \exp(-c_{T,q_0} a)
\end{split}
\end{equation*}
for any $a \ge A$. Taking $A$ sufficiently large depending on $\varepsilon, T, q_0$ gives
\begin{equation*}
\E[\cost_D] < \varepsilon \; \text{for any}\; a \ge A.
\end{equation*}
Combining this with \eqref{eq: las j0} gives
\begin{equation}\label{eq: dme 5}
\E[\cost_D] < \varepsilon \cdot  C_{T,q_0}\cdot  \cJ_0(a;T,q_0) \; \text{for any}\; a \ge A.
\end{equation}

\subsection{Proof of Theorem \ref{thm: las}}
Note that
\begin{equation}\label{eq: las thm 1}
\cJ(\las,a;T,q_0) = \E[\cost_T] + \E[\cost_C] + \E[\cost_D].
\end{equation}
By \eqref{eq: test 2}, \eqref{eq: ctrl 10}, and \eqref{eq: dme 5} we have
\[
\cJ(\las,a;T,q_0) < (1+C_{T,q_0} \cdot \varepsilon^{1/4}) \cdot \cJ_0(a;T,q_0)\;\text{for any}\; a \ge A.
\]
By \eqref{eq: test 1}, \eqref{eq: ctrl 11}, and \eqref{eq: dme 4}, we have
\[
\cJ(\las,a;T,q_0) < C_{T,q_0}\cdot A^2\;\text{for any}\; a \le A.
\]
This proves Theorem \ref{thm: las}.

\bibliographystyle{plain}
\bibliography{ref}

\begin{thebibliography}{10}

\bibitem{abbasi2011regret}
Yasin Abbasi-Yadkori and Csaba Szepesv{\'a}ri.
\newblock Regret bounds for the adaptive control of linear quadratic systems.
\newblock In {\em Proceedings of the 24th Annual Conference on Learning Theory}, pages 1--26. JMLR Workshop and Conference Proceedings, 2011.

\bibitem{abeille2017thompson}
Marc Abeille and Alessandro Lazaric.
\newblock Thompson sampling for linear-quadratic control problems.
\newblock In {\em Artificial Intelligence and Statistics}, pages 1246--1254. PMLR, 2017.

\bibitem{astrom}
Karl Astr\"{o}m.
\newblock {\em Introduction to Stochastic Control Theory}.
\newblock Academic Press, 1970.

\bibitem{bertsekas2012dynamic}
Dimitri Bertsekas.
\newblock {\em Dynamic programming and optimal control: Volume I}, volume~1.
\newblock Athena scientific, 2012.

\bibitem{Brazy:2009}
D.P. Brazy.
\newblock Group chairman's factual report of investigation.
\newblock {\em National Transportation Safety Board Docket No. SA-532, Exhibit No. 12}, 2009.

\bibitem{carruth2022controlling}
Jacob Carruth, Maximilian~F. Eggl, Charles Fefferman, Clarence~W. Rowley, and Melanie Weber.
\newblock Controlling unknown linear dynamics with bounded multiplicative regret.
\newblock {\em Revista Matem{\'a}tica Iberoamericana}, 38(7):2185--2216, 2022.

\bibitem{boundeda}
Jacob Carruth, Maximilian~F. Eggl, Charles~L Fefferman, and Clarence~W. Rowley.
\newblock Optimal agnostic control of unknown linear dynamics in a bounded parameter range.
\newblock {\em forthcoming}, 2023.

\bibitem{cesa2006prediction}
Nicolo Cesa-Bianchi and G{\'a}bor Lugosi.
\newblock {\em Prediction, learning, and games}.
\newblock Cambridge university press, 2006.

\bibitem{chen2021black}
Xinyi Chen and Elad Hazan.
\newblock Black-box control for linear dynamical systems.
\newblock In {\em Conference on Learning Theory}, pages 1114--1143. PMLR, 2021.

\bibitem{chen2023regret}
Xinyi Chen, Edgar Minasyan, Jason~D Lee, and Elad Hazan.
\newblock Regret guarantees for online deep control.
\newblock {\em Proceedings of Machine Learning Research vol XX}, 1:34, 2023.

\bibitem{Cohen:2019}
Alon Cohen, Tomer Koren, and Yishay Mansour.
\newblock Learning linear-quadratic regulators efficiently with only $\sqrt{T}$ regret.
\newblock In Kamalika Chaudhuri and Ruslan Salakhutdinov, editors, {\em Proceedings of the 36th International Conference on Machine Learning}, volume~97 of {\em Proceedings of Machine Learning Research}, pages 1300--1309. PMLR, 09--15 Jun 2019.

\bibitem{dean2018regret}
Sarah Dean, Horia Mania, Nikolai Matni, Benjamin Recht, and Stephen Tu.
\newblock Regret bounds for robust adaptive control of the linear quadratic regulator.
\newblock {\em arXiv preprint arXiv:1805.09388}, 2018.

\bibitem{dean2019safely}
Sarah Dean, Stephen Tu, Nikolai Matni, and Benjamin Recht.
\newblock Safely learning to control the constrained linear quadratic regulator.
\newblock In {\em 2019 American Control Conference (ACC)}, pages 5582--5588. IEEE, 2019.

\bibitem{duchi2011adaptive}
John Duchi, Elad Hazan, and Yoram Singer.
\newblock Adaptive subgradient methods for online learning and stochastic optimization.
\newblock {\em Journal of Machine Learning Research}, 12(7), 2011.

\bibitem{faury2021regret}
Louis Faury, Yoan Russac, Marc Abeille, and Cl{\'e}ment Calauz{\`e}nes.
\newblock Regret bounds for generalized linear bandits under parameter drift.
\newblock {\em arXiv preprint arXiv:2103.05750}, 2021.

\bibitem{fefferman2021optimal}
Charles Fefferman, Bernat~Guill{\'e}n Pegueroles, Clarence~W Rowley, and Melanie Weber.
\newblock Optimal control with learning on the fly: a toy problem.
\newblock {\em Revista matem{\'a}tica iberoamericana}, 38(1):175--187, 2021.

\bibitem{feller}
Willliam Feller.
\newblock {\em An Introduction to Probability Theory and Its Applications, Volume 2}.
\newblock John Wiley \& Sons, Inc., 1971.

\bibitem{furieri2020learning}
Luca Furieri, Yang Zheng, and Maryam Kamgarpour.
\newblock Learning the globally optimal distributed {LQ} regulator.
\newblock In {\em Learning for Dynamics and Control}, pages 287--297. PMLR, 2020.

\bibitem{goel2021competitive}
Gautam Goel and Babak Hassibi.
\newblock Competitive control.
\newblock {\em arXiv preprint arXiv:2107.13657}, 2021.

\bibitem{gurevich2022optimal}
Daniel Gurevich, Debdipta Goswami, Charles~L Fefferman, and Clarence~W Rowley.
\newblock Optimal control with learning on the fly: System with unknown drift.
\newblock In {\em Learning for Dynamics and Control Conference}, pages 870--880. PMLR, 2022.

\bibitem{hazancontrol}
Elad Hazan and Karan Singh.
\newblock Online nonstochastic control.
\newblock {\em arXiv preprint arXiv:2211.09619}, 2022.

\bibitem{jedra2022minimal}
Yassir Jedra and Alexandre Proutiere.
\newblock Minimal expected regret in linear quadratic control.
\newblock In {\em International Conference on Artificial Intelligence and Statistics}, pages 10234--10321. PMLR, 2022.

\bibitem{kargin2022thompson}
Taylan Kargin, Sahin Lale, Kamyar Azizzadenesheli, Anima Anandkumar, and Babak Hassibi.
\newblock Thompson sampling achieves $\backslash$tilde o ($\backslash$sqrt $\{$T$\}$) regret in linear quadratic control.
\newblock {\em arXiv preprint arXiv:2206.08520}, 2022.

\bibitem{kumar2022online}
Raunak Kumar, Sarah Dean, and Robert~D Kleinberg.
\newblock Online convex optimization with unbounded memory.
\newblock {\em arXiv preprint arXiv:2210.09903}, 2022.

\bibitem{malik2019derivative}
Dhruv Malik, Ashwin Pananjady, Kush Bhatia, Koulik Khamaru, Peter Bartlett, and Martin Wainwright.
\newblock Derivative-free methods for policy optimization: Guarantees for linear quadratic systems.
\newblock In {\em The 22nd International Conference on Artificial Intelligence and Statistics}, pages 2916--2925. PMLR, 2019.

\bibitem{mania2019certainty}
Horia Mania, Stephen Tu, and Benjamin Recht.
\newblock Certainty equivalence is efficient for linear quadratic control.
\newblock {\em arXiv preprint arXiv:1902.07826}, 2019.

\bibitem{martin2022safe}
Andrea Martin, Luca Furieri, Florian D{\"o}rfler, John Lygeros, and Giancarlo Ferrari-Trecate.
\newblock Safe control with minimal regret.
\newblock In {\em Learning for Dynamics and Control Conference}, pages 726--738. PMLR, 2022.

\bibitem{minasyan2021online}
Edgar Minasyan, Paula Gradu, Max Simchowitz, and Elad Hazan.
\newblock Online control of unknown time-varying dynamical systems.
\newblock {\em Advances in Neural Information Processing Systems}, 34:15934--15945, 2021.

\bibitem{powell2007approximate}
Warren~B Powell.
\newblock {\em Approximate Dynamic Programming: Solving the curses of dimensionality}, volume 703.
\newblock John Wiley \& Sons, 2007.

\bibitem{robbins1952some}
Herbert Robbins.
\newblock Some aspects of the sequential design of experiments.
\newblock {\em Bulletin of the American Mathematical Society}, 58(5):527--535, 1952.

\bibitem{simchowitz2018learning}
Max Simchowitz, Horia Mania, Stephen Tu, Michael~I Jordan, and Benjamin Recht.
\newblock Learning without mixing: Towards a sharp analysis of linear system identification.
\newblock In {\em Conference On Learning Theory}, pages 439--473. PMLR, 2018.

\bibitem{vermorel2005multi}
Joannes Vermorel and Mehryar Mohri.
\newblock Multi-armed bandit algorithms and empirical evaluation.
\newblock In {\em European Conference on Machine Learning}, pages 437--448. Springer, 2005.

\bibitem{wagenmaker2020active}
Andrew Wagenmaker and Kevin Jamieson.
\newblock Active learning for identification of linear dynamical systems.
\newblock In {\em Conference on Learning Theory}, pages 3487--3582. PMLR, 2020.

\bibitem{wei2021non}
Chen-Yu Wei and Haipeng Luo.
\newblock Non-stationary reinforcement learning without prior knowledge: An optimal black-box approach.
\newblock In {\em Conference on Learning Theory}, pages 4300--4354. PMLR, 2021.

\end{thebibliography}
\end{document}